\documentclass{article}
\usepackage{graphicx}
\graphicspath{{../}}
\usepackage{amssymb,amsmath}
\usepackage{cancel}
\usepackage{epstopdf}
\usepackage{pdflscape}
\usepackage{empheq}
\usepackage{verbatim}

\usepackage{xurl}
\usepackage[
colorlinks=true,
linkcolor=blue,
citecolor=darkgreen,
urlcolor=darkblue,
]{hyperref}

\usepackage{enumitem}
\usepackage{calc}
\usepackage{mathtools}
\usepackage{color}
\definecolor{darkgreen}{rgb}{0,0.5,0}
\definecolor{purple}{rgb}{0.5,0,0.5}
\definecolor{darkblue}{rgb}{0,0,0.6}

\DeclareMathOperator\sgn{sgn}

\numberwithin{equation}{section}

\newtheorem{theorem}{Theorem}[section]
\newtheorem{lemma}[theorem]{Lemma}

\newtheorem{prop}[theorem]{Proposition}
\newtheorem{corollary}[theorem]{Corollary}

\newtheorem{defi}[theorem]{Definition}
\newtheorem{rem}[theorem]{Remark}

\newcommand{\RR}{\mathbb{R}}
\newcommand{\C}{\mathbb{C}}
\newcommand{\Z}{\mathbb{Z}}
\newcommand{\N}{\mathbb{N}}
\newcommand{\T}{\mathbb{T}}

\newcommand{\supp}{{\rm supp}\thinspace}

\newcommand{\la}{\lambda}
\newcommand{\te}{\theta}
\newcommand{\eps}{\epsilon}
\newcommand{\ep}{\varepsilon}
\newcommand{\de}{\delta}
\newcommand{\foba}[1]{{{\bf e}^{#1}}}

\newcommand{\pa}{\partial}
\newcommand{\om}{\omega}

\newcommand{\intpi}{\int_{-\pi}^\pi}
\newcommand{\intpin}{\frac 1{2\pi}\intpi}
\renewcommand{\Re}{{\rm Re}\thinspace}

\newcommand{\ha}{\widehat}
\newcommand{\ti}{\widetilde}

\newcommand{\Pcal}{\mathcal{P}}

\newcommand{\LfibD}{L^{{\rm fib}\,D}}

\renewcommand{\exp}{\mathrm{exp}} % avoid the thin space TeX inserts between an \mathop and an opening parenthesis
\newcommand{\norm}[1]{\|#1\|}
\newcommand{\bignorm}[1]{\big\|#1\big\|}
\newcommand{\D}{\mathbb{D}}
\newcommand{\al}{\alpha}

\newcommand{\spn}{\operatorname{span}}
\newcommand{\Scal}{\mathcal{S}}
\newcommand{\Q}{\mathbb{Q}}

\newcommand{\OLdisk}[2]{{L^{\rm hol}_{#1,#2}}}
\newcommand{\Hard}[2]{{\mathsf H^{#1}_{#2}}}

\newcommand{\Fholh}[2]{
{F_{\la,\te,#1}^{h,#2}}
}
\newcommand{\Fholhte}[3]{
{F_{\la,\te^#3,#1}^{h,#2}}
}
\newcommand{\Fholp}[1]{
{F_{G,\la,\te}^{p,#1}}
}

\newcommand{\Pglo}{
{P_{\la,\te}}
}
\newcommand{\Qglo}{
{Q_{\la,\te}}
}

\newcommand{\Qglolm}{
{Q_{-\la,\te}}
}
\newcommand{\Pglolm}{
{P_{-\la,\te}}
}

\newcommand{\Ploc}[1]{
{P_{\la,\te}^{#1}}
}
\newcommand{\Qloc}[1]{
{Q_{\la,\te}^{#1}}
}
\newcommand{\ploc}[1]{
{p_{\la,\te}^{#1}}
}
\newcommand{\qloc}[1]{
{q_{\la,\te}^{#1}}
}

\newcommand{\Affterm}[2]{
{ G_{#1}^{#2}}
}
\newcommand{\FholpAff}[3]{
{F_{\la,#2,#1}^{p,#3}}
}
\newcommand{\oldDFp}[2]{%
    \FholpAff{#1}{#2}{1}-\FholpAff{#1}{#2}{0}%
}

\newcommand{\Mcritnov}{
{{\bf M}}
}
\newcommand{\Mcrit}{
{\Mcritnov_{\la,\te}}
}
\newcommand{\Mcritap}{
{\Mcritnov_{\la,\te}^{\rm ap}}
}
\newcommand{\Fholhol}[2]{
{\mathcal{F}_{\la,\te^{#2}}^{{#2},{#1}}}
}
\newcommand{\detcrit}{
\nu
}

\newcommand{\XH}[1]{{H^{#1}(\T)}}
\newcommand{\XHZmean}[1]{{\dot{H}^{#1 }(\T)}}
\newcommand{\XHhol}[1]{{\mathcal{H}^{#1}}}
\newcommand{\XHholpair}[1]{\XHhol{#1} \oplus \XHhol{#1}}
\newcommand{\XHholo}{{\mathcal{H}(\T)}}
\newcommand{\XL}{{L^2(\T)}}

\newcommand{\Itheta}{\left[\frac13,\frac12\right]}
\newcommand{\Ilambd}{\left(\getvar{low_bnd_eig},\getvar{upp_bnd_eig}\right)}
\newcommand{\cIlambd}{\left[\getvar{low_bnd_eig},\getvar{upp_bnd_eig}\right]}
\newcommand{\Slambd}{\{\getvar{low_bnd_eig},\getvar{upp_bnd_eig}\}}

\newcommand{\andspace}{\qquad\text{and}\qquad}

\newif\ifhideproofs
\ifhideproofs

\else 
    \newenvironment{proof}{\begin{trivlist} \item[] {\bf Proof:}}{\hfill $\Box$
                       \end{trivlist}}
\fi 

\newenvironment{proofsth}[1]{\begin{trivlist} \item[] {\textbf{Proof of #1:}}}{\hfill $\Box$
                       \end{trivlist}} % Used

\newcommand{\proofcompassi}{
\begin{trivlist}
\item[] {\bf Proof:}
 The proof is computer-assisted and the code can be found in the supplementary material. We refer to the Appendix~\ref{sec:appendix:impl} for details about the implementation.
\hfill $\Box$ \end{trivlist}}

\newcommand{\prooffuchstoappendixH}{
\begin{proof}
    The statements follow from standard Frobenius theory for regular singular points after computing the corresponding indicial equations.
    
    For completeness, and since explicit recursive formulas are useful for constructive computations, detailed derivations are provided in Appendix~\ref{sec:appendix:frobenius}.
\end{proof}
}

\newcommand{\prooffuchstoappendixW}[1]{
\begin{proof}
    The solution is constructed directly from the Frobenius recurrence for the inhomogeneous equation, using the local normal form of {#1}. See Appendix~\ref{sec:appendix:frobenius}.
\end{proof}
}

\newcommand{\ga}{\gamma}
\newcommand{\ka}{\kappa}
\newcommand{\linsol}{\eta^{\rm lin}_\de}
\newcommand{\nonsol}{\eta_\de}
\newcommand{\difsol}{\varpi_\de}
\newcommand{\resol}{\eta_{\de,n}}
\newcommand{\rthr}{\varrho}

\newcommand{\dx}{\,dx}
\newcommand{\dy}{\,dy}
\newcommand{\ds}{\,ds}

\newcommand{\wiajustao}[2]{%
  \mathord{%
    \mathrlap{%
      \ha{%
        \mathrlap{#1_{#2}}%
        \phantom{#1}%
      }%
    }%
    \phantom{#1_{#2}}%
  }%
}

\newcommand{\setvar}[2]{%
  \expandafter\def\csname var@#1\endcsname{#2}%
}

\newcommand{\getvar}[1]{%
  \csname var@#1\endcsname
}

\usepackage[
backend=biber,
style=alphabetic,
url=false,
eprint=false,
maxbibnames=199,
giveninits=true,
block=space,
]{biblatex}
\title{On the instability of De Gregorio excited states}
\author{Miguel M. G. Pascual-Caballo}  

\begin{document}

\maketitle

\begin{abstract}
    We prove that all excited states $-\sin(nx)$, $n\geq2$, of the De Gregorio equation on the torus are nonlinearly unstable in the sense of Hadamard. More precisely, perturbations arbitrarily small in $H^2$ leave a fixed $H^1$ neighborhood of the state. For every $n\geq2$, the linearized operator has a real eigenvalue $\lambda_n\in(0.56,0.6)$. To prove this, we transform the eigenvalue problem into a Fuchsian equation on the complex disk and obtain a matching condition for the eigenvalue. This condition is verified rigorously using interval arithmetic. The nonlinear instability follows from a bootstrap argument in the spirit of Guo, Hallstrom, and Spirn.
\end{abstract}

{\hypersetup{linkcolor=black}\tableofcontents}

\setvar{low_bnd_eig}{0.56}
\setvar{upp_bnd_eig}{0.6}

\setvar{prime_split_N}{300}
\setvar{proof_n0}{31}
\setvar{proof_nres}{62}
\setvar{proof_bits}{200}
\setvar{proof_center}{\frac{1}{2}}
\setvar{proof_rho}{\frac{3}{4}}
\setvar{proof_bzero}{\frac{1}{2}}

\setvar{linkdgreg}{
\href{https://github.com/MiguelMGPascualCaballo/dgreg/tree/v2.0}{\nolinkurl{github.com/MiguelMGPascualCaballo/dgreg/tree/v2.0}}
}

\section{Introduction}

The De Gregorio equation is the one-dimensional active scalar model
\begin{equation}\label{DeGreg}
    \left\{\begin{array}{l}
    \partial_t \om + u \pa_x \om = \om \pa_x u,\quad (x,t)\in\RR\times [0,\infty).\\
    u=-\Lambda^{-1} \om,\\
     \om(x,0)= \om_0(x).
    \end{array}\right.
\end{equation}
Here $\Lambda:=(-\Delta)^{1/2}$.

Equations of this type were introduced as one-dimensional models for the vorticity formulation of the three-dimensional incompressible Euler equations, where the dynamics is governed by the competition between transport and vortex stretching. The model of Constantin, Lax and Majda~\cite{CLM85} keeps only the stretching term $\om\pa_x u$ and develops singularities in finite time from smooth data. De Gregorio~\cite{DeGregorio1990,DeGregorio1996} added back the transport term $u\pa_x\om$, whose regularizing effect can compete with the stretching and prevent blow-up. The two terms are interpolated by the one-parameter family of Okamoto, Sakajo and Wunsch~\cite{OSW08}, of which~\eqref{DeGreg} is the endpoint where transport is fully present.

The behavior of~\eqref{DeGreg} depends strongly on both the regularity of the data and the underlying domain.
On the real line, Chen, Hou and Huang~\cite{CHH21} proved finite-time blow-up from smooth odd initial data.
Singularity formation for lower-regularity data was established in~\cite{EJ20}. The periodic problem on $\T$ has a rather different character. 
Local well-posedness for $\om_0\in\XH{1}$ is proved in~\cite{OSW08}.
It extends to $s>\frac12$ by the methods of~\cite{Danchin,Chemin-book}.
Jia, Stewart and \v Sver\'ak~\cite{JSS19} outline this extension and prove that, for initial data $C^2$-close to an equilibrium $A\sin(x-x_0)$, the solution exists globally and converges exponentially to a nearby equilibrium in $\XH{s}$ for any $s<\frac32$.
Lei, Liu and Ren~\cite{LLR20} proved global well-posedness for initial data with nonnegative vorticity and stability near the ground state $-\sin x$.
Chen~\cite{Chen21} proved finite-time blow-up for the related gCLM equation on the circle when the transport coefficient is slightly below one, starting from smooth data with the same oddness and sign as the $\pi$-periodic class studied in~\cite{chen2023regularity}. For the De Gregorio equation, Chen~\cite{chen2023regularity} established a one-point blow-up criterion and global regularity for $C^1$ data in this class under an additional integral condition, while constructing finite-time blow-up solutions from $C^\alpha\cap H^s$ data for any $0<\alpha<1$ and $s<\frac32$. The latter construction yields blow-up data arbitrarily close to $-\sin x$ in $H^1$ and, by spatial rescaling, to $-\sin(2x)$ within the $\pi$-periodic subspace of $\T$. Thus these equilibria admit blow-up data in arbitrarily small $H^1$ neighborhoods even though the corresponding sufficiently regular data in Chen's symmetry class are global. Our instability result concerns $H^2$ perturbations in the full $2\pi$-periodic space.
The periodic stationary states
\begin{align*}
    x\mapsto -\sin(nx),\qquad n \geq 1
\end{align*}
form a natural hierarchy. The case $n=1$ is the ground state, while the higher-frequency states $n \geq 2$ may be viewed as excited states. A natural question is then whether stability is a special feature of the lowest mode, or whether it persists along the whole family.

We prove that every excited state $-\sin(nx)$, $n \geq 2$, is nonlinearly unstable. Perturbations arbitrarily small in $\XH2$ leave a fixed $\XH1$ neighborhood of the stationary state. The proof starts by locating a real unstable eigenvalue in $(0.56,0.6)$ with an eigenfunction in $\XH2$.

Instability of the first excited state $-\sin(2x)$ was recently obtained by Guo and Jiu~\cite{guo2025stabilityinstabilitygregoriomodification} through an energy method. Our spectral approach applies to every excited state and gives the following linear instability result.

\begin{theorem}
    \label{thm:eig}
    For every $n \geq 2$, the stationary solution
    \begin{align*}
        x\mapsto -\sin(n x)
    \end{align*}
    of the De Gregorio equation is linearly unstable in $\XH2$.
\end{theorem}

We then prove nonlinear instability in the sense of Hadamard.

\begin{theorem}
\label{thm:alln:nonlinear}
    For every integer $n\ge2$ there exists $\ep_n>0$ and $\eta_{0,n}\in\XH2$, such that for every $\de\in(0,\ep_n)$, the solution $\om$ to~\eqref{DeGreg} with $\om_0=-\sin(nx)+\de\eta_{0,n}$ satisfies, writing $\resol:=\om+\sin(nx)$, $\norm{\resol(T_{\de,n})}_\XH1\geq\ep_n$ at a finite escape time $T_{\de,n}<\infty$.
\end{theorem}

\paragraph{Proof strategy.}

We decompose the linearized equation into Fourier modes in the same residue class modulo $n$. Each resulting eigenvalue problem reduces to an ODE on the complex disk with regular singular points at $0$, $1$ and $-1$. Matching the local Frobenius solutions gives a $4\times4$ matrix whose determinant vanishes at the eigenvalues.
This Fuchsian approach was already used in~\cite{JSS19} for the linearization around the ground state. We also refer to our previous work~\cite{BHTW}, where a matching condition for a Fuchsian system is used in another context.

It suffices to consider prime values $n=p$. We select the residue class $h_p=\big\lfloor\frac{p}{2}\big\rfloor$, so that the parameter $\te_p=\frac{h_p}{p}$ equals $\frac12$ for $p=2$ and approaches $\frac12$ as $p\to\infty$. Interval arithmetic verifies the matching condition at finitely many parameter values and on a short interval ending at $\frac12$. These checks give an unstable eigenvalue $\la_p\in(0.56,0.6)$ for every prime $p$.

For nonlinear instability, we compare the solution with the linear evolution of a real unstable eigenfunction, following Guo, Hallstrom and Spirn~\cite{GuoHallstromSpirn2004,GuoHallstromSpirn}. In $\XH1$, the linear semigroup satisfies an estimate of the form
\begin{align*}
    \norm{e^{tL_p}\eta}_\XH1\leq C_{1,p}e^{1.1t}\norm\eta_\XH1.
\end{align*}
An energy estimate in $\XH2$ has linear rate $\frac12$, a lower-order term controlled by the $\XH1$ norm, and a nonlinear term bounded by the product of the two norms. Since $\frac12<\la_p$ and $1.1<2\la_p$, a bootstrap controls the $\XH2$ norm and keeps the nonlinear correction smaller than the growing linear solution until its $\XH1$ norm reaches a fixed size. Spatial rescaling transfers the instability to every $n$ divisible by $p$.

\subsection{Computer-assisted methods}

The entries of the matching matrix are enclosed by interval arithmetic, with rigorous bounds on the tails of the Frobenius series. The semigroup estimate in $\XH1$ also uses computer-assisted bounds on finite matrices, together with analytic estimates for the remaining Fourier modes. 
The code is available on GitHub at \getvar{linkdgreg}.

There are by now many works in which computer-assisted arguments play a substantial role. We mention, for instance, Guo, Had\v zi\'c, Jang, and Schrecker~\cite{Guo-Hadzic-Jang-Schrecker:gravitational-collapse-stars-self-similar} for Euler--Poisson, Kobayashi~\cite{Kobayashi:global-uniqueness-stokes} for extreme Stokes waves, Chen, Hou, and Huang~\cite{Chen-Hou-Huang:blowup-degregorio} for the De Gregorio model, Castro, C\'ordoba, and G\'omez-Serrano~\cite{Castro-Cordoba-GomezSerrano:global-smooth-solutions-sqg} for SQG, Elgindi and Pasqualotto~\cite{ElgindiPasqualotto2025_invertibility_boussinesq} and Chen and Hou~\cite{Chen-Hou:nearly-self-similar-blowup-boussinesq-euler-analysis,Chen-Hou:nearly-self-similar-blowup-boussinesq-euler-numerics,Chen-Hou:singularity-formation-3d-euler-smooth-data} for Boussinesq-type models and 3D Euler, Cadiot~\cite{Cadiot:proofs-existence-stability-capillary-gravity-whitham} for Whitham-type equations, several works on the Kuramoto--Sivashinsky equation~\cite{Arioli-Koch:cap-stationary-ks,Figueras-DeLaLLave:cap-periodic-orbits-kuramoto,Gameiro-Lessard:periodic-orbits-ks,Figueras-Gameiro-Lessard-DeLaLLave:framework-cap-invariant-objects,Zgliczynski:periodic-orbit-kuramoto,Zgliczynski-Mischaikow:rigorous-numerics-kuramoto}, Buckmaster, Cao-Labora, and G\'omez-Serrano~\cite{Buckmaster-CaoLabora-GomezSerrano:implosion-compressible} for compressible Euler and Navier--Stokes, and Arioli, Gazzola, and Koch together with Bedrossian and Punshon-Smith~\cite{Arioli-Gazzola-Koch:uniqueness-bifurcation-ns,Bedrossian-PunshonSmith:chaos-stochastic-2d-galerkin-ns} for incompressible Navier--Stokes.

For general background on validated numerics we refer to Tucker's book~\cite{Tucker:validated-numerics-book}. For computer-assisted proofs in PDE, see also the survey~\cite{GomezSerrano:survey-cap-in-pde} and the book~\cite{Nakao-Plum-Watanabe:cap-for-pde-book}. 

\subsection{Functional setting and notation}

We consider complex-valued functions. When needed, we will explicitly indicate that a function is real-valued.

In this paper we use the following convention for the $L^2$ norm
\begin{align*}
    \norm{f}_\XL^2:=
    \intpin |f(x)|^2\dx
    =\sum_{k\in\Z}|\ha{f}_k|^2,
\end{align*}
where 
\begin{align*}
    \ha f_k:=\intpin f(x)e^{-ikx} \dx.
\end{align*}

For $s\geq0$, $\XH{s}$ denotes the completion of $C^\infty(\T)$ with respect to the norm
\begin{align*}
    \norm{f}_\XH{s}^2 :=
    \intpin |f(x)|^2+|\Lambda^sf(x)|^2\dx
    =\sum_{k\in\Z}(1+|k|^{2s})|\ha f_k|^2.
\end{align*}
Here $\Lambda^0=I$, so the weight at $k=0$ is also defined when $s=0$.
For $s>0$, whenever $\XH{-s}$ appears, it denotes the dual of $\XH{s}$ with respect to the $L^2$ pairing. In particular, every $L^2(\T)$ function belongs to $\XH{-s}$.
$\XHZmean{s}$ denotes the subspace of zero-mean functions in $\XH{s}$, endowed with the homogeneous norm
\begin{align*}
    \norm{f}_\XHZmean{s}^2 :=
    \sum_{k\neq0} |k|^{2s} |\ha f_k|^2.
\end{align*}

The operator $\pa_x^{-1}:\XH{s}\to\XH{s+1}$, for $s\geq0$, is the Fourier multiplier
\begin{align*}
    \widehat{(\pa_x^{-1}g)}_k=
    \begin{cases}
        (ik)^{-1}\ha g_k,&k\ne0,\\
        0,&k=0.
    \end{cases}
\end{align*}
Thus $\pa_x\pa_x^{-1}g=g-\ha g_0$. For smooth $g$, the equivalent integral formula is
\begin{align*}
    \pa_x^{-1}g(x)=
    \int_0^x (g(y)-\ha g_0)\dy
    -\intpin\left(\int_0^x (g(y)-\ha g_0)\dy\right)\dx.
\end{align*}
We use the same zero-mode convention for $\Lambda^{-1}$: its symbol is $|k|^{-1}$ for $k\ne0$ and zero for $k=0$.

The Hilbert transform $H$ is the Fourier multiplier
\begin{align*}
    \ha{(Hf)}_k = -i \sgn(k) \ha f_k.
\end{align*}
With these conventions $\Lambda = H\pa_x$, and the velocity field in~\eqref{DeGreg} satisfies $\pa_x u = H\om$.

$\XHholo$ denotes the space of holomorphic functions on the unit disk $\D$. If $F\in \XHholo$, then $F$ admits the expansion
\begin{align*}
    F(z)=\sum_{k \geq 0}F_k z^k
\end{align*}
with $F_k\in \C.$

For such $F$, write formally $\left.F\right|_\T(x):=\sum_{k\geq0}F_k e^{ikx}$.
We also introduce the Hardy spaces 
\begin{align*}
    \XHhol{s}
    =\{ F\in \XHholo:\, \left. F\right|_\T\in \XH{s}\},
\end{align*}
endowed with the norm
\begin{align*}
    \norm{F}_\XHhol{s} :=
    \norm{F|_\T}_\XH{s}.
\end{align*}
We will use the following notation  
\begin{align*}
    f^+ := \frac12(f + i Hf) + \frac12 \ha f_0
    ,\qquad
    f^- := \frac12(f^\sharp + i Hf^\sharp) + \frac12 \ha f_{-1},
\end{align*}
where $f^\sharp(x):=e^{-ix}f(-x)$. 

It will be convenient to write 
\begin{align*}
    \Pcal^\pm f:=f^\pm
\end{align*}
since sometimes we will need to apply $\Pcal^\pm$ to equations rather than functions.

Notice that for $f \in \XH s$, with $s>\frac{1}{2}$, the Fourier series of $f^+$ and $f^-$ are
\begin{align*}
    f^+(x) = \sum_{k \geq 0} \ha f_k \foba k
    ,\qquad
    f^-(x) = \sum_{k \geq 0} \ha f_{-k-1} \foba k.
\end{align*}

Replacing $\foba 1$ by $z$, we obtain holomorphic extensions $F^+,F^-\in \XHhol s$ of $f^+$ and $f^-$ to the unit disk
\begin{align*}
    F^+(z) := \sum_{k \geq 0} \ha f_k z^k
    ,\qquad
    F^-(z) := \sum_{k \geq 0} \ha f_{-k-1} z^k.
\end{align*}
In other words, $F^+$ encodes the nonnegative Fourier modes of $f$, while $F^-$ encodes the negative ones after reindexing.

For $n \geq 2$ and $l=1,\dots,n-1$, we also use the unfolding maps
\begin{align*}
    \ha{(\Pcal_{l,n}f)}_k=\ha{f}_{nk+l}.
\end{align*}
Thus $\Pcal_{l,n}$ selects the Fourier modes congruent to $l$ modulo $n$ and reindexes them by $nk+l\mapsto k$.

We denote by $S_n$ and $C_n$ the trigonometric functions given by
\begin{align*}
    S_n(x) = \sin(n x),\qquad
    C_n(x) = \cos(n x).
\end{align*}

Finally, we write $D_r(z_0)\subset \C$ for the open complex disk of radius $r>0$ centered at $z_0$, i.e.
\begin{align*}
    D_r(z_0):=
    \{z\in\C:|z-z_0|<r\}.
\end{align*}

\section{Linearized instability}
This section proves Theorem~\ref{the:mainlin}. We begin by linearizing~\eqref{DeGreg} around $-S_n$.

Let $\om$ and $u$ be functions solving~\eqref{DeGreg} and write
\begin{align}\label{def:pert}
    \om = -S_n + \eta.
\end{align}
We can then write $u$ as
\begin{align}
    u = \frac1n S_n + H \pa_x^{-1} \eta.
\end{align}
Linearizing~\eqref{DeGreg} around the steady state $-S_n$, the equation for $\eta$ reads
\begin{align}\label{eqn:linf}
    \pa_t \eta = L_n \eta,
\end{align}
where
\begin{align}
    \label{def:Ln}
    L_n \eta :=
        - \frac1n S_n \pa_x \eta
        + n C_n H \pa_x^{-1}\eta
        - S_n H\eta
        + C_n \eta.
\end{align}

\begin{rem}
\label{rem:linf:wellposed}
    For $s\geq0$, let $D_s$ be the Fourier multiplier with symbol $\langle k\rangle^s$, where $\langle k\rangle:=(1+k^2)^{1/2}$ denotes the Japanese bracket. Write
    \begin{align*}
        D_sL_nD_s^{-1} = \underbrace{-\tfrac1nS_n\pa_x-\tfrac12C_n}_{\mathfrak{A}_n} + \underbrace{\big(D_sL_nD_s^{-1}+\tfrac1nS_n\pa_x+\tfrac12C_n\big)}_{K_{n,s}}.
    \end{align*}
    Here $\mathfrak{A}_n$ is a skew-adjoint operator and $K_{n,s}$ is bounded on $\XL$, see Remark~\ref{rem:Kns:bounded} below. By Stone's theorem~\cite{Stone1932}, $\mathfrak{A}_n$ generates a unitary $C_0$-group on $\XL$, and by the bounded perturbation theorem (\cite{EngellNagel00}, Chapter III, Theorem 1.3), $L_n$ generates a $C_0$-semigroup on every $\XH s$. In particular, for $\eta_0\in\XH{s+1}$ there is a unique $\eta\in C^0([0,\infty);\XH{s+1})\cap C^1([0,\infty);\XH s)$ solving~\eqref{eqn:linf} with $\eta(0)=\eta_0$.
\end{rem}

\begin{theorem}
\label{the:mainlin}
    For every $n \geq 2$, there exist $\la_n\in\Ilambd$ and $0\neq \eta_n\in \XH s$ for every $s < \frac 3 2+\la_n$ such that
    \begin{align*}
        e^{\la_n t}\eta_n(x)
    \end{align*}
    solves~\eqref{eqn:linf}.
\end{theorem}

\begin{rem}
    As we will see (Remark~\ref{rem:ess_spec_frobenius}), the spectrum of $L_n$ on $\XH s$ contains the band
   \begin{align*}
        \big\{z\in\C:|\Re(z)|<\tfrac32-s\big\}\qquad 
        \text{for all }
        s\in\big(0,\tfrac32\big).
    \end{align*}
    The spectral band prevents isolation when $s<\frac32-\la_n$. For the selected fiber, Corollary~\ref{coro:spec} and the analyticity of the matching determinant show that $\la_n$ is isolated on $\XH s$ for every $s\in[1,2]$.
\end{rem}

The rest of this section is devoted to the proof of Theorem~\ref{the:mainlin}. Notice that Theorem~\ref{thm:eig} follows from Theorem~\ref{the:mainlin} since $\la_n>\getvar{low_bnd_eig}>\frac12$, giving $2<\frac32+\la_n$ and hence $\eta_n\in\XH2$.

\subsection{Decoupling the modes}
\label{sec:decoupmodes}

We begin by reformulating the problem in a way that works uniformly for all $n$.
Denote by $\foba k:\RR\to\C$ the function given by
\begin{align*}
    \foba k(x) := e^{ikx}.
\end{align*}

The key point is that $L_n$ only couples neighboring modes inside each congruence class. In Fourier variables this takes the form
\begin{align}
\label{id:Ln:fourier}
    L_n \foba{k}=
        A_{k,n} \foba{k-n}
        +
        B_{k,n} \foba{k+n},
\end{align}
where
\begin{align}
\label{def:AkBkn}
    A_{k,n}=
        \frac{(k+n)(k-n\sgn(k))}{2nk}
    ,&&
    B_{k,n}=
        -\frac{(k-n)(k-n\sgn(k))}{2nk},
\end{align}
for $k\neq 0$.

Guided by this formula, we introduce the formal operator $L_\te^{\rm fib}$, with $\te\in(0,1)$, defined by
\begin{align}
\label{def:Lfib}
    L_\te^{\rm fib} \foba{k}
    :=
        \ti{A}_{k,\te} \foba{k-1}
        +
        \ti{B}_{k,\te} \foba{k+1},
\end{align}
where
\begin{align}
\label{def:ABkt}
    \ti{A}_{k,\te}=
        \frac{(k+\te+1)(k+\te-\sgn(k+\te))}{2(k+\te)}
    ,&&
    \ti{B}_{k,\te}=
        -\frac{(k+\te-1)(k+\te-\sgn(k+\te))}{2(k+\te)}.
\end{align}

\begin{rem}
\label{rem:AkBksym}
    Let $\te\in(0,1)$. For every $k\in\Z$,
    \begin{align*}
        \ti{A}_{-k-1,\te}=
            \ti{B}_{k,1-\te}.
    \end{align*}
    This relation will be useful when proving Lemma~\ref{lem:red_to_fin}.
\end{rem}

\begin{rem}
\label{rem:Kns:bounded}
    Recall Remark~\ref{rem:linf:wellposed}. For $s\geq0$, $K_{n,s}$ is bounded on $\XL$.
    Indeed, notice that $\mathfrak{A}_n\foba k=a_{k,n}\foba{k-n}+b_{k,n}\foba{k+n}$ with $a_{k,n}:=\frac k{2n}-\frac14$ and $b_{k,n}:=-\frac k{2n}-\frac14$.
    Define $\psi(k,n):= \big(\tfrac{\langle k-n\rangle}{\langle k\rangle}\big)^s - 1$ for all $k\in\Z$. By the mean value theorem, there is $\xi$ between $\tfrac{\langle k-n\rangle}{\langle k\rangle}$ and $1$ such that
    \begin{align*}
        \psi(k,n)=s\xi^{s-1}\frac{\langle k-n\rangle-\langle k\rangle}{\langle k\rangle}.
    \end{align*}
    Since $\langle\cdot\rangle$ is $1$-Lipschitz, $|\langle k-n\rangle-\langle k\rangle|\leq n$, so for $|k|\geq2n$ we get $\tfrac{\langle k-n\rangle}{\langle k\rangle}\in[\frac12,\frac32]$, hence $\xi\in[\frac12,\frac32]$, so $|\xi^{s-1}|\leq C_s:=\max\big\{(\tfrac12)^{s-1},(\tfrac32)^{s-1}\big\}$ and
    \begin{align*}
        |\psi(k,n)|\leq sC_sn\,\tfrac1{\langle k\rangle}\leq sC_sn\,\tfrac1{|k|},
    \end{align*}
    that is, $\psi(k,n)=O(\tfrac{1}{k})$ for a fixed $n$.

    Hence, by~\eqref{id:Ln:fourier} and~\eqref{def:AkBkn},
    \begin{align}
        (D_sL_nD_s^{-1}-\mathfrak{A}_n)\foba k
        &=
        \Big(A_{k,n}\Big(\tfrac{\langle k-n\rangle}{\langle k\rangle}\Big)^s-a_{k,n}\Big)\foba{k-n}
        +
        \Big(B_{k,n}\Big(\tfrac{\langle k+n\rangle}{\langle k\rangle}\Big)^s-b_{k,n}\Big)\foba{k+n}
        \nonumber
        \\
        &=
        \Big((A_{k,n}-a_{k,n})+A_{k,n} \psi(k,n)\Big)\foba{k-n}
        +
        \Big((B_{k,n}-b_{k,n})+B_{k,n} \psi(k,-n)\Big)\foba{k+n}
        \nonumber
        \\
        &=
        \Big(\frac34-\frac{\sgn(k)}2-\frac n{2|k|} + A_{k,n} \psi(k,n)\Big)\foba{k-n}
        \nonumber\\
        &\qquad+
        \Big(\frac34+\frac{\sgn(k)}2-\frac n{2|k|} + B_{k,n} \psi(k,-n)\Big)\foba{k+n}.
        \label{lin_semig_last_id}
    \end{align}
    Since $A_{k,n},B_{k,n}$ are $O(k)$ and $\psi(k,n),\psi(k,-n)$ are $O(k^{-1})$, every term in~\eqref{lin_semig_last_id} is bounded.
\end{rem}

The following lemma describes the precise relationship between $L_n$ and $L_\te^{\rm fib}$.
\begin{lemma}
\label{lem:unfold}
    Let $n\in\N$ with $n \geq 2$, and let $\eta:\T\to\C$ be an $\XH1$ function satisfying
    \begin{align*}
        \ha \eta_{nk} = 0\qquad \text{for all } k\in\Z.
    \end{align*}
    Then $\eta$ satisfies~\eqref{eqn:linf} if and only if, for every $1 \leq l \leq n-1$, the function $\eta^\te:=\Pcal_{l,n} \eta$ satisfies
    \begin{align}
    \label{eqn:unfold__}
        \pa_t \eta^\te = L_\te^{\rm fib} \eta^\te,\quad \text{where }
        \te:=\frac ln.
    \end{align}
\end{lemma}
\begin{proof}
    Projecting~\eqref{eqn:linf} onto $e^{ikx}$ for integers $k$ that are not divisible by $n$ gives
    \begin{align*}
        \pa_t \ha \eta_k 
        = 
        (\ha{L_n \eta})_k
        =
            A_{k+n,n}
            \ha \eta_{k+n}
            +
            B_{k-n,n}
            \ha \eta_{k-n}.
    \end{align*}
    Recall that the Fourier modes of $\Pcal_{l,n} \eta$ are given by
    \begin{align*}
        \ha{\Pcal_{l,n} \eta}_k = \ha \eta_{nk+l}.
    \end{align*}
    Setting $\te:=\frac ln$ and $\eta^\te:=\Pcal_{l,n}\eta$, we substitute $k\mapsto nk+l$ to obtain
    \begin{align}
        \pa_t (\ha{\eta^\te})_k
        =
        \pa_t
        \ha \eta_{nk+l}
        &=
            A_{n(k+1)+l,n}
            \ha \eta_{n(k+1)+l}
            +
            B_{n(k-1)+l,n}
            \ha \eta_{n(k-1)+l}
        \nonumber\\
        &=
            \ti{A}_{k+1,\te}
            (\ha{\eta^\te})_{k+1}
            +
            \ti{B}_{k-1,\te}
            (\ha{\eta^\te})_{k-1},
        \label{eqn:unfold:aux}
    \end{align}
    which holds for all $k\in\Z$. Finally, \eqref{eqn:unfold:aux} is precisely the projection of~\eqref{eqn:unfold__} onto $e^{ikx}$.
\end{proof}

This reduces the problem to the family of operators $L_\te^{\rm fib}$. We now show that it is enough to restrict attention to $\te\in\Itheta$.

\begin{lemma}\label{lem:1312enough}
    Let $n\in\N$ with $n \geq 2$. Then there exists $1 \leq h \leq n-1$ such that
    \begin{align*}
        \frac hn \in 
        \Itheta.
    \end{align*}
\end{lemma}
\begin{proof}
    For $n=2,3$ the choice $h=1$ works. Assume $n \geq 4$ and set
    \begin{align*}
        h=\left\lceil \frac n3\right\rceil.
    \end{align*}
    Then $h \geq \frac n3$ and $3h \leq n+2$, which gives
    \begin{align*}
        h \leq \frac n3+\frac23 \leq \frac n2,
    \end{align*}
    where the last inequality uses $n \geq 4$. Hence $\frac13 \leq \frac hn \leq \frac12$.
\end{proof}

By Lemma~\ref{lem:1312enough}, for every $n$ there exists a fraction $\te=\frac{h}{n}\in\Itheta$. We now choose this fraction so that the divisibility argument can be applied. The verification in Section~\ref{sec:eigte} is carried out at finitely many values of $\te$ and on one short interval.

At this point it is enough to consider prime values of $n$. Lemma~\ref{lem:unfold} shows that once instability is known for $L_p$, it follows for every $L_n$ with $p\mid n$. The next lemma tells us which values of $\te=\frac h p\in\Itheta$ need to be certified.

\begin{lemma}
\label{lem:cover}
    Let $N:=\getvar{prime_split_N}$ and set
    \begin{align*}
        J:=\Big[\frac12-\frac1{2N},\ \frac12\Big].
    \end{align*}
    For a prime $p$ put $h_p:=\big\lfloor\frac p2\big\rfloor$ and
    \begin{align*}
        \te_p:=\al_{0,p}:=\frac{h_p}p,
    \end{align*}
    the largest fiber of $L_p$ not exceeding $\frac12$. Then $h_p\in\{1,\dots,p-1\}$ and $\te_p\in\Itheta$ for every prime $p$, $J\subset\Itheta$, and $\te_p\in J$ for every prime $p \geq N$.
\end{lemma}

\begin{proof}
    For $p=2$ we have $h_2=1$ and $\te_2=\frac12$. For odd $p$ we have $h_p=\frac{p-1}2$ and
    \begin{align*}
        \te_p=\frac{p-1}{2p}=\frac12-\frac1{2p},
    \end{align*}
    which is at least $\frac13$ because $p \geq 3$, and smaller than $\frac12$. In both cases $1 \leq h_p \leq p-1$ and $\te_p\in\Itheta$. Since $N \geq 3$ we have $\frac1{2N} \leq \frac16$, so $J\subset\Itheta$. Finally, if $p \geq N$ then $\frac1{2p} \leq \frac1{2N}$, hence $\te_p \geq \frac12-\frac1{2N}$ and $\te_p\in J$.
\end{proof}

It remains to prove the following spectral statement.

\begin{prop}
\label{prop:eigte}
    Let $\te=\te_p$ for some prime $p<N$, or $\te\in J$, with $\te_p,J,N$ as in
    Lemma~\ref{lem:cover}. Then there exist $\la_\te\in\Ilambd$ and $0\neq \eta_\te\in \XH s$ for every $s < \frac32 + \la_\te$ satisfying
    \begin{align}
    \label{eqn:eigthe}
        \la_\te \eta_\te = L_\te^{\rm fib} \eta_\te.
    \end{align}
\end{prop}

\begin{proof}
    Section~\ref{sec:eigte} is devoted to proving Proposition~\ref{prop:eigte}.
\end{proof}

We refer to Figure~\ref{fig:greg_eigs} for a numerical illustration of the eigenvalue $\la_\te$.

\begin{figure}[htbp]
    \centering
    \includegraphics[width=8.5cm]{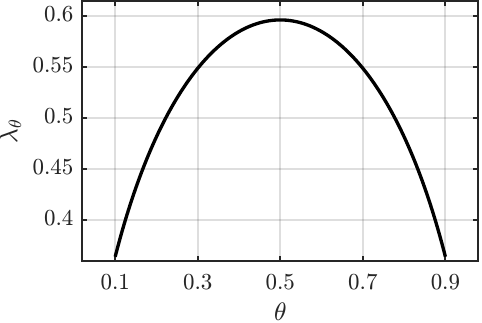}
    \caption{Numerical approximation of the eigenvalue $\la_\te$ of $L_\te^{\rm fib}$.}
    \label{fig:greg_eigs}
\end{figure}

\subsection{Conclusion of the proof of Theorem~\ref{the:mainlin}}
Fix $n \geq 2$, and choose a prime divisor $p$ of $n$. Write $n=pm$ with $m\in\N$.
By Lemma~\ref{lem:cover}, $h:=h_p\in\{1,\dots,p-1\}$ satisfies $\te:=\frac h p=\te_p\in\Itheta$, and $\te\in J$ when $p \geq N$. We may then apply Proposition~\ref{prop:eigte} and obtain $\la>\getvar{low_bnd_eig}$ and $0\neq\eta_0\in\XH s$ for every $s < \frac 3 2+\la$ satisfying~\eqref{eqn:eigthe}.

Set $l:=hm$. Since $1 \leq h \leq p-1$, we have $1 \leq l \leq n-1$, and also $\frac ln=\te$. Therefore
\begin{align*}
    \eta^\te(x,t) = e^{\la t} \eta_0(x)
\end{align*}
solves~\eqref{eqn:unfold__} with this choice of $l$.

We now define
\begin{align*}
    \ti{\eta}_0(x) := e^{ilx} \eta_0(n x).
\end{align*}
Since the $(nk+l)$-th Fourier coefficient of $\ti{\eta}_0$ is $\ha{\eta}_0(k)$ for all $k\in\Z$ and all other Fourier modes of $\ti{\eta}_0$ vanish, for $s\geq0$ we have $\norm{\ti{\eta}_0}_\XH{s}^2 = \sum_{k} (1+|nk+l|^{2s})|\ha{\eta}_0(k)|^2 \sim_{n,l,s} \norm{\eta_0}_\XH{s}^2$, so $\ti{\eta}_0\in\XH s$ for every $0\leq s < \frac 3 2+\la$.
For $s<0$, the claim follows from $\ti{\eta}_0\in\XH0\subset\XH{s}$.

Then the $nk$-th Fourier coefficient of $\ti{\eta}_0$ vanishes for every $k\in\Z$.
Moreover, for every $1 \leq l' \leq n-1$,
\begin{align*}
    \Pcal_{l',n} \ti{\eta}_0 =
    \de_{l,l'} \eta_0.
\end{align*}

By Lemma~\ref{lem:unfold}, the function $\ti{\eta}$ given by
\begin{align*}
    \ti{\eta}(x,t) = e^{\la t} \ti{\eta}_0(x)
\end{align*}
solves~\eqref{eqn:linf}. This proves Theorem~\ref{the:mainlin}, with $\la_n:=\la$ and $\eta_n:=\ti\eta_0$.

\section{Nonlinear instability}
\label{sec:nonlinear}

Let $\om$ in~\eqref{DeGreg} be given as in~\eqref{def:pert}, namely
\begin{align*}
    \om = -S_n + \eta.
\end{align*}
Then the function $\eta$ satisfies
\begin{align}
    \pa_t \eta = L_n \eta + Q(\eta,\eta),
    \label{eqn:pertsinenonline}
\end{align}
where $L_n$ is given in~\eqref{def:Ln} and $Q(\eta,\eta)$ by
\begin{align}
    Q(\eta,\eta) := 
    \eta H\eta
    -
    \pa_x^{-1}\!H\eta\,\pa_x\eta.
    \label{def:Quadop}
\end{align}
We work with real initial data of zero mean. The local existence and uniqueness result of Theorem~3.1 in~\cite{OSW08} applies to $\om=\eta-S_n$ in $\XH1$. Integration by parts and the skew-adjointness of $H$ give
\begin{align*}
    \pa_t\intpin\eta(x,t)\dx
    =\pa_t\intpin\om(x,t)\dx
    =2\intpin\om(x,t)H\om(x,t)\dx=0,
\end{align*}
so the mean of $\eta$ remains zero.

For $\eta_0\in\XHZmean2$, the solution above has a maximal interval of existence $[0,T)$ on which
\begin{align*}
    \eta\in C([0,T);\XHZmean2)\cap C^1([0,T);\XHZmean1).
\end{align*}
This follows from the Lagrangian formulation in Theorem~4.2 of~\cite{EKW10} with $k=3$ and $a=-1$.
Their vorticity has zero mean, while their velocity $U$ is normalized by $U(0,t)=0$.
Their conventions are converted to ours by reversing the sign of the vorticity and translating by the time integral of the mean velocity, that is,
\begin{align*}
    m(t):=\int_0^t\intpin U(x,s)\dx\,ds,\qquad
    u(x,t):=U(x+m(t),t)-m'(t),\qquad
    \om(x,t):=-H\pa_xU(x+m(t),t).
\end{align*}

We also need continuous dependence in $\XH2$.
With our convention for $H$, the flow $\varphi$ and the transported velocity $V=U\circ\varphi$ satisfy
\begin{align*}
    \pa_t\varphi&=V,\\
    \pa_tV&=\left(U\pa_xU+
    \int_0^x H\big(UH\pa_x^2U-\pa_xU\,H\pa_xU\big)(y,\cdot)\dy\right)\circ\varphi,
    \qquad U=V\circ\varphi^{-1},
\end{align*}
with $\varphi(x,0)=x$ and $V(x,0)=U(x,0)$.
Theorem~4.2 and Corollary~4.4 of~\cite{EKW10} show that this system has a smooth vector field at $H^3(\T)$ regularity, so its solutions exist locally and depend continuously on the initial data.
To recover $U=V\circ\varphi^{-1}$, we use continuity of $\varphi\mapsto\varphi^{-1}$ in the Sobolev diffeomorphism group, recalled at the beginning of Section~4 of~\cite{EKW10}, and continuity of composition from Lemma~A.2 of~\cite{EKW10}.
This gives continuous dependence of the velocity in $\XH3$.
Since $\om=-H\pa_xu$, this yields continuous dependence of $\eta$ in $\XH2$.
Restarting on finitely many local intervals shows that, whenever $\eta_0^{(j)}\to\eta_0$ in $\XHZmean2$, the corresponding solutions exist on every fixed $[0,T']\subset[0,T)$ for $j$ sufficiently large and satisfy
\begin{align*}
    \sup_{0\leq t\leq T'}
    \norm{\eta^{(j)}(t)-\eta(t)}_\XH2
    \longrightarrow0.
\end{align*}

Proposition~5.1 of~\cite{OSW08} preserves higher Sobolev regularity throughout the $\XH1$ lifespan.
Thus smooth approximating data give smooth solutions on $[0,T']$, and the convergence above justifies passage to the limit in the integral energy inequality of Lemma~\ref{lem:estH2}.
Finally, Theorem~3.2 and Proposition~5.1 of~\cite{OSW08} allow continuation in $\XH2$ past any finite time at which the $\XH1$ norm remains bounded, since $\norm{H\om}_{L^\infty}\lesssim\norm\om_\XH1$.

We now prove Theorem~\ref{thm:alln:nonlinear}, following the argument of Guo, Hallstrom, and Spirn~\cite{GuoHallstromSpirn2004,GuoHallstromSpirn}. The proof uses the following three estimates.

\begin{lemma}
    \label{lem:estH1}
    Let $p$ be a prime and $L_p$ the linear operator given in~\eqref{def:Ln} with $n=p$, and let $\om=1.1$.
    Then there exists $C_{1,p} > 0$ such that
    \begin{align*}
        \norm{e^{t L_p} \eta}_\XH1
        \leq
        C_{1,p} e^{t\om}
        \norm{\eta}_\XH1\qquad
        \text{for all }
        \eta\in\XHZmean{1}.
    \end{align*}
\end{lemma}
\begin{proof}
    The proof will be given in Section~\ref{sec:estH1}.
\end{proof}

\begin{lemma}
    \label{lem:estquad}
    The quadratic operator $Q$ given in~\eqref{def:Quadop} satisfies
    \begin{align*}
        \norm{Q(\eta,\eta)}_\XH1\leq C_Q\norm\eta_\XH2^2
        \qquad\text{for all }\eta\in\XHZmean2.
    \end{align*}
    For every smooth real-valued $\eta$ with zero mean, it also satisfies
    \begin{align*}
        \big|(Q(\eta,\eta),\eta)_{\XH2}\big|
        \leq C\norm\eta_\XH1\norm\eta_\XH2^2.
    \end{align*}
    Here $C_Q$ and $C$ are positive absolute constants.
\end{lemma}

\begin{proof}
    The algebra property of $\XH1$ and the boundedness of $H$ give
    \begin{align*}
        \norm{Q(\eta,\eta)}_\XH1
        \lesssim\norm\eta_\XH1^2
        +\norm\eta_\XH1\norm\eta_\XH2
        \lesssim\norm\eta_\XH2^2.
    \end{align*}
    For the second estimate,
    \begin{align*}
        \big|(\pa_x^2Q(\eta,\eta),\pa_x^2\eta)_\XL\big|
        &=
        \bigg|-\frac12\intpin H\eta(x)(\pa_x^2\eta(x))^2\dx+
        \intpin
        \pa_x\eta(x)H\pa_x\eta(x)\pa_x^2\eta(x)\dx
        \\&\qquad+
        \intpin
        \eta(x)H\pa_x^2\eta(x)\pa_x^2\eta(x)\dx
        \bigg|\lesssim 
        \norm\eta_\XH1\norm\eta_\XH2^2
        ,\\
        \big|(Q(\eta,\eta),\eta)_\XL\big|
        &=\bigg|\frac32\intpin H\eta(x)\eta(x)^2\dx\bigg|
        \lesssim
        \norm\eta_\XH1\norm\eta_\XH2^2.
    \end{align*}
\end{proof}

\begin{lemma}
    \label{lem:estH2}
Let $p$ be a prime. There are constants $C_{2,p}>0$ and $C>0$ such that, for every $T>0$, every real-valued solution $\eta\in C([0,T);\XHZmean2)$ of~\eqref{eqn:pertsinenonline} with $n=p$ satisfies
\begin{align*}
    \norm{\eta(t)}_\XH2
    \leq
    \norm{\eta(0)}_\XH2 e^{t/2}
    +
    \int_{0}^{t}
    e^{(t-s)/2}
    \Big(C_{2,p}\norm{\eta(s)}_\XH1+C\norm{\eta(s)}_\XH1\norm{\eta(s)}_\XH2\Big)
    \,ds
\end{align*}
for all $t\in[0,T)$.
\end{lemma}

\begin{proof}
We first take a smooth real-valued solution and write $g=\pa_x^2\eta$.
Using the skew-adjoint operator $\mathfrak{A}_p$ from Remark~\ref{rem:linf:wellposed}, we have
\begin{align*}
    \pa_x^2L_p\eta
    &=\mathfrak{A}_pg+\underbrace{(-\tfrac12C_p-S_pH)}_{\mathfrak{B}_p}g+R_p\eta,\\
    R_p\eta
    &:=-pS_p\pa_x\eta-pC_pH\pa_x\eta-p^2S_pH\eta
      -p^2C_p\eta-p^3C_pH\pa_x^{-1}\eta.
\end{align*}
The operator $\mathfrak{B}_p$ is bounded on $\XL$, and
\begin{subequations}
\label{aux:lem33:1}
    \begin{align}
        \norm{R_p\eta}_\XL
        &\leq(p^3+2p^2+2p)\norm\eta_\XH1,
        \\
        \norm{L_p\eta}_\XL
        &\leq(p+p^{-1}+2)\norm\eta_\XH1.
    \end{align}
\end{subequations}
Using the skew-adjointness of $\mathfrak{A}_p$ and $H$, we have
\begin{subequations}
    \begin{align}
    \Re(\mathfrak{A}_p g,g)_\XL
    &=0,
    \label{aux:lem33:2a}
    \\
    \Re(-S_pHg,g)_\XL
    &=\frac12([H,S_p]g,g)_\XL,
    \label{aux:lem33:2b}
\end{align}
\end{subequations}
where $[H,S_p]g:=H(S_pg)-S_pHg$.
Multiplication by $S_p$ shifts Fourier indices by $\pm p$, so the commutator vanishes on input frequencies $|k|>p$ and
$\norm{[H,S_p]g}_\XL\leq 2p\norm\eta_\XH1$. Using~\eqref{aux:lem33:2b},
\begin{align}
    \Re(\mathfrak{B}_p g,g)_\XL \leq
    \frac12|\Re(C_p g,g)_\XL|
    +
    |\Re(-S_pHg,g)_\XL|
    \leq
    \frac12 \norm{g}_\XL^2 +
    p\norm{g}_\XL\norm\eta_\XH1.
    \label{aux:lem33:3}
\end{align}
This gives
\begin{align*}
    \Re(L_p\eta,\eta)_{\XH2}
    &=
    \Re(L_p\eta,\eta)_\XL
    +
    \Re(\pa_x^2 L_p\eta,\pa_x^2\eta)_\XL\\
    &\leq (p+p^{-1}+2)\norm{\eta}_\XL\norm{\eta}_\XH1
       +\Re(\mathfrak{B}_p g,g)_\XL+\Re(R_p\eta,g)_\XL\\
    &\leq
    \frac12\norm\eta_\XH2^2
    +C_{2,p}\norm\eta_\XH1\norm\eta_\XH2.
\end{align*}
Here, using~\eqref{aux:lem33:1},\eqref{aux:lem33:2a} and \eqref{aux:lem33:3}, one may take $C_{2,p}=p^3+2p^2+4p+2+p^{-1}$.
By Lemma~\ref{lem:estquad}, we also obtain
\begin{align*}
    \frac12\pa_t\norm\eta_\XH2^2
    \leq\frac12\norm\eta_\XH2^2
    +C_{2,p}\norm\eta_\XH1\norm\eta_\XH2
    +C\norm\eta_\XH1\norm\eta_\XH2^2.
\end{align*}
Finally, set $u_\ep:=(\norm\eta_\XH2^2+\ep^2)^{1/2}$ for $\ep>0$. Dividing the previous inequality by $u_\ep$ gives
\begin{align*}
    \pa_t u_\ep
    \leq\frac12 u_\ep
    +C_{2,p}\norm\eta_\XH1
    +C\norm\eta_\XH1 u_\ep.
\end{align*}
Multiplying by $e^{-t/2}$ and integrating on $[0,t]$ gives
\begin{align*}
    u_\ep(t)
    \leq u_\ep(0)e^{t/2}
    +\int_0^t e^{(t-s)/2}\Big(C_{2,p}\norm{\eta(s)}_\XH1+C\norm{\eta(s)}_\XH1 u_\ep(s)\Big)\,ds.
\end{align*}
Letting $\ep\to0$ proves the asserted inequality for smooth solutions.
The estimate for $\XH2$ data follows by approximation with smooth initial data, which remain smooth on every compact interval of existence by Proposition~5.1 of~\cite{OSW08}, and continuous dependence in $C([0,T'];\XH2)$ for every $T'<T$.
\end{proof}

\subsection{Proof of Theorem~\ref{thm:alln:nonlinear}}

Assume $n=p$ for a prime $p$, so that Lemmas~\ref{lem:estH1} and~\ref{lem:estH2} both apply, and let $(\la,\eta_0):=(\la_p,\eta_p)$ denote the pair given by Theorem~\ref{the:mainlin} with $n=p$, chosen real-valued by taking a nonzero real or imaginary part and normalized so that $\norm{\eta_0}_{\XH1}=1$.

Denote now by $\linsol$ the solution of the linearized equation~\eqref{eqn:linf} with initial condition $\delta \eta_0$, indeed we have
\begin{align*}
    \linsol(x,t) = \delta e^{t\la}\, \eta_0(x).
\end{align*}

Let $\nonsol$ denote the maximal solution of the nonlinear equation~\eqref{eqn:pertsinenonline} with the same initial condition $\delta\eta_0$, defined on $[0,T_{\max,\de})$. We define $\difsol$ as the difference between these two solutions,
\begin{align*}
    \difsol := \nonsol - \linsol.
\end{align*}

Let $\om=1.1$ and $C_{1,p}$, $C_{2,p}$, $C$, and $C_Q$ be the constants of Lemmas~\ref{lem:estH1}--\ref{lem:estH2}.
Define $\ti{C}_{2,p}$ and $\ti{C}_p$ by
\begin{align*}
    \ti{C}_{2,p}:=\norm{\eta_0}_\XH2+\frac{2C_{2,p}}{\la-\frac12},
    \qquad
    \ti{C}_p:=\frac{4C_{1,p}C_Q\ti{C}_{2,p}^2}{2\la-\om},
\end{align*}
which are well-defined and positive since $\la>\frac12$ and $2\la>\om$, recall $\la\in\Ilambd$ by Theorem~\ref{the:mainlin}.

Fix
\begin{align}
    \label{eqn:theta0}
    \rthr:=\min\bigg\{\frac{2\la-\frac12}{8C},\ \frac1{2\ti{C}_p}\bigg\}>0,
\end{align}
and define for every $0<\de<\rthr$,
\begin{align*}
    T_\de^{\rthr}&:=\frac1\la\log\frac\rthr\de,
    \\
    T_\de^1&:=\sup\big\{t\in[0,T_{\max,\de}):\norm{\nonsol(\cdot,s)}_\XH1\leq2\de e^{s\la}\text{ for all }s\in[0,t]\big\},
    \\
    T_\de^2&:=\sup\big\{t\in[0,T_{\max,\de}):\norm{\nonsol(\cdot,s)}_\XH2\leq2\ti{C}_{2,p}\de e^{s\la}\text{ for all }s\in[0,t]\big\}.
\end{align*}
Both $T_\de^1,T_\de^2>0$ by continuity of $\nonsol$ in $\XH2$ and the strict inequalities at $t=0$, since $\norm{\eta_0}_\XH1=1$ and $\ti{C}_{2,p}>\norm{\eta_0}_\XH2$.

Notice that $\difsol(\cdot,0)=0$, and differentiating $\difsol$ in time, we get
\begin{align}
    \label{eqn:diffsol}
    \pa_t \difsol =
    L_p \difsol + Q(\nonsol,\nonsol).
\end{align}
Therefore, Duhamel's formula applied to~\eqref{eqn:diffsol} yields
\begin{align*}
    \difsol(x,t) =
    \int_0^t e^{(t-s) L_p}
    Q(\nonsol,\nonsol)(x,s)\,ds.
\end{align*}
Taking the $\XH1$ norm and applying Lemmas~\ref{lem:estH1} and~\ref{lem:estquad},
\begin{align}
    \label{eqn:duhamelh1}
    \norm{\difsol(\cdot,t)}_\XH1
    \leq
    C_{1,p}C_Q
    \int_0^t
    e^{(t-s)\om}
    \norm{\nonsol(\cdot,s)}_\XH2^2
    \,ds.
\end{align}

Let $0\leq t<\min\{T_\de^1,T_\de^2\}$. From~\eqref{eqn:duhamelh1} and the definition of $T_\de^2$, we have
\begin{align}
    \label{eqn:difsolbound}
    \norm{\difsol(\cdot,t)}_\XH1
    \leq
    4C_{1,p}C_Q\ti{C}_{2,p}^2\de^2
    \int_0^t
    e^{(t-s)\om}
    e^{2s\la}
    \,ds
    \leq
    \underbrace{\frac{4C_{1,p}C_Q\ti{C}_{2,p}^2}{2\la-\om}}_{\ti{C}_p}\,\de^2 e^{2t\la}.
\end{align}
Using~\eqref{eqn:difsolbound}, we get the following estimate for $\nonsol$
\begin{align}
    \norm{\nonsol(\cdot,t)}_\XH1
    \leq
    \norm{\linsol(\cdot,t)}_\XH1
    +
    \norm{\difsol(\cdot,t)}_\XH1
    \leq
    \de e^{t\la}
    \big({\underbrace{\norm{\eta_0}_\XH1}_{1}}+\ti{C}_p\,\de e^{t\la}\big).
    \label{eqn:h1bound}
\end{align}
Applying Lemma~\ref{lem:estH2} to $\nonsol$, we obtain
\begin{align}
    \norm{\nonsol(\cdot,t)}_\XH2
    &\leq\de\norm{\eta_0}_\XH2e^{t/2}
    +2C_{2,p}\de\int_0^te^{(t-s)/2}e^{s\la}\,ds+4C\ti{C}_{2,p}\de^2\int_0^te^{(t-s)/2}e^{2s\la}\,ds
    \nonumber
    \\
    \nonumber
    &\leq\de\norm{\eta_0}_\XH2e^{t/2}
    +
    2C_{2,p}\de
    \frac{e^{t\la}}{\la-\frac12}
    +
    4C\ti{C}_{2,p}\de^2
    \frac{e^{2t\la}}{2\la-\frac12}\\
     &\leq
     \de e^{t\la} \bigg(
     {\underbrace{\norm{\eta_0}_\XH2
    +
    \frac{2C_{2,p}}{\la-\frac12}}_{\ti C_{2,p}}}
    +
    \ti{C}_{2,p}
    \frac{4C}{2\la-\frac12}\de e^{t\la}
    \bigg).
    \label{eqn:h2bound}
\end{align}

Set $T_\de^\ast:=\min\{T_\de^1,T_\de^2\}$. For $0\leq t<\min\{T_\de^\ast,T_\de^{\rthr}\}$, the choice~\eqref{eqn:theta0},~\eqref{eqn:h1bound}, and~\eqref{eqn:h2bound} imply
\begin{align*}
    \norm{\nonsol(\cdot,t)}_\XH1\leq\tfrac32\de e^{t\la},
    \qquad
    \norm{\nonsol(\cdot,t)}_\XH2\leq\tfrac32\ti{C}_{2,p}\de e^{t\la}.
\end{align*}
If $T_\de^\ast<T_\de^{\rthr}$, the $\XH2$ bound allows continuation past $T_\de^\ast$, and both inequalities are strict improvements of the defining bounds of $T_\de^1$ and $T_\de^2$. Continuity then contradicts their maximality. Hence $T_\de^\ast\geq T_\de^{\rthr}$. The same continuation criterion rules out $T_{\max,\de}=T_\de^{\rthr}$, so the solution is defined at the escape time and the bounds there follow by continuity.

At $t=T_\de^{\rthr}$ we have $\de e^{T_\de^{\rthr}\la}=\rthr$, so $\norm{\linsol(\cdot,T_\de^{\rthr})}_\XH1=\rthr$, while~\eqref{eqn:difsolbound} and $\rthr\leq\frac1{2\ti{C}_p}$ implies
\begin{align*}
\norm{\difsol(\cdot,T_\de^{\rthr})}_\XH1\leq\ti{C}_p\rthr^2\leq\frac12\rthr.
\end{align*}
Hence,
\begin{align*}
    \norm{\nonsol(\cdot,T_\de^{\rthr})}_\XH1
    \geq
    \norm{\linsol(\cdot,T_\de^{\rthr})}_\XH1
    -
    \norm{\difsol(\cdot,T_\de^{\rthr})}_\XH1
    \geq
    \tfrac12\rthr\qquad
    \text{for all }\de < \rthr.
\end{align*}
This shows that solutions with initial data of arbitrarily small $\XH1$ norm $\de<\rthr$ eventually leave the $\XH1$-ball of fixed radius $\frac12\rthr$.

Finally, let $n\geq2$ and let $p$ be a prime divisor of $n$, so that $n=pl$ for some integer $l\geq1$, and let $\rthr$ and $T_\de^{\rthr}$ be the quantities defined above, with $\nonsol$ the corresponding solution of~\eqref{eqn:pertsinenonline} with $n=p$.

Define $\resol(x,t):=\nonsol(lx,t)$, so that $\resol$ solves~\eqref{eqn:pertsinenonline}, with initial condition $\resol(\cdot,0)=\de\eta_0(l\,\cdot)=:\de\eta_{0,n}$.

Notice now that the Sobolev norm is equivalent after rescaling. Indeed, for every $s\geq0$ and $f\in\XH s$,
\begin{align*}
    \norm f_\XH s^2
    \leq
    \norm{f(l\,\cdot)}_\XH s^2
    \leq
    l^{2s}\norm f_\XH s^2,
\end{align*}
since $1+|j|^{2s}\leq1+l^{2s}|j|^{2s}\leq l^{2s}(1+|j|^{2s})$ for $l\geq1$. Setting $\ep_n:=\frac12\rthr$ and $T_{\de,n}:=T_\de^{\rthr}$, we have for any $\de<\rthr$ that
\begin{align*}
\norm{\resol(\cdot, T_{\de,n})}_\XH1 \geq
\norm{\nonsol(\cdot, T_{\de,n})}_\XH1 \geq
 \ep_n,
\end{align*}
while
\begin{align*}
\norm{\resol(\cdot,0)}_\XH1 \leq l\de\norm{\eta_0}_\XH1 = l\de
\end{align*}
and $\norm{\resol(\cdot,0)}_\XH2 < \infty$, both arbitrarily small as $\de\to0$.

\section{Proof of Proposition~\ref{prop:eigte}}
\label{sec:eigte}

To prove Proposition~\ref{prop:eigte}, we construct a real-valued continuous function $\detcrit$ of $\te\in\Itheta$ and $\la\in\Ilambd$ whose zeros characterize the eigenvalues of $L_\te^{\rm fib}$.
The function $\detcrit$ is defined from a real $4\times 4$ matrix $\Mcrit$ by
\begin{align}
    \detcrit(\la,\te) := \det(\Mcrit),
\end{align}
so that Proposition~\ref{prop:eigte} reduces to showing that $\detcrit(\cdot,\te)$ has a zero in $\Ilambd$ for $\te=\te_p$ at every prime $p<N$, and for every $\te\in J$.

Section~\ref{sec:redfindim} reduces
\begin{align}
\label{Lteet=laet}
    L_\te^{\rm fib}\eta = \la\eta
\end{align}
to a second-order singular ODE system. We study this system using Fuchsian theory in Section~\ref{sec:fuchs2}, where the matrix $\Mcrit$ is also constructed in Definition~\ref{def:Mat}.

\begin{rem}
    The function $\detcrit(\cdot,\te)$ characterizes the eigenvalues of $L_\te^{\rm fib}$ on a domain larger than $\Ilambd$. Indeed, we will work with $\la\in\C$ and specify any additional assumptions on $\la$ within each lemma.
\end{rem}

\subsection{Reduction to finite dimension}
\label{sec:redfindim}

Assume we have $\la\in\C$ and $0\neq \eta\in\XL$ solving~\eqref{Lteet=laet}. Taking the Fourier transform, we obtain
\begin{align}
\label{eqn:rec}
    \la \ha \eta_k =
    \ti{A}_{k+1,\te} \ha \eta_{k+1} +
    \ti{B}_{k-1,\te} \ha \eta_{k-1}
    \qquad\text{for all }k\in\Z,
\end{align}
where $\ti{A}_{k,\te}$ and $\ti{B}_{k,\te}$ were defined in~\eqref{def:ABkt}. 

Rather than solving the recurrence~\eqref{eqn:rec} directly, we follow~\cite{BHTW} and translate it into a holomorphic problem on the unit disk. To a solution $\eta$ of~\eqref{eqn:rec} we associate holomorphic functions $(F^+,F^-)$ whose generating series record the Fourier modes of $\eta$. These functions satisfy a second-order ODE system of the form~\eqref{holsys}, and its right-hand sides involve only the values $F^\pm(0)$. This reduces the eigenvalue problem to the existence of holomorphic solutions of~\eqref{holsys}.

For $\la\in\C$ and $\te\in(0,1)$, define the operator $\OLdisk\la\te:\XHholo \to \XHholo$ by
\begin{equation}
\label{def:holo}
    \OLdisk\la\te F(z)
    := 
    z^2(z^2-1)\pa_z^2 F(z) +
    z \Pglo(z)\pa_z F(z)+
    \Qglo(z)F(z),
\end{equation}
where
\begin{align}
\label{def:PQtl}
    \Pglo(z):=
        (2\te-1)z^2+2\la z-(2\te+1),&&
    \Qglo(z):=
        (\te-1)^2z^2+2\la\te z-(\te^2-1).
\end{align}
In terms of this operator, we seek pairs $(F^+,F^-)$ of holomorphic functions on the open unit disk $\D:=\{z\in\C:|z|<1\}$ satisfying the system
\begin{equation}
    \label{holsys}
    \left.
    \begin{aligned}
        \OLdisk\la{\te^+} F^+(z) =
            z \te (2-\te)
            F^-(0)
            +
            (1-\te^2)
            F^+(0)
    \\
        \OLdisk\la{\te^-} F^-(z) =
            z (1-\te^2)
            F^+(0)
            +
            \te (2-\te)
            F^-(0)
    \end{aligned}
    \right\}
\end{equation}
Here $\te^+:=\te$ and $\te^-:=1-\te$. This is the equation the generating functions $F^\pm$ satisfy directly (Lemma~\ref{lem:red_to_fin} below).

As we will see, near $z=1$, the kernel of $\OLdisk\la\te$ is locally spanned by a holomorphic function and a second independent solution of the form
\begin{align*}
    (z-1)^{2-\la}
\end{align*}
times a holomorphic function. Near $z=-1$,
the identities
\begin{align*}
    \Pglo(z)=\Pglolm(-z)
    \andspace
    \Qglo(z)=\Qglolm(-z)
\end{align*}
similarly yield a holomorphic solution and a singular solution of the form
\begin{align*}
    (z+1)^{2+\la}
\end{align*}
times a holomorphic function.

Our strategy is to construct $F^\pm$, holomorphic on $\D$, so that the coefficient of the singular term at $z=1$ vanishes. The regularity of the eigenfunction is then dictated entirely by the behavior at $z=-1$, where the exponent $2+\la$ implies that
\begin{align*}
    \eta\in\XH{s}\qquad 
    \text{for all }s<\frac32+\la.
\end{align*}

Before showing the equivalence of both problems, we prove the following technical lemma.
\begin{lemma}
\label{lem:hol_extension}
    Let $s\geq0$ and $f\in\XH{s}$ with $\ha{f}_k=0$ for all $k<0$. 
    Then the power series $F(z):=\sum_{k\geq 0}\ha{f}_k z^k$ converges absolutely for every $z\in\D$ and defines $F\in\XHhol{s}$.
\end{lemma}

\begin{proof}
    Since $f\in\XH{s}$, its Fourier coefficients satisfy
    \begin{align*}
        |\ha{f}_k|\lesssim (1+k^2)^{|s|}
        \qquad \text{for all } k\geq 0.
    \end{align*}
    For any $r\in(0,1)$, the series $\sum_{k\geq 0}|\ha{f}_k|r^k$ converges by comparison with $C\sum_{k\geq 0}(1+k^2)^{|s|} r^k<\infty$.
    Since this bound is uniform on ${|z|\leq r}$, the series converges locally uniformly in $\D$, so $F$ is holomorphic there.
    By construction $F|_\T\in \XH s$, hence $F\in\XHhol{s}$.
\end{proof}

\begin{lemma}
\label{lem:red_to_fin}
    Let $\la\in\C$, $\te\in(0,1)$, and $g\in\XH s$ for some $s\geq0$.
    Set $\te^+:=\te$, $\te^-:=1-\te$. Define $g^\pm:=\Pcal^\pm g\in\XH s$
    and let $(G^+,G^-)\in\XHholpair{s}$ given by
    \begin{align*}
        G^+(z) := 2i\sum_{k\geq 0}\wiajustao{g^+}{k} z^k
        \andspace
        G^-(z) :=-2i\sum_{k\geq 0}\wiajustao{g^-}{k} z^k
        .
    \end{align*}
    Then there exists $\eta\in\XH s$ satisfying
    \begin{align}
    \label{eqn:reducforcedcase}
        (L_\te^{\rm fib}-\la)\eta = g,
    \end{align}
    if and only if there exists
    $(F^+,F^-)\in\XHholpair{s+1}$ solving
    \begin{equation}
        \label{holsys:forced}
        \left.
        \begin{aligned}
            \OLdisk\la{\te^+}F^+(z)
                &= z\te(2-\te)F^-(0)+(1-\te^2)F^+(0)+zG^+(z)\\
            \OLdisk\la{\te^-}F^-(z)
                &= z(1-\te^2)F^+(0)+\te(2-\te)F^-(0)+zG^-(z)
        \end{aligned}
        \right\}
    \end{equation}
    In particular, taking $g=0$, the system~\eqref{holsys:forced} reduces
    to~\eqref{holsys}, recovering the eigenvalue equivalence with~\eqref{eqn:eigthe}.
\end{lemma}

\begin{proof}
    Let $\eta \in \XH s$ satisfy~\eqref{eqn:reducforcedcase}.
    Define $f := (\pa_x + i\te)^{-1} \eta\in\XH{s+1}$.
	Applying $\Pcal^+$ to~\eqref{eqn:reducforcedcase}, we obtain
	\begin{align}
        \la (\pa_x + i\te) f^+ + g^+
        =
		\Pcal^+ L_{\te}^{\rm fib} \eta
        =
        \sum_{k \geq 0}
		(
			\ha \eta_{k+1} \ti{A}_{k+1,\te}
			+
			\ha \eta_{k-1} \ti{B}_{k-1,\te}
		)
        \,\foba k.
	\end{align}

	Thus, for every $k \geq 0$,
	\begin{align}
        \la (k+\te) 
        \wiajustao{f^+}{k}=
		(k+1+\te)
        \ha f_{k+1}
        \ti{A}_{k+1,\te}
		+
		(k-1+\te)
        \ha f_{k-1}
        \ti{B}_{k-1,\te}
        +i \wiajustao{g^+}{k}.
	\end{align}
    From~\eqref{def:ABkt}, we get
    \begin{align}
		(k-2+\te)(k-2+2\de_{k,0}+\te) \ha f_{k-1}
        +
        2\la (k+\te) \wiajustao{f^+}{k}
        -
		(k+2+\te)(k+\te) \ha f_{k+1}
        =2i\wiajustao{g^+}{k}.
        \label{eqn:krelgp}
	\end{align}
	
	The projection $f^+:=\Pcal^+ f$ belongs to $\XH{s+1}$ and has only nonnegative Fourier modes. By Lemma~\ref{lem:hol_extension} applied to $f^+\in\XH{s+1}$, the power series
	\begin{align*}
		F^+(z) := \sum_{k \geq 0} \ha f_k^+ z^k
	\end{align*}
	converges absolutely for every $z\in\D$ and defines $F^+\in\XHhol{s+1}$.

    Multiplying~\eqref{eqn:krelgp} by $z^k$ and summing over $k \geq 0$, we obtain
    \begin{align}
        I_1+I_2+I_3
        =
        G^+
        \label{eqn:ktohol:sum}
    \end{align}
    where
    \begin{align*}
        I_1(z) &:= \sum_{k \geq 0}
            (k-2+\te)(k-2(1-\de_{k,0})+\te)
            \ha f_{k-1} z^k,
            \\
        I_2(z) &:= \sum_{k \geq 0}
            2\la (k+\te) 
            (\ha{f^+})_k z^k,
            \qquad
        I_3(z) := -\sum_{k \geq 0}
            (k+2+\te)(k+\te) 
            \ha f_{k+1} z^k.
    \end{align*}

    For $I_1$ we have
    \begin{align*}
        I_1(z) &= 
            z\sum_{k \geq -1}
            (k-1+\te)(k+1-2(1-\de_{k+1,0})+\te)
            \ha f_k z^k\\
        &=
            z\sum_{k \geq 0}
            (k-1+\te)^2 (\ha{f^+})_k z^k
            -
            \te (2 - \te) \ha f_{-1}\\
        &=
            z^3 \pa_z^2 F^+(z) +
            (2\te - 1) z^2 \pa_z F^+(z) +
            (1-\te)^2 z F^+(z)
            -
            \te (2-\te) \ha f_{-1}.
    \end{align*}
    Next,
    \begin{align*}
        I_2(z) &= 
            \sum_{k \geq 0}
            2\la (k+\te) (\ha{f^+})_k z^k
        =
            2\la z \pa_z F^+(z) +
            2\la\te F^+(z),
    \end{align*}
    and
    \begin{align*}
        zI_3(z) &= 
            z\sum_{k \geq 0}
            -(k+2+\te)(k+\te) \ha f_{k+1} z^k\\
        &= 
            \sum_{k \geq 0}
            -(k+1+\te)(k-1+\te) (\ha{f^+})_k z^k
            -(1-\te)(1+\te) \ha f_0\\
        &=  -
            z^2 \pa_z^2 F^+(z) -
            (2\te + 1) z \pa_z F^+(z) +
            (1-\te)(1+\te) F^+(z)
            - (1-\te)(1+\te) \ha f_0.
    \end{align*}

    Substituting the expressions for $I_1$, $I_2$, and $I_3$ into~\eqref{eqn:ktohol:sum}, we find that $F^+$ satisfies the first equation in~\eqref{holsys:forced}. 
    
    By Remark~\ref{rem:AkBksym} we have $\ti{A}_{-k,\te^+}=\ti{B}_{k-1,\te^-}$. Substituting $k\to -k-2$ and $\te^+\to\te^-$ in the same identity also yields $\ti{B}_{-k-2,\te^+}=\ti{A}_{k+1,\te^-}$. These two identities allow us to find the recurrence that the Fourier modes of $\eta^\sharp$ satisfy, as the following computations show
    \begin{align*}
        \la \wiajustao{\eta^\sharp}{k} &=
        \la \ha \eta_{-k-1} =
            \ti{A}_{-k,\te^+}
            \ha \eta_{-k}
            +
            \ti{B}_{-k-2,\te^+}
            \ha \eta_{-k-2} - \ha{g}_{-k-1}\\
        &=
            \ti{B}_{k-1,\te^-}
            \ha \eta_{-k}
            +
            \ti{A}_{k+1,\te^-}
            \ha \eta_{-k-2}
             - \ha{g}_{-k-1}
        =
            \ti{B}_{k-1,\te^-}
            \wiajustao{\eta^\sharp}{k-1}
            +
            \ti{A}_{k+1,\te^-}
            \wiajustao{\eta^\sharp}{k+1}
             - 
             \wiajustao{g^\sharp}{k}.
    \end{align*}
    For $f^- := \Pcal^- f$, we have
    \begin{align*}
        \wiajustao{f^-}{k} = \frac{i}{k+\te^-} \wiajustao{\eta^\sharp}{k}.
    \end{align*}
    Hence, for all $k \geq 0$, it holds that
    \begin{align}
    \label{eqn:krelgp2}
		(k-2+\te^-)
        (k-2(1-\de_{k,0})+\te^-) \ha{f}^\sharp_{k-1}
        &+
        2\la (k+\te^-) \wiajustao{f^-}{k}
        \nonumber
        \\
		&\quad-
		(k+2+\te^-)(k+\te^-) \ha{f}^\sharp_{k+1}
		=
		-2i\wiajustao{g^-}{k},
	\end{align}
    so that the extension $F^-$ satisfies the second equation in~\eqref{holsys:forced}.

    Conversely, given $(F^+,F^-)\in\XHholpair{s+1}$ solving~\eqref{holsys:forced},
    set $f^\pm:=F^\pm|_\T$, and define $f$ and $\eta:=\pa_x f+i\te f$ by
    \begin{align*}
        f(x) = f^+(x) + e^{-ix}f^-(-x).
    \end{align*}
    The Fourier modes of $f^\pm$ satisfy~\eqref{eqn:krelgp} and~\eqref{eqn:krelgp2},
    and reversing the computation gives~\eqref{eqn:reducforcedcase}.
\end{proof}

\subsection{Second-order ODE system}
\label{sec:fuchs2}

The analysis of~\eqref{holsys} proceeds in two steps. In Section~\ref{sec:fu:loc} we study the local behavior of solutions to~\eqref{def:holo}, while in Section~\ref{sec:fu:glo} we use these local solutions to construct global pairs of solutions to~\eqref{holsys} with good boundary behavior on $\partial\D=\T$.

The main analytical problem is to understand, for a given holomorphic function $G$, whether there exists a holomorphic solution $F:\D\to\C$ to
\begin{align}
\label{ode:LF=G}
    \OLdisk\la\te F = G,
\end{align}
whether such a solution is unique, and what boundary regularity its trace $F|_\T$ possesses. Note that~\eqref{holsys} is precisely a coupled system of two equations of the form~\eqref{ode:LF=G}, where the right-hand sides are linear in $F^\pm(0)$.
The operator $\OLdisk\la\te$ is a second-order Fuchsian differential operator with regular singular points at $z=0$ and $z=\pm 1$. 
We refer to~\cite{coddington1955theory,ince1956ordinary,whittaker1996modern,hille1997ordinary,TeschlODE} for further details and background.
This structure lets us describe the local behavior of solutions near each singularity explicitly, which is the content of Section~\ref{sec:fu:loc}.

\subsubsection{Local analysis}
\label{sec:fu:loc}

We begin with the homogeneous equation
\begin{equation}
\label{ode:hom}
    z^2(z^2-1)\pa_z^2 F(z) +
    z 
        \Pglo(z)
    \pa_z F(z)
    +
        \Qglo(z)
    F(z)=0,
\end{equation}
where $\Pglo$ and $\Qglo$ were given in~\eqref{def:PQtl}.

The equation~\eqref{ode:hom} has three regular singular points, namely
\begin{align*}
    z=0,\qquad z=1,\qquad z=-1.
\end{align*}

The next three lemmas describe the local form of the solutions near each
singular point.

\begin{lemma}
\label{lem:loc:0}
    Let $\te\in(0,1)$ and $\la\in\C$.
    Let $\Omega \subset \C$ be open and simply connected, with $0\in\Omega$ and $1,-1\notin\Omega$.
    
    Then there exists $r>0$ such that, after fixing a branch of $\log(z)$ on any simply connected set $U\subset D_r(0)-\{0\}$, every solution of~\eqref{ode:hom} on $U$ is a linear combination of the functions $\Fholh10$ and $\Fholh20$, given by
    \begin{align*}
        \Fholh10(z) = 
            z^{-\te+1} \phi_1(z)
            \andspace
        \Fholh20(z) =  
            z^{-\te-1} \phi_2(z) 
            - 
            \Fholh10(z) \log (z),
    \end{align*}
    where $\phi_1,\phi_2:D_r(0)\to\C$ are holomorphic functions with $\phi_1(0)=1$ and $\phi_2(0)=-\frac12$.
\end{lemma}

\prooffuchstoappendixH

\begin{lemma}
\label{lem:loc:1}
    Let $\te\in(0,1)$ and $\la\in\C-\Z$.
    Let $\Omega \subset \C$ be open and simply connected, with $1\in\Omega$ and $0,-1\notin\Omega$.

    Then there exists $r>0$ such that, after fixing a branch of $\log(z-1)$ on any simply connected set $U\subset D_r(1)-\{1\}$, every solution of~\eqref{ode:hom} on $U$ is a linear combination of the functions $\Fholh11$ and $\Fholh21$, given by
    \begin{align*}
        \Fholh11(z) =
            \phi_1(z)\andspace
        \Fholh21(z) =
            (z-1)^{2-\la} \phi_2(z),
    \end{align*}
    where $\phi_1$ and $\phi_2$ are holomorphic functions on $D_r(1)$, with $\phi_1(1)=\phi_2(1)=1$.
\end{lemma}

\begin{rem}
    By Theorem~4.5 of~\cite{TeschlODE}, in both Lemma~\ref{lem:loc:0} and Lemma~\ref{lem:loc:1} the functions $\phi_1$ and $\phi_2$ have radius of convergence at least $1$ (the distance from each base point to the nearest other singularity).
\end{rem}

\prooffuchstoappendixH

\begin{lemma}
\label{lem:loc:-1}
    Let $\te\in(0,1)$ and $\la\in\C-\Z$.
    Let $\Omega \subset \C$ be open and simply connected, with $-1\in\Omega$ and $0,1\notin\Omega$.
    
    Then for any simply connected set $U\subset\Omega-\{-1\}$, after fixing a branch of $\log(z+1)$ on $U$, every solution of~\eqref{ode:hom} on $U$ is a linear combination of $\Fholh1{-1}$ and $\Fholh2{-1}$, defined by
    \begin{align*}
        \Fholh1{-1}(z) =
            \phi_1(z)\andspace
        \Fholh2{-1}(z) =
            (z+1)^{2+\la} \phi_2(z),
    \end{align*}
    where $\phi_1$ and $\phi_2$ are holomorphic functions on $\Omega$, with $\phi_1(-1)=\phi_2(-1)=1$.
\end{lemma}

\prooffuchstoappendixH

\begin{rem}
\label{rem:ess_spec_frobenius}
    Fix $0<s<\frac32$ and a noninteger $\la$ with $|\Re\la|<\frac32-s$. For any $\te\in(0,1)$, prescribe nonzero values $(a^+,a^-)$ in~\eqref{holsysa}.
    The resulting solutions are holomorphic throughout $\D$, and Lemma~\ref{lem:hol_impl} gives $F^\pm(0)=a^\pm$. Near $z=\pm1$, their only possible nonanalytic terms are the Frobenius branches with exponents $2\mp\la$, respectively.
    The condition on $\la$ places both components in $\XHhol{s+1}$, so Lemma~\ref{lem:red_to_fin} yields a nonzero eigenfunction of $L_\te^{\rm fib}$ in $\XH s$.
\end{rem}

In general, the local solutions from the previous lemmas are not holomorphic. However, the next lemmas show that one can still construct holomorphic solutions to~\eqref{ode:LF=G} when we have a holomorphic affine term $G$.

\begin{lemma}
\label{lem:loc:aff0}
    Let $r>0$ be as in Lemma~\ref{lem:loc:0} and let $G:D_r(0)\to\C$ be holomorphic. Then there exists a holomorphic function $\Fholp 0:D_r(0)\to\C$ satisfying~\eqref{ode:LF=G}.
    
    Moreover, the Taylor coefficients of $\Fholp 0$ can be computed recursively to arbitrary order.
\end{lemma}

\prooffuchstoappendixW{Lemma~\ref{lem:loc:0}}

We now give conditions under which, for every holomorphic function $G$ at $0$, the function $\Fholp0$ from Lemma~\ref{lem:loc:aff0} is the unique holomorphic solution to~\eqref{ode:LF=G}. To prove this, we first need the following technical lemma.
\begin{lemma}
\label{lem:tec:aff0}
	Let $\te\in(0,1)$ and $\la\in\C$.
    Let $\Fholh10, \Fholh20$ be as in Lemma~\ref{lem:loc:0}.
    Then the unique linear combination of $\Fholh10$ and $\Fholh20$, which is holomorphic at $0$, is the trivial zero function.
\end{lemma}

\begin{proof}
	Assume that there exists a holomorphic function $F$ at $0$ such that
	\begin{align*}
		F 
		= 
		C_1 \Fholh10
		+
		C_2 \Fholh20
	\end{align*}
	for some $C_1,C_2\in\C$.
	
	Let $r>0$ as in Lemma~\ref{lem:loc:0} and consider the set $U$ given by
	\begin{align*}
		U:=\{s e^{ix}:\, 0<s<r\ \text{and }0<x<2\pi\},
	\end{align*}
	so that there is a choice of $\log(z)$ for which $\Fholh10$ and $\Fholh20$ are holomorphic on $U$. 
	
	Therefore, since $F$ is holomorphic, we find for all $s\in(0,r)$ that
	\begin{align}
		0 &= 
		\lim_{x\to 2\pi^-} 
			F(s e^{ix})
		-
		\lim_{x\to 0} 
			F(s e^{ix})
		\nonumber
		\\
		%%%%%%%%%%%%%%%%%%%%%%%%%%%%%%%%%%%%%%%%%%
		&=
		C_1
		\big(
		\lim_{x\to 2\pi^-}
			\Fholh10(s e^{ix})
		-
		\lim_{x\to 0}
			\Fholh10(s e^{ix})
		\big)
		+
		C_2
		\big(
		\lim_{x\to 2\pi^-} 
			\Fholh20(s e^{ix})
		-
		\lim_{x\to 0}
			\Fholh20(s e^{ix})
		\big)
		\nonumber
		\\
		%%%%%%%%%%%%%%%%%%%%%%%%%%%%%%%%%%%%%%%%%%
		&=
		C_1
		\big(
			e^{(-\te+1) 2\pi i}
		-
			1
		\big) s^{-\te+1}  \phi_1(s) -
			2\pi i C_2
			s^{-\te+1}
			e^{(-\te+1) (2\pi i)}
			\phi_1(s)
		\nonumber\\
		&\qquad
		-
			C_2
			\big(
				e^{(-\te+1) 2\pi i}
				-
				1
			\big)
			s^{-\te+1}
			\log(s)
			\phi_1(s) +
			C_2
			\big(
				e^{(-\te-1) 2\pi i}
				-
				1
			\big)
			s^{-\te-1}
			\phi_2(s)
		.\label{eqn:tec:aff0}
	\end{align}
	
	Hence, multiplying~\eqref{eqn:tec:aff0} by $s^{\te+1}$ yields
		\begin{align}
		0
		&=
		C_1
		\big(
			e^{(-\te+1) 2\pi i}
		-
			1
		\big) s^2  \phi_1(s) -
			2\pi i C_2
			s^2
			e^{(-\te+1) (2\pi i)}
			\phi_1(s)
		\nonumber\\
		&\qquad
		-
			C_2
			\big(
				e^{(-\te+1) 2\pi i}
				-
				1
			\big)
			s^2
			\log(s)
			\phi_1(s) +
			C_2
			\big(
				e^{(-\te-1) 2\pi i}
				-
				1
			\big)
			\phi_2(s)
		.\label{eqn:tec:aff1}
	\end{align}
	
	Since $\phi_1$ and $\phi_2$ are in particular smooth, the coefficient of $s^2\log(s)$ must vanish, so that
	\begin{align*}
		C_2 \big(
				e^{(-\te+1) 2\pi i}
				-
				1
			\big) =0.
	\end{align*}
	Since $\te\notin\Z$, the factor $e^{(-\te+1) 2\pi i}-1$ is nonzero, hence $C_2=0$. From~\eqref{eqn:tec:aff1} we then get
	\begin{align*}
		0
		&=
		C_1
		\big(
			e^{(-\te+1) 2\pi i}
		-
			1
		\big) s^2  \phi_1(s).
	\end{align*}
   Finally, we have $C_1=0$ since $\te\notin\Z$.
\end{proof}

\begin{corollary}\label{cor:Fp0}
	Let $G$ be a holomorphic function at $0$. Then $\Fholp0$ is the unique holomorphic function solving~\eqref{ode:LF=G}.
\end{corollary}
\begin{proof}
    Let $F$ be a holomorphic function at $0$ solving~\eqref{ode:LF=G}. Then $\ti F:=F-\Fholp0$ is holomorphic at $0$ and solves~\eqref{ode:hom}. By Lemma~\ref{lem:loc:0}, on a punctured neighborhood of $0$ we may write $\ti F=C_1\Fholh10+C_2\Fholh20$ for some $C_1,C_2\in\C$. Since $\ti F$ is holomorphic at $0$, Lemma~\ref{lem:tec:aff0} gives $C_1=C_2=0$, so $\ti F= 0$ and $F=\Fholp0$.
\end{proof}

\begin{lemma}
\label{lem:loc:aff1}
    Let $r>0$ and $\Fholh11:D_r(1)\to\C$ be as in Lemma~\ref{lem:loc:1}, let $G:D_r(1)\to\C$ be holomorphic and assume $\la\notin\Z$. Then there exists a holomorphic function $\Fholp1:D_r(1)\to\C$ satisfying~\eqref{ode:LF=G}.

    Therefore, every holomorphic function $F$ at $1$ solving~\eqref{ode:LF=G} can be written as
    \begin{align*}
        F = \Fholp1 + C \Fholh11,
    \end{align*}
    for some $C\in\C$.

    Moreover, the Taylor coefficients of $\Fholp1$ can be computed recursively to arbitrary order.
\end{lemma}

\begin{proof}
    Existence of $\Fholp1$ follows by solving the Frobenius recurrence for the inhomogeneous equation at $z=1$. For the characterization, note that among the two local solutions of~\eqref{ode:hom} at $1$, only $\Fholh21=(z-1)^{2-\la}\phi_2$ fails to be holomorphic there (as $\la\notin\Z$), so any holomorphic solution of~\eqref{ode:LF=G} at $1$ can only involve $\Fholh11$. Full details are in Appendix~\ref{sec:appendix:frobenius}.
\end{proof}

\subsubsection{Global analysis}
\label{sec:fu:glo}

We now combine the local solutions of~\eqref{holsys} near $0$ and near $1$ into global pairs $(F^+,F^-)$ holomorphic on $\D\cup\{1\}$. The key step is analytic continuation along the segment $(0,1)$. Solutions of~\eqref{holsys} near $0$ and near $1$ extend to overlap on this segment, and the requirement that they agree there reduces to a $4\times 4$ homogeneous linear system whose matrix we denote by $\Mcrit$. Proposition~\ref{prop:eigte} then reduces to finding a zero of $\det\Mcritnov_{\cdot,\te}$ in $\Ilambd$ for the finitely many $\te_p = \frac{h_p}p$ with $p < N = \getvar{prime_split_N}$ prime and for the interval $J$ of Lemma~\ref{lem:cover}, which is verified with computer assistance in Lemma~\ref{lem:negposjump}.

The following lemma shows that, in~\eqref{holsys}, the values $F^+(0)$ and $F^-(0)$ may be replaced by arbitrary constants, provided we only require $F^+$ and $F^-$ to be holomorphic at $0$.
\begin{lemma}
\label{lem:hol_impl}
    Let $\te^+\in(0,1)$, $\la\in\C$ and $\te^-:=1-\te^+$.
    Let $a^+,a^-\in\C$. Let $F^+,F^-,G^+,G^-\in\XHholo$.
    
    Assume that the following identities hold
    \begin{align*}
        \OLdisk
        \la{\te^+} F^+(z) &=
            z \te^+ (2-\te^+) 
            a^-
            +
            (1-\te^+) (1+\te^+)
            a^+ + z G^+(z)\\
        %%%%%%%%%%%
        \OLdisk
        \la{\te^-} F^-(z) &=
            z \te^- (2-\te^-)
            a^+
            +
            (1-\te^-) (1+\te^-)
            a^- + z G^-(z).
    \end{align*}
    
    Then we have
    \begin{align*}
        F^+(0)=a^+\andspace
        F^-(0)=a^-.
    \end{align*}
\end{lemma}
\begin{proof}
    We evaluate the system at $z=0$, yielding 
    \begin{align*}
        Q_{\la,\te^\pm}(0) F^\pm(0) &=
            (1-\te^\pm) (1+\te^\pm)
            a^\pm.
    \end{align*}

    Since $\te^\pm\notin\Z$, we find that $F^\pm(0) = a^\pm$.
\end{proof}

We now present the criterion we use to find eigenvalues. This result is purely theoretical, but its application will still require computer assistance.
We introduce the affine terms
\begin{align*}
    \Affterm 1 \te(z) :=
     \te(2-\te)z \andspace
    \Affterm 2 \te(z) :=
    (1-\te)(1+\te).
\end{align*}

The following definition introduces the notation needed for the next proposition.
\begin{defi}\label{def:Mat}
    Let $\te^+\in(0,1)$, $\la\in\C$, $\te^-:=1-\te^+$, and set the midpoint $x_0:=\frac12$. For $a^+,a^-\in\C$, consider the system
    \begin{equation}
        \label{holsysa}
        \left.
        \begin{aligned}
            \OLdisk\la{\te^+} F^+(z) =
                a^- \Affterm1{\te^+}(z)
                +
                a^+ \Affterm2{\te^+}(z),
        \\
            \OLdisk\la{\te^-} F^-(z) =
                a^+ \Affterm1{\te^-}(z)
                +
                a^- \Affterm2{\te^-}(z).
        \end{aligned}
        \right\}
    \end{equation}

    By Corollary~\ref{cor:Fp0}, the holomorphic solutions of~\eqref{holsysa} near $0$ are
    \begin{align*}
        \Fholhol0+ =
        a^- \FholpAff 1 {\te^+} 0
        +
        a^+ \FholpAff 2 {\te^+} 0
        \andspace
        \Fholhol0- =
        a^+ \FholpAff 1 {\te^-} 0
        +
        a^- \FholpAff 2 {\te^-} 0.
    \end{align*}
    By Lemma~\ref{lem:loc:aff1}, the holomorphic solutions of~\eqref{holsysa} near $1$ depend on two additional constants $C^+,C^-\in\C$ and read
    \begin{align*}
        \Fholhol{1}+ &=
        a^- \FholpAff 1 {\te^+} {1}
        +
        a^+ \FholpAff 2 {\te^+} {1}
        + C^+ \Fholhte1{1}+,
        \\
        \Fholhol{1}- &=
        a^+ \FholpAff 1 {\te^-} {1}
        +
        a^- \FholpAff 2 {\te^-} {1}
        + C^- \Fholhte1{1}-.
    \end{align*}
    By analytic continuation along the segment $(0,1)$, all four functions extend holomorphically across this segment.
    
    Define the function-valued $2\times 2$ blocks $M_\pm$ and column vectors $h_\pm$ by
    \begin{align*}
        M_+ &:=
        \begin{pmatrix}
            \oldDFp{1}{\te^+} &
            \oldDFp{2}{\te^+}\\
            \partial_z\big(
            \oldDFp{1}{\te^+}\big) &
            \partial_z\big(
            \oldDFp{2}{\te^+}\big)
        \end{pmatrix},
        &
        h_+ &:=
        \begin{pmatrix}
            \Fholhte1{1}+\\
            \pa_z\Fholhte1{1}+
        \end{pmatrix},
        \\
        M_- &:=
        \begin{pmatrix}
            \oldDFp{2}{\te^-} &
            \oldDFp{1}{\te^-}\\
            \partial_z\big(
            \oldDFp{2}{\te^-}\big) &
            \partial_z\big(
            \oldDFp{1}{\te^-}\big)
        \end{pmatrix},
        &
        h_- &:=
        \begin{pmatrix}
            \Fholhte1{1}-\\
            \pa_z\Fholhte1{1}-
        \end{pmatrix}.
    \end{align*}
    Finally, set
    \begin{align*}
        \Mcrit :=
        \begin{pmatrix}
            h_+ & 0 & M_+\\
            0 & h_- & M_-
        \end{pmatrix}\!(x_0) \in \C^{4\times 4},
    \end{align*}
    where the $0$'s denote $2\times 1$ zero blocks.
\end{defi}

The function $\detcrit$ is the determinant of $\Mcrit$. The following Proposition proves that studying the zeros of $\det \Mcrit$ is indeed a good criterion to see if $\la\notin\Z$ is an eigenvalue. In addition, requiring $\Re\la>\frac12$ also allows us to characterize the resolvent set by the nonvanishing of $\detcrit$. In our application, the zero of $\detcrit(\cdot,\te)$ will moreover be real (Lemma~\ref{lem:lareal}).

\begin{prop}
\label{prop:matrix}
    Let $\te\in(0,1)$ and $\la\in\C-\Z$.

    \emph{Homogeneous problem.}
    There exists a non-trivial pair $(F^+,F^-)\in\XHholpair{s}$ for all $s<\frac52+\Re\la$ of holomorphic functions on
    $\D\cup\{1\}$ solving~\eqref{holsys} if and only if $\det\Mcrit = 0$.
    The space of such pairs has dimension $\dim\ker\Mcrit$.

    \emph{Forced problem.}
    Assume additionally $\Re\la>\frac12$ and let $s\in[1,2]$.
    For every $(G^+,G^-)\in\XHholpair{s}$, there exists a unique pair $(F^+,F^-)\in\XHholpair{s+1}$ solving~\eqref{holsys:forced} if and only if $\det\Mcrit\neq 0$.
    Moreover, $\norm{(F^+,F^-)}_{\XHholpair{s+1}}\lesssim\norm{(G^+,G^-)}_{\XHholpair{s}}$.
\end{prop}
\begin{proof}
    Both $\Fholhol{1}\pm$ and $\Fholhol0\pm$ solve the same second-order ODE on $(0,1)$. By uniqueness of the Cauchy problem at a regular point, $\Fholhol{1}+$ and $\Fholhol0+$ coincide on $(0,1)$ if and only if
    \begin{align*}
        \Fholhol{1}+(x_0) = \Fholhol0+(x_0)
        \andspace
        \pa_z\Fholhol{1}+(x_0) = \pa_z\Fholhol0+(x_0),
    \end{align*}
    and analogously for $\Fholhol{1}{-}$ and $\Fholhol{0}{-}$. These four scalar equations form a homogeneous linear system in 
    $(C^+,C^-,a^-,a^+)\in\C^4$ whose matrix is precisely 
    \begin{align*}
        \Mcrit,
    \end{align*}
    and a non-trivial solution exists if and only if 
    \begin{align*}
        \det\Mcrit = 0.
    \end{align*}

    Assume now that $\det\Mcrit=0$, and let $(C^+,C^-,a^-,a^+)$ be a non-trivial vector in the kernel. The functions $\Fholhol0\pm$ and $\Fholhol{1}\pm$ then agree on $(0,1)$ and define a single holomorphic function $F^\pm$ on a neighborhood of $[0,1]$. Since $\D\cup\{1\}$ is simply connected and the only interior singular point of~\eqref{def:holo}, namely $z=0$, is one where $F^\pm$ is holomorphic by Corollary~\ref{cor:Fp0}, $F^\pm$ extends uniquely by analytic continuation to all of $\D\cup\{1\}$. Moreover, $(F^+,F^-)$ is non-trivial: if $(F^+,F^-)=0$, then applying $\OLdisk\la{\te^+}$ to $\Fholhol0+=0$ yields 
    \begin{align*}
        a^-\Affterm1{\te^+}+a^+\Affterm2{\te^+}=0.
    \end{align*}
    Since $\Affterm1{\te^+}$ and $\Affterm2{\te^+}(z)$ are linearly independent and nonzero for $\te^+\in(0,1)$, this forces $a^\pm=0$. Therefore, 
    \begin{align*}
        \Fholhol{1}\pm=C^\pm\Fholhte1{1}\pm=0,
    \end{align*}
    which forces $C^\pm=0$, contradicting non-triviality of $(C^+,C^-,a^-,a^+)$.

    Conversely, given a non-trivial pair $(F^+,F^-)$ holomorphic on $\D\cup\{1\}$ solving~\eqref{holsys}, set $a^\pm:=F^\pm(0)$ (consistent with~\eqref{holsysa} by Lemma~\ref{lem:hol_impl}). By Corollary~\ref{cor:Fp0}, $F^\pm$ coincides with $\Fholhol0\pm$ near $0$. By Lemma~\ref{lem:loc:aff1}, it coincides with $\Fholhol{1}\pm$ near $1$ for some $C^\pm\in\C$. The matching at $x_0$ yields
    \begin{align*}
        \Mcrit
        \begin{pmatrix}
            C^+ \\ C^- \\ a^- \\ a^+
        \end{pmatrix} = 0,
    \end{align*}
    and the vector is non-trivial since otherwise $F^\pm$ would vanish near $0$ and hence everywhere by analytic continuation. Therefore $\det\Mcrit=0$.

    The forced part ($\Leftarrow$) is Lemma~\ref{lem:forced-H1} below and uses the estimate in Lemma~\ref{lem:loc:apriori}, proven in the appendix.
    For the converse ($\Rightarrow$), if $\det\Mcrit=0$, the homogeneous part gives a non-trivial pair $(F_0^+,F_0^-)\neq(0,0)$ solving~\eqref{holsys:forced} with $G^+=G^-=0$. Then both $(0,0)$ and $(F_0^+,F_0^-)$ solve the forced system for $G=0$, so uniqueness fails.
\end{proof}

\begin{lemma}
\label{lem:loc:apriori}
    Let $\te\in(0,1)$, $\la\in\C-\Z$ with $\Re\la>\frac12$, $s\in[1,2]$, and $\eps\in(0,1)$.
    There exists $C=C_{\la,\te,s,\eps}>0$ such that for every $F\in\XHhol{s+1}$
    satisfying $\OLdisk\la\te F=G$ with $G\in\XHhol{s}$,
    \begin{align*}
        \norm{F}_{\XHhol{s+1}}
        \leq C\bigl(
        \norm{G}_{\XHhol{s}}+\norm{F}_{\XHhol{s+1-\eps}}
        \bigr).
    \end{align*}
\end{lemma}
\begin{proof}
    See Appendix~\ref{sec:appendix:apriori}.
\end{proof}

\begin{lemma}
\label{lem:forced-H1}
    Let $\te\in(0,1)$, $\la\in\C-\Z$ with $\Re\la>\frac12$, $s\in[1,2]$, and
    $\det\Mcrit\neq0$.
    For every $(G^+,G^-)\in\XHholpair{s}$, there exists a unique
    pair $(F^+,F^-)\in\XHholpair{s+1}$ solving~\eqref{holsys:forced}.
    
    Moreover, $\norm{(F^+,F^-)}_{\XHholpair{s+1}} \lesssim \norm{(G^+,G^-)}_{\XHholpair{s}}$.
\end{lemma}
\begin{proof}
    \textbf{A priori estimate.}
    Apply Lemma~\ref{lem:loc:apriori} (with the given $s\in[1,2]$) to each scalar
    equation with $\eps\in(0,1)$ fixed.
    The coupling terms $F^\pm(0)$ are bounded by point evaluation, continuous on
    $\XHhol{s+1-\eps}$ since $s+1-\eps>\frac12$.
    Adding the two inequalities gives
    \begin{align}
    \label{eq:apriori-coupled}
        \norm{(F^+,F^-)}_{\XHholpair{s+1}}
        \leq C\bigl(
            \norm{(G^+,G^-)}_{\XHholpair{s}}
            +\norm{(F^+,F^-)}_{\XHholpair{s+1-\eps}}
        \bigr).
    \end{align}

    Suppose the stated estimate $\norm{(F^+,F^-)}_{\XHholpair{s+1}} \lesssim \norm{(G^+,G^-)}_{\XHholpair{s}}$ fails.
    Then there exist solutions $(F_j^+,F_j^-)$ with
    \begin{align*}
        \norm{(F_j^+,F_j^-)}_{\XHholpair{s+1}}=1
        \andspace
        \norm{(G_j^+,G_j^-)}_{\XHholpair{s}}\to0.
    \end{align*}
    The compact embedding $\XHholpair{s+1}\Subset\XHholpair{s+1-\eps}$ yields,
    after passing to a subsequence, strong convergence
    $(F_j^+,F_j^-)\to(F^+,F^-)$ in $\XHholpair{s+1-\eps}$.
    The limit lies in $\XHholpair{s+1}$ and solves the homogeneous
    system~\eqref{holsys}.

    Since $\Re\la>\frac12\geq\frac32-s$ (as $s\ge1$), the branch
    $(z-1)^{2-\la}$ lies in $\XHhol{t}$ only for $t<\frac52-\Re\la<2$, so
    $(z-1)^{2-\la}\notin\XHhol{s+1}$.
    Hence $(F^\pm)$ carry no contribution from this mode and extend
    holomorphically to $z=1$.
    By the homogeneous part of Proposition~\ref{prop:matrix}, the condition
    $\det\Mcrit\neq0$ forces $(F^+,F^-)=0$.
    Therefore $(F_j^+,F_j^-)\to0$ in $\XHholpair{s+1-\eps}$,
    and~\eqref{eq:apriori-coupled} gives a
    contradiction.

    \textbf{Existence.}
    If $(G^+,G^-)$ is polynomial then $zG^\pm$ is analytic near $z=1$.
    Lemmas~\ref{lem:loc:aff0} and~\ref{lem:loc:aff1} then provide holomorphic
    local solutions near $z=0$ and $z=1$, and the matching system at $x_0$ is
    uniquely solvable since $\det\Mcrit\neq0$, giving a global pair in
    $\XHholpair{s+1}$.
    For general $(G^+,G^-)$, approximate in $\XHholpair{s}$ by polynomials.
    The a priori estimate shows the corresponding solutions are Cauchy in
    $\XHholpair{s+1}$ so their limit solves the system.

    \textbf{Uniqueness} follows from the a priori estimate with $G^+=G^-=0$.
\end{proof}

\begin{corollary}
\label{coro:spec}
    Let $\te\in(0,1)$, $\la\in\C-\Z$ with $\Re\la>\frac12$.
    \begin{enumerate}
        \item\label{coro:spec:it1} If $\det\Mcrit\neq 0$, then
              $\la\in\rho(L_\te^{\rm fib})$ on $\XH s$ for every $s\in[1,2]$, that is,
              the operator $L_\te^{\rm fib}-\la:\mathcal D_s\to\XH s$ is bijective and its inverse is bounded on $\XH s$, where $\mathcal D_s:=\{\eta\in\XH s:L_\te^{\rm fib}\eta\in\XH s\}$.

        \item\label{coro:spec:it2} If $\det\Mcrit=0$, then $\la$
              is an eigenvalue of $L_\te^{\rm fib}$ on $\XH s$ for every $s<\frac32+\Re\la$. For $s\in[1,2]$, its geometric multiplicity is $\dim\ker\Mcrit$.
    \end{enumerate}
\end{corollary}

\begin{proof}
    \emph{Part~\ref{coro:spec:it1}.}
    Fix any $s\in[1,2]$. For any $g\in\XH s$, Lemma~\ref{lem:red_to_fin} reduces
    solving $(L_\te^{\rm fib}-\la)\eta=g$ in $\XH s$ to finding $(F^+,F^-)\in\XHholpair{s+1}$
    solving~\eqref{holsys:forced}.
    Under $\Re\la>\frac12$ and $\det\Mcrit\neq0$, Lemma~\ref{lem:forced-H1}
    provides a unique such pair with
    $\norm{(F^+,F^-)}_{\XHholpair{s+1}}\lesssim\norm{(G^+,G^-)}_{\XHholpair{s}}
    \lesssim\norm{g}_{\XH s}$.
    Translating back via Lemma~\ref{lem:red_to_fin} gives a unique $\eta\in\XH s$
    satisfying $(L_\te^{\rm fib}-\la)\eta=g$ and $\norm{\eta}_{\XH s}\lesssim\norm{g}_{\XH s}$.

Since $L_\te^{\rm fib}\eta=\la\eta+g\in\XH s$, we have $\eta\in\mathcal D_s$ so $(L_\te^{\rm fib}-\la)^{-1}$ is bounded on $\XH s$.

\emph{Part~\ref{coro:spec:it2}.}
If $\det\Mcrit=0$, the homogeneous part of Proposition~\ref{prop:matrix} gives a nonzero space of solution pairs of dimension $\dim\ker\Mcrit$, whose members belong to $\XHholpair{r}$ for every $r<\frac52+\Re\la$.
By Lemma~\ref{lem:red_to_fin}, the reconstruction $\eta:=\pa_xf+i\te^+f$ satisfies $(L_\te^{\rm fib}-\la)\eta=0$ and belongs to $\XH s$ for every $s<\frac32+\Re\la$.
For $s<0$, the regularity follows from $\eta\in\XH0\subset\XH{s}$.

For $s\in[1,2]$, Lemma~\ref{lem:red_to_fin} associates every eigenfunction with a pair in $\XHholpair{s+1}$. Since $s+1\geq2>\frac52-\Re\la$, this regularity excludes the branch $(z-1)^{2-\la}$. The pair is therefore holomorphic at $z=1$, and Proposition~\ref{prop:matrix} identifies the eigenspace with $\ker\Mcrit$.
\end{proof}

The following lemma reduces the analysis of $\det\Mcrit$ to real values of $\la$.
\begin{lemma}
\label{lem:lareal}
    Let $\te\in(0,1)$ and $\la\in\RR-\Z$. Then $\det\Mcrit\in\RR$.
\end{lemma}
\begin{proof}
    Fix $\te\in(0,1)$ and $\la\in\RR-\Z$. The polynomial coefficients $\Pglo, \Qglo$ of~\eqref{def:PQtl} and the affine right-hand sides $\Affterm 1{\te^\pm}, \Affterm 2{\te^\pm}$ have real coefficients in $z$. Therefore the Frobenius recurrences at the regular singular points $z=0$ and $z=1$ (Lemmas~\ref{lem:loc:0} and~\ref{lem:loc:1}) are real, and the local series $\phi_1, \phi_2$ together with the particular solutions $\FholpAff i{\te^\pm}{k}$ supplied by Lemmas~\ref{lem:loc:aff0} and~\ref{lem:loc:aff1} all have real Taylor coefficients at their respective base points. In particular each takes real values on a real interval around its base point.

    Analytic continuation along the real segment $(0,1)$ preserves real-valuedness, and the same holds for first derivatives. Hence every entry of $\Mcrit$ is real, and $\det\Mcrit\in\RR$.
\end{proof}

The next lemma provides the sign change needed for the intermediate value theorem, and its verification is computer-assisted.
\begin{lemma}
\label{lem:negposjump}
    For $\te=\te_p$, for every prime $p<N$, and for every $\te\in J$,
    \begin{align*}
        \det\Mcrit < 0
        \quad\text{when } \la = \getvar{low_bnd_eig},
        \qquad
        \det\Mcrit > 0
        \quad\text{when } \la = \getvar{upp_bnd_eig}.
    \end{align*}
\end{lemma}
\proofcompassi

\begin{corollary}
\label{cor:detzero}
    For $\te=\te_p$, for every prime $p<N$, and for every $\te\in J$, there
    exists $\la_\te\in\Ilambd$ such
    that $\det\Mcrit = 0$.
\end{corollary}
\begin{proof}
    Fix one such $\te$. The map $\la\mapsto\det\Mcrit$ is continuous on $\cIlambd$, since by standard Frobenius theory~\cite[Theorem~4.2]{TeschlODE} the entries of $\Mcrit$ depend analytically on $(\la,\te)\in(\C-\Z)\times(0,1)$.

    By Lemma~\ref{lem:lareal} this map is real-valued, and by Lemma~\ref{lem:negposjump} it changes sign between the endpoints of $\cIlambd$, so the intermediate value theorem yields a zero in the interior.
\end{proof}

\begin{rem}
    Lemma~\ref{lem:negposjump} also certifies that $\detcrit(\cdot,\te)\neq 0$. At
    $\la=\getvar{upp_bnd_eig}$ one has $\detcrit>0$, so $\detcrit$ is not the zero function.
    Since $\detcrit(\cdot,\te)$ is analytic on $\C-\Z$ by the Frobenius theory cited above,
    the Identity Theorem implies that its zero set is discrete in $\C-\Z$.
    For each $s\in[1,2]$, Corollary~\ref{coro:spec} identifies the zeros in $\{\la\in\C-\Z:\Re\la>\frac12\}$ with the eigenvalues of $L_\te^{\rm fib}$ on $\XH s$ in that set. All these eigenvalues are therefore isolated on $\XH s$.
    See Figure~\ref{fig:crit_curves} for a numerical plot of $\detcrit(\cdot,\te)$ at the two endpoint values $\te=\frac13$ and $\te=\frac12$.
    For the endpoint values $\theta=\frac13$ and $\theta=\frac12$, the corresponding numerical eigenfunctions are shown in Figure~\ref{fig:eta-theta-n}.
\end{rem}

\begin{figure}[htbp]
    \centering
    \includegraphics[width=9.5cm]{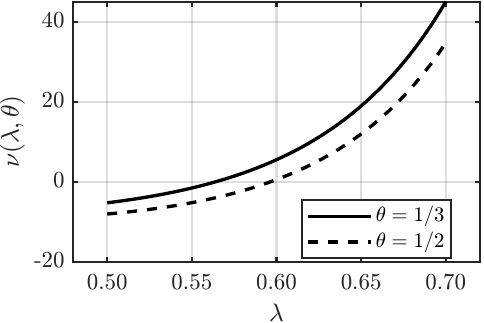}
    \caption{Numerical evaluation of $\detcrit(\cdot,\te)$ for $\te=\frac13$ and $\te=\frac12$. The zeros correspond to eigenvalues of $L_\te^{\rm fib}$.}
    \label{fig:crit_curves}
\end{figure}

\begin{figure}[htbp]
    \centering
    \includegraphics[width=13.5cm]{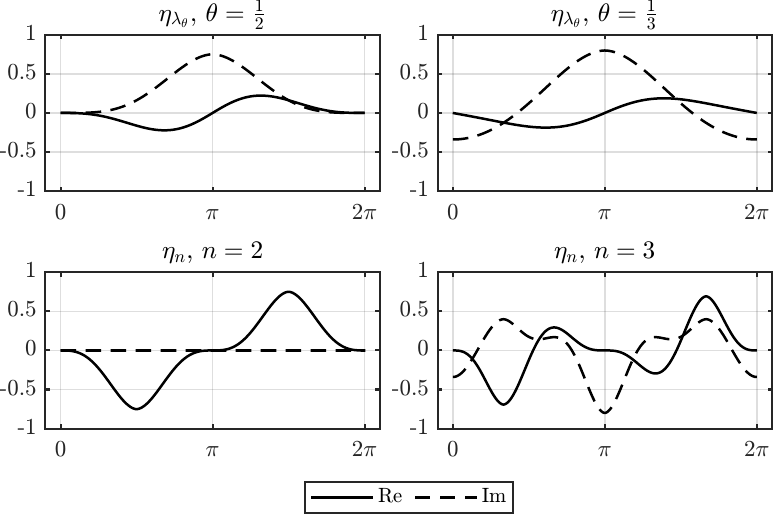}
    \caption{
    Numerical plots of the eigenfunctions obtained from the matching construction
    for $\theta=\frac12$ and $\theta=\frac13$, together with their lifted
    counterparts for $n=2$ and $n=3$. 
    }
    \label{fig:eta-theta-n}
\end{figure}

\subsection{Conclusion of the proof}
Fix $\te=\te_p$ for some prime $p<N$, or $\te\in J$. Corollary~\ref{cor:detzero} gives $\la_\te\in\Ilambd$ such that $\det\Mcrit=0$ at $(\la_\te,\te)$. By Proposition~\ref{prop:matrix}, there is then a non-trivial pair $(F^+_\te,F^-_\te)$ of holomorphic functions on $\D\cup\{1\}$ solving~\eqref{holsys} at $(\la_\te,\te)$.

The only point of $\overline{\D}$ where $F^\pm_\te$ may fail to be holomorphic is $z=-1$. Indeed, the right-hand sides of~\eqref{holsys} are analytic at $z=-1$, so local Frobenius theory shows that, near $z=-1$, $F^\pm_\te$ is of the form
\begin{align*}
    F^\pm_\te(z)=\psi^\pm(z)+\beta^\pm(z+1)^{2+\la_\te}\phi_2^\pm(z),
\end{align*}
where $\psi^\pm$ and $\phi_2^\pm$ are analytic at $-1$ and $\beta^\pm\in\C$.

Restricting to $\T$ via $z=e^{ix}$, the branch $(e^{ix}+1)^{2+\la_\te}\phi_2^\pm(e^{ix})$ behaves like $|x-\pi|^{2+\la_\te}$ near $x=\pi$. Its Fourier coefficients therefore decay like $|k|^{-3-\la_\te}$. Hence the traces $F^\pm_\te|_\T$ lie in $\XH{s}$ for every $s<\frac52+\la_\te$, equivalently $(F^+_\te,F^-_\te)\in\XHholpair{s}$ for the same range of $s$.

By Lemma~\ref{lem:red_to_fin}, this yields an eigenfunction $\eta_\te$ of $L_\te^{\rm fib}$ with eigenvalue $\la_\te$ such that $\eta_\te\in\XH{s}$ for every $s<\frac32+\la_\te$. Since $\la_\te>\getvar{low_bnd_eig}>\frac12$, it follows in particular that $\eta_\te\in\XH2$. This completes the proof of Proposition~\ref{prop:eigte}.

\section{The $H^1$ energy estimates: Lemma~\ref{lem:estH1}}

We prove Lemma~\ref{lem:estH1} by estimating the linear evolution on every fiber of $L_p$. The estimate is needed on the whole zero-mean space, so we first extend the decomposition of Lemma~\ref{lem:unfold} to include the $0$-fiber.

\subsection{The fiber decomposition}
\label{sec:fibandgen}

Lemma~\ref{lem:unfold} was stated for functions with $\ha\eta_{nk}=0$ and only for the fibers $1\leq l\leq n-1$.
We now include the $0$-fiber. The zero-mean condition used below is preserved by the nonlinear flow, as shown at the start of Section~\ref{sec:nonlinear}.
 
Extend the definition of $\Pcal_{l,n}$ to $l=0$ by the same formula,
\begin{align*}
    (\ha{\Pcal_{0,n}\eta})_k:=\ha\eta_{nk},
\end{align*}
Hence, for every $\eta\in\XL$,
\begin{align}
\label{eqn:planch:L2}
    \norm\eta_\XL^2=\sum_{l=0}^{n-1}\norm{\Pcal_{l,n}\eta}_\XL^2.
\end{align}
Therefore, we can decompose
\begin{align*}
\XL\equiv\bigoplus_{l=0}^{n-1}\Pcal_{l,n}\XL.
\end{align*}
Moreover $\ha\eta_0=(\ha{\Pcal_{0,n}\eta})_0$, so $\eta$ is zero-mean if and only if its $0$-fiber is.

\begin{lemma}
\label{lem:unfold:all}
    Let $n\geq2$ and let $\eta\in\XHZmean1$. Then $\eta$ satisfies~\eqref{eqn:linf} if and only if, for every $0\leq l\leq n-1$,
    \begin{align*}
        \pa_t\Pcal_{l,n}\eta=L_{l/n}^{\rm fib}\Pcal_{l,n}\eta.
    \end{align*}
\end{lemma}

\begin{proof}
    The computation is the same as in the proof of Lemma~\ref{lem:unfold}, but we also allow $\te=0$ in~\eqref{eqn:unfold:aux}. Although $\ti A_{0,0}$ and $\ti B_{0,0}$ are not defined there, the zero-mean condition on $\eta$ makes those terms vanish.
\end{proof}

\begin{rem}
	Notice that
	\begin{align*}
		\Pcal_{0,n} L_n \foba{kn} = L_1 \foba{k} = L_0^{\rm fib} \foba{k},\qquad k\neq0.
	\end{align*}
	Therefore, the analysis of $L_1$ allows us to study the $0$-fiber for every $n\geq2$, at the linear level.
\end{rem}

Fix a prime $p$ as in Lemma~\ref{lem:estH1}. Denote by
\begin{align*}
    \Gamma_{1/p}:=\Big\{\frac hp : h=0,1,\dots,p-1\Big\}
\end{align*}
the set of fibers of $L_p$. For every $\te\in\Gamma_{1/p}$ there exists $0\leq h\leq p-1$ such that $\te=\tfrac hp$, and $\Pcal_{h,p}\eta$ is the $\te$-fiber of $\eta$. By Lemma~\ref{lem:unfold:all}, we have
\begin{align*}
    \Pcal_{h,p}(L_p\eta) = L_\te^{\rm fib}\Pcal_{h,p}\eta,
    \qquad \te=\tfrac hp.
\end{align*}

\begin{rem}
\label{rem:JSS19:periodic}
    Since the $0$-fiber is invariant under $Q$, the results studied in~\cite{JSS19} extend to the excited state $-\sin(nx)$ with $n\geq2$ if we restrict the space to $\frac{2\pi}n$-periodic functions.
\end{rem}

We now define the weights that measure the fiber norms. We will only use $s=1$, although the following identities hold for every $s>0$. Let $\te\in[0,1)$ and define the weights $D_{s,\te}$ and $\Omega_{s,\te}$ by
\begin{align}
\label{def:W1te}
    (\ha{D_{s,\te}})_k:=\sqrt{1+|k+\te|^{2s}}
    \andspace
    (\ha{\Omega_{s,\te}})_k:=|k+\te|^s,
    \qquad \text{for all }k\in\Z,
\end{align}
The symbol of $D_{s,\te}$ never vanishes, so $D_{s,\te}$ is always invertible.
The symbol of $\Omega_{s,\te}$ vanishes exactly when $k+\te=0$, that is at the mode $k=0$ of the $0$-fiber, which is excluded by the zero-mean condition.
The introduction of $\Omega_{s,\te}$ is motivated by the following identity,
\begin{align*}
    \sum_{l=0}^{p-1}
    \norm{\Omega_{s,l/p} \Pcal_{l,p}\eta}_\XL^2
    &=
    \sum_{l=0}^{p-1}
    \sum_{k\in\Z}
    \big|(\ha{\Omega_{s,l/p}})_k (\ha{\Pcal_{l,p}\eta})_k\big|^2
    =
    \sum_{l=0}^{p-1}
    \sum_{k\in\Z}
    |k+l/p|^{2s}
    |\ha\eta_{pk+l}|^2
    \\
    &=
    p^{-2s}
    \sum_{k\in\Z}
    |k|^{2s}
    |\ha\eta_{k}|^2
    =
    \frac{1}{p^{2s}}
    \norm{\eta}_\XHZmean{s}^2
\end{align*}
The operator $D_{s,\te}$ is used for technical reasons. Notice that, since
\begin{align*}
    \sum_{l=0}^{p-1}
    \norm{D_{s,l/p} \Pcal_{l,p}\eta}_\XL^2
    =
    \norm{\eta}_\XL^2 +
    \frac{1}{p^{2s}}\norm{\eta}_\XHZmean{s}^2,
\end{align*}
we consequently obtain
\begin{align}
\label{eqn:planch:equiv}
    \frac{1}{p^{2s}}
    \norm{\eta}_\XH{s}^2
    \leq
    \sum_{l=0}^{p-1}
    \norm{D_{s,l/p} \Pcal_{l,p}\eta}_\XL^2
    \leq
    \norm{\eta}_\XH{s}^2.
\end{align}
Moreover, conjugating $L_\te^{\rm fib}$ by $D_{s,\te}$ gives
\begin{align*}
    \LfibD_{s,\te}:=D_{s,\te}L_\te^{\rm fib} D_{s,\te}^{-1}.
\end{align*}
By~\eqref{eqn:planch:equiv}, the $\XH s$ norm of $\eta$ is equivalent to the sum over its fibers of the $\XL$ norms of $D_{s,l/p}\Pcal_{l,p}\eta$, so $\XH s$ estimates for $L_\te^{\rm fib}$ can be obtained by bounding $\LfibD_{s,\te}$ in $\XL$ on each fiber.

For $s=1$, we first compute the quadratic forms for smooth periodic functions. Their Hermitian parts, and the weighted Hermitian parts used below, extend to bounded forms on $\XL$. This follows from the Jacobi coefficients below and the tail bounds proved in this section. At $\te=0$ we restrict to zero-mean functions.

\begin{rem}[The Jacobi form of $\LfibD_{1,\te}$]
\label{rem:tm:jacobi}
    Write $\Scal_{1,\te}:=\frac12\big(\LfibD_{1,\te}+(\LfibD_{1,\te})^\ast\big)$ for the Hermitian part of $\LfibD_{1,\te}$. Since
    \begin{align*}
        \LfibD_{1,\te}\foba m=\ti A_{m,\te}\,\frac{(\ha D_{1,\te})_{m-1}}{(\ha D_{1,\te})_m}\foba{m-1}+\ti B_{m,\te}\,\frac{(\ha D_{1,\te})_{m+1}}{(\ha D_{1,\te})_m}\foba{m+1},
    \end{align*}
    expanding $\Re(\LfibD_{1,\te}\eta,\eta)_\XL$ against $(\foba j,\foba m)=\de_{jm}$ gives, for every $\eta\in\XL$,
    \begin{align}
    \label{eq:H1:jacobi:sum}
        \Re(\LfibD_{1,\te}\eta,\eta)_\XL
        =
        (\Scal_{1,\te}\eta,\eta)_\XL
        =
        \sum_{m\in\Z} t_{m,1,\te}\,\Re\big(\ha\eta_{m+1}\overline{\ha\eta_m}\big),
    \end{align}
    where $t_{m,1,\te}$ is given by
    \begin{align*}
        t_{m,1,\te}:=\ti A_{m+1,\te}\,\frac{(\ha D_{1,\te})_m}{(\ha D_{1,\te})_{m+1}}+\ti B_{m,\te}\,\frac{(\ha D_{1,\te})_{m+1}}{(\ha D_{1,\te})_m}.
    \end{align*}
    At $\te=0$, $\ti A_{0,0}$ and $\ti B_{0,0}$ are not defined. However, this is precisely the case where we work with zero-mean functions, $\ha\eta_0=0$. Therefore~\eqref{eq:H1:jacobi:sum} is given by
    \begin{align*}
        \Re(\LfibD_{1,0}\eta,\eta)_\XL
        =
        \sum_{m\neq-1,0} t_{m,1,0}\,\Re\big(\ha\eta_{m+1}\overline{\ha\eta_m}\big).
    \end{align*}
\end{rem}

\subsection{The $H^1$ estimate: Proof of Lemma~\ref{lem:estH1}}
\label{sec:estH1}

The proof relies on the following lemma.

\begin{lemma}
\label{lem:fib:bilform}
    Let $\om=1.1$. For every $\te\in[0,1)$ there is a Hermitian form $B_\te$ equivalent to the usual norm on $\XL$ such that, for every $v\in\XL$,
    \begin{align*}
        \Re B_\te(\LfibD_{1,\te}v,v)\leq\om B_\te(v,v).
    \end{align*}
\end{lemma}

\begin{proof}
    Section~\ref{sec:fiberest:h1} is devoted to proving Lemma~\ref{lem:fib:bilform}.
\end{proof}

We now prove Lemma~\ref{lem:estH1}.
Let $p$ be prime and let $\om=1.1$. Recall that $v$ has zero mean whenever $\te=0$, so every ``for all $v\in\XL$'' below is understood with that restriction. By Lemma~\ref{lem:fib:bilform}, for every $\te\in\Gamma_{1/p}$ there is a Hermitian form $B_\te$ in $\XL$, equivalent to the usual norm, such that
    \begin{align}
    \label{eq:h1:fib:lmi}
        \Re B_\te(\LfibD_{1,\te}v,v)\leq\om B_\te(v,v)
        \qquad\text{for all } v\in\XL,
    \end{align}
    Since $\Gamma_{1/p}$ has $p$ elements, there are constants $0<m_p\leq1\leq M_p$, independent of $\te$, such that
    \begin{align}
    \label{eq:h1:fib:equiv}
        m_p\norm v_\XL^2\leq B_\te(v,v)\leq M_p\norm v_\XL^2
        \qquad\text{for all } v\in\XL,\ \te\in\Gamma_{1/p}.
    \end{align}

    First take smooth initial data of zero mean and let $\eta(t)$ solve~\eqref{eqn:linf} with $n=p$ and define
    \begin{align*}
        E_p(\eta):=
        \sum_{\te\in\Gamma_{1/p}}
        B_\te(
            D_{1,\te}\Pcal_{h,p}\eta,
            D_{1,\te}\Pcal_{h,p}\eta
        ),
        \qquad \te=\tfrac hp.
    \end{align*}

    Applying~\eqref{eq:h1:fib:equiv} to $v=D_{1,\te}\Pcal_{h,p}\eta(t)$ and summing over $\te\in\Gamma_{1/p}$,
    \begin{align*}
        m_p\sum_{\te\in\Gamma_{1/p}}\norm{D_{1,\te}\Pcal_{h,p}\eta(t)}_\XL^2\leq E_p(\eta(t))\leq M_p\sum_{\te\in\Gamma_{1/p}}\norm{D_{1,\te}\Pcal_{h,p}\eta(t)}_\XL^2,
    \end{align*}
    and, by~\eqref{eqn:planch:equiv} with $s=1$,
    \begin{align}
    \label{eq:h1:equiv:new}
        m_pp^{-2}\norm\eta_{\XH1}^2\leq E_p(\eta)\leq M_p\norm\eta_{\XH1}^2.
    \end{align}

    By Lemma~\ref{lem:unfold:all}, each fiber evolves independently,
    \begin{align}
        \label{aux:H1:proof}
        \pa_t\Pcal_{h,p}\eta=L_\te^{\rm fib}\Pcal_{h,p}\eta,
    \end{align}
    and writing $v_\te(t):=D_{1,\te}\Pcal_{h,p}\eta(t)$ in~\eqref{aux:H1:proof}, we get
    \begin{align}
        \label{aux:H1:proof2}
        \pa_tv_\te=D_{1,\te}L_\te^{\rm fib}\Pcal_{h,p}\eta=\LfibD_{1,\te}v_\te.
    \end{align}

    Since $B_\te$ is Hermitian, from~\eqref{aux:H1:proof2}
    \begin{align*}
        \pa_tB_\te(v_\te,v_\te)
        =
        2\Re B_\te(\pa_t v_\te,v_\te)
        =
        2\Re B_\te(\LfibD_{1,\te}v_\te,v_\te)
        \leq
        2\om B_\te(v_\te,v_\te),
    \end{align*}
    where we used~\eqref{eq:h1:fib:lmi} in the last step. Summing again over $\te\in\Gamma_{1/p}$ gives
    \begin{align}
    \label{eq:h1:diff:new}
        \frac12\pa_tE_p(\eta(t))\leq\om E_p(\eta(t)),
    \end{align}
    since
    \begin{align*}
        \pa_t E_p(\eta):=
        2\Re\sum_{\te\in\Gamma_{1/p}}
        B_\te(
            \pa_t v_\te,
            v_\te
        ).
    \end{align*}

    Finally, Gr\"onwall's inequality applied to~\eqref{eq:h1:diff:new}, together with~\eqref{eq:h1:equiv:new},
    \begin{align*}
        m_p p^{-2}
        \norm{\eta(t)}_\XH1^2\leq
        E_p(\eta(t))\leq
        E_p(\eta(0))e^{2t\om}
        \leq
        M_p
        \norm{\eta(0)}_\XH1^2
        e^{2t\om},
    \end{align*}
    which is Lemma~\ref{lem:estH1} with $C_{1,p}=p\sqrt{M_p/m_p}$ for smooth initial data. The result for $\XHZmean1$ data follows by density and the $C_0$ semigroup property established in Remark~\ref{rem:linf:wellposed}.

\subsubsection{Fiber estimate}
\label{sec:fiberest:h1}

This section proves Lemma~\ref{lem:fib:bilform} above. To bound the high modes of $\LfibD_{1,\te}$, we first prove the following tail estimate, uniformly over $\te\in[0,1)$.

\begin{lemma}
\label{lem:tm:tail}
    Let $\te\in[0,1)$ and let $t_{m,1,\te}$ be as in Remark~\ref{rem:tm:jacobi}. For $m\in\Z$ set
    \begin{align*}
        \ka_m:=\min\{|m+\te|,|m+1+\te|\}.
    \end{align*}
    Then, for every $m$ with $\ka_m\geq1$,
    \begin{align}
    \label{ineq:tm:tail}
        |t_{m,1,\te}|<\frac12\Big(1+\frac3{\ka_m^2}\Big).
    \end{align}
    Since the right-hand side decreases in $\ka_m$, for every $N_0\geq1$,
    \begin{align}
    \label{ineq:tm:tail:N0}
        \sup_{\ka_m>N_0}|t_{m,1,\te}|\leq\frac12\Big(1+\frac3{N_0^2}\Big).
    \end{align}
\end{lemma}

\begin{proof}
	Denote $\kappa:=m+\te$ and set $w(\kappa):=\kappa^2 + \kappa + 1$, so that
	\begin{align*}
		\langle\kappa\rangle
		\langle\kappa+1\rangle
		=
		\sqrt{w(\kappa)^2+1}.
	\end{align*}
	Multiplying $t_{m,1,\te}$ by $\langle\ka\rangle\langle\ka+1\rangle$,
    \begin{align}
    	\label{eqn:H1:tail:aux1}
        t_{m,1,\te}\sqrt{w(\ka)^2+1}=
        \ti A_{m+1,\te}(1+\ka^2)+\ti B_{m,\te}\big(1+(\ka+1)^2\big).
    \end{align}

    For $\ka>0$, using the definitions (see~\eqref{def:ABkt}) of $\ti A_{m,\te}$ and $\ti B_{m,\te}$, we have from~\eqref{eqn:H1:tail:aux1}
    \begin{align*}
        t_{m,1,\te}=\frac{\ka^4+2\ka^3+5\ka^2-2}{2\ka(\ka+1)\sqrt{w(\ka)^2+1}}.
    \end{align*}
    For $\ka\geq1$ the numerator is positive and, since $\sqrt{w(\ka)^2+1}>w(\ka)$ ,
    \begin{align*}
        0<t_{m,1,\te}<
        \frac{\ka^4+2\ka^3+5\ka^2-2}{2\ka(\ka+1)w(\ka)}=
        \frac12+\frac{3\ka^2-\ka-2}{2\ka(\ka+1)w(\ka)}<
        \frac12+\frac3{2\ka^2}.
    \end{align*}

    For $\ka=m+\te<-1$, recall Remark~\ref{rem:AkBksym}, the symmetry $k\mapsto-k-1$, $\te\mapsto1-\te$. Applied to $k=m+1$ and to $k=m$, it gives $\ti A_{m+1,\te}=\ti B_{-m-2,1-\te}$ and $\ti B_{m,\te}=\ti A_{-m-1,1-\te}$. Writing $\nu:=-\ka-1=(-m-2)+(1-\te)$, namely~\eqref{eqn:H1:tail:aux1} at $(1-\te,\nu)$ with the roles of $\ti A,\ti B$ exchanged,
    \begin{align*}
        t_{m,1,\te}\sqrt{w(\ka)^2+1}=\ti A_{-m-1,1-\te}\big(1+\nu^2\big)+\ti B_{-m-2,1-\te}\big(1+(\nu+1)^2\big),
    \end{align*}
    and since $w(\ka)=w(\nu)$, $t_{m,1,\te}$ equals the value already found for $\ka>0$, evaluated at $\nu$ in place of $\ka$.

    In either case the bound is evaluated at the positive argument $\ka_m$, which gives~\eqref{ineq:tm:tail}.
\end{proof}

The last lemma gives a direct proof of the case $\te=0$.

\begin{lemma}
\label{lem:H1:the:0}
    Let $\om=1.1$ as above. For every smooth periodic function $v$ with zero mean,
    \begin{align*}
        \Re (\LfibD_{1,0}v,v)_\XL\leq\om \norm{v}_\XL^2.
    \end{align*}
\end{lemma}
\begin{proof}
    Recall Remark~\ref{rem:tm:jacobi}, we have
    \begin{align*}
        \Re(\LfibD_{1,0}v,v)_\XL
        \leq
        \sup_{m\neq-1,0}|t_{m,1,0}|\,\norm v_\XL^2.
    \end{align*}
    By Lemma~\ref{lem:tm:tail} with $N_0=3$,
    \begin{align*}
        |t_{m,1,0}|<\frac23\qquad\text{for } m\geq4\text{ or }m\leq-5.
    \end{align*}
    Finally, since $t_{-m-1,1,0}=t_{m,1,0}$, we just have to compute the following cases
    \begin{align*}
        t_{1,1,0}=\frac{3\sqrt{10}}{20},
        \qquad
        t_{2,1,0}=\frac{5\sqrt2}{12},
        \qquad
        t_{3,1,0}=\frac{89\sqrt{170}}{2040},
    \end{align*}
    each below $\frac23$. Hence $\Re(\LfibD_{1,0}v,v)_\XL\leq\frac23\norm v_\XL^2\leq\om\norm v_\XL^2$.
\end{proof}

In order to prove Lemma~\ref{lem:fib:bilform} for $\te\in(0,1)$, the following lemma will allow us to reduce the study to $\te\in\big(0,\tfrac12\big]$.

\begin{lemma}
\label{lem:tm:symmetry}
    Let $\te\in(0,1)$, and let $\mathcal R$ be given by $(\ha{\mathcal R\eta})_k:=\ha\eta_{-k-1}$. Then
    \begin{align*}
        \LfibD_{1,1-\te}=\mathcal R\LfibD_{1,\te}\mathcal R.
    \end{align*}
    Consequently, if $B_\te$ is a Hermitian form on $\XL$, equivalent to the usual norm, and $r>0$ is such that, for every $v\in\XL$,
    \begin{align*}
        \Re B_\te(\LfibD_{1,\te}v,v)\leq rB_\te(v,v),
    \end{align*}
    then $B_{1-\te}(v,w):=B_\te(\mathcal Rv,\mathcal Rw)$ satisfies the same inequality.
\end{lemma}

\begin{proof}
	Notice that $(\ha D_{1,\te})_k= (\ha D_{1,1-\te})_{-k-1}$ for all $k\in\Z$. Moreover, recall Remark~\ref{rem:AkBksym}, where we also stated that $\ti A_{-k-1,\te}=\ti B_{k,1-\te}$. Using both identities,
    \begin{align*}
        \mathcal R\LfibD_{1,\te}\mathcal R\foba k
        =\mathcal R\LfibD_{1,\te}\foba{-k-1}
        &=\ti A_{-k-1,\te}\,\frac{(\ha D_{1,\te})_{-k-2}}{(\ha D_{1,\te})_{-k-1}}\foba{k+1}+\ti B_{-k-1,\te}\,\frac{(\ha D_{1,\te})_{-k}}{(\ha D_{1,\te})_{-k-1}}\foba{k-1}\\
        &=\ti B_{k,1-\te}\,\frac{(\ha D_{1,1-\te})_{k+1}}{(\ha D_{1,1-\te})_k}\foba{k+1}+\ti A_{k,1-\te}\,\frac{(\ha D_{1,1-\te})_{k-1}}{(\ha D_{1,1-\te})_k}\foba{k-1}
        =\LfibD_{1,1-\te}\foba k.
    \end{align*}

    Since $\mathcal R$ is unitary, $B_{1-\te}$ is also equivalent to the usual norm, as $\norm{\mathcal Rv}_\XL=\norm v_\XL$. Finally, using $\mathcal R^2=I$,
    \begin{align*}
        \Re B_{1-\te}(\LfibD_{1,1-\te}v,v)
        =\Re B_\te(\mathcal R^2\LfibD_{1,\te}\mathcal Rv,\mathcal Rv)
        \leq rB_\te(\mathcal Rv,\mathcal Rv)
        =rB_{1-\te}(v,v).
    \end{align*}
\end{proof}

Lemma~\ref{lem:core:lmi:s1} below reduces the inequality of Lemma~\ref{lem:fib:bilform} to a condition that can be checked by computer.

\begin{lemma}
\label{lem:core:lmi:s1}
    Let $T$ be defined on smooth periodic functions, with a Hermitian part that extends to a bounded operator on $\XL$. Assume that $T$ is tridiagonal with zero diagonal in the basis $(\foba k)_{k\in\Z}$, so that
    \begin{align*}
    	(T\foba k,\foba j)_\XL=0\qquad\text{whenever }|j-k|\neq1.
    \end{align*}
    Let $m\geq1$, define $V_m:=\spn\{\foba k:|k|\leq m\}$, and let $\ga\in\RR$ be such that
    \begin{align}
    \label{ineq:core:lmi:s1:ga}
        \big(\tfrac12(T+T^\ast)\eta,\eta\big)_\XL\leq\ga\norm\eta_\XL^2
        \qquad\text{for every }\eta\perp V_m.
    \end{align}
    Fix $r>\ga$. For $j,k\in\Z$ with $|j|\leq m$, set
    \begin{align*}
        M_{jk}:=2r(\foba k,\foba j)_\XL-(T\foba k,\foba j)_\XL-(\foba k,T\foba j)_\XL.
    \end{align*}
    Let
    \begin{align*}
        M^{V_m}:=(M_{jk})_{|j|,|k|\leq m}
        \andspace
        M^{\rm bd}:=(M_{jk})_{|j|\leq m,\,k\in\{m+1,-m-1\}},
    \end{align*}
    matrices of size $(2m+1)\times(2m+1)$ and $(2m+1)\times2$. If, for every $v=(v_1,v_2)\in\C^{2m+1}\times\C^2$,
    \begin{align}
    \label{eqn:core:lmi:psd}
        v^\ast\begin{pmatrix}M^{V_m}&M^{\rm bd}\\(M^{\rm bd})^\ast&2(r-\ga)I_2\end{pmatrix}v\geq0,
    \end{align}
    then, for every $v\in\XL$,
    \begin{align*}
        \Re(Tv,v)_\XL\leq r\norm{v}_\XL^2.
    \end{align*}
    If moreover $G$ is a positive diagonal operator in the basis with $G,G^{-1}$ bounded, then the sesquilinear form $B(u,v):=(Gu,Gv)_\XL$ satisfies, for every $v\in\XL$,
    \begin{align*}
        \Re B(G^{-1}TGv,v)\leq rB(v,v).
    \end{align*}
\end{lemma}

\begin{proof}
    Let $M$ denote the continuous extension to $\XL$ of the sesquilinear form
    \begin{align*}
        M(u,v):=2r(u,v)_\XL-(Tu,v)_\XL-(u,Tv)_\XL,
        \qquad u,v\text{ smooth periodic functions},
    \end{align*}
    so that $M_{jk}=M(\foba k,\foba j)$ recovers the entries above, and $M(v,u)=\overline{M(u,v)}$. Since $M(\eta,\eta)=2r\norm\eta_\XL^2-2\Re(T\eta,\eta)_\XL$, it suffices to show $M(\eta,\eta)\geq0$ for every $\eta\in\XL$.

	Define $W_m:=\overline{\spn}\{\foba k: |k|\geq m+2\}$, and let $v\in V_m$ and $w\in W_m$.
	Compute
	\begin{align}
	\label{eqn:core:lmi:decouple}
		M(w,v) = M(v,w) = 2r(v,w)_\XL-(Tv,w)_\XL-(v,Tw)_\XL = 0.
	\end{align}
	Here, since $T$ shifts indices by one place,
	every term is $0$ because both entries of the inner product are orthogonal.
	
    On $V_m^\perp$, $M$ is coercive. Indeed, for $\ti\eta\perp V_m$,
    \begin{align}
    \label{eqn:core:lmi:coercive}
        M(\ti\eta,\ti\eta)=
        2r\norm{\ti\eta}_\XL^2-\big((T+T^\ast)\ti\eta,\ti\eta\big)_\XL\geq2(r-\ga)\norm{\ti\eta}_\XL^2\geq0.
    \end{align}

    Now fix $\eta\in\XL$ and write its orthogonal decomposition as $\eta=\eta_{V_m}+\eta_{\rm bd}+\eta_{W_m}$. Here $\eta_{V_m}\in V_m$, $\eta_{\rm bd}\in\spn\{\foba{m+1},\foba{-m-1}\}$, and $\eta_{W_m}\in W_m$. We identify the first two components with their coordinate vectors in the bases $(\foba k)_{|k|\leq m}$ and $(\foba{m+1},\foba{-m-1})$, so that $M^{V_m},M^{\rm bd}$ act on them directly. By~\eqref{eqn:core:lmi:decouple}, the cross terms between $\eta_{V_m}$ and $\eta_{\rm bd}+\eta_{W_m}$ reduce to those between $\eta_{V_m}$ and $\eta_{\rm bd}$, and, since $\eta_{\rm bd}+\eta_{W_m}\perp V_m$, the coercivity bound applies to it. Hence
    \begin{align*}
        M(\eta,\eta)
        &=M(\eta_{V_m},\eta_{V_m})+2\Re M(\eta_{\rm bd},\eta_{V_m})+M(\eta_{\rm bd}+\eta_{W_m},\eta_{\rm bd}+\eta_{W_m})\\
        &\geq\underbrace{\eta_{V_m}^\ast M^{V_m}\eta_{V_m}+2\Re\big(\eta_{V_m}^\ast M^{\rm bd}\eta_{\rm bd}\big)+2(r-\ga)\norm{\eta_{\rm bd}}_\XL^2}_{\geq0\text{, \eqref{eqn:core:lmi:psd}}}+\underbrace{2(r-\ga)\norm{\eta_{W_m}}_\XL^2}_{\geq0\text{, \eqref{eqn:core:lmi:coercive}}}\geq0
    \end{align*}

    For the second claim, first take a smooth periodic function $v$. Since $G$ preserves that space, we have
    \begin{align*}
        \Re B(G^{-1}TGv,v)=\Re(TGv,Gv)_\XL\leq r(Gv,Gv)_\XL=rB(v,v).
    \end{align*}
    The inequality extends to $\XL$ by continuity of the forms.
\end{proof}

For $\te$ near $0$ we will use the operator of the following lemma, recovering $L_{1,\te}^{{\rm fib}\,\Sigma}$ by taking $G=\Sigma_{1,\te}$ in Lemma~\ref{lem:core:lmi:s1}.

\begin{lemma}
\label{lem:rescaled:fiber}
    Let $\te\in(0,1)$. Let $\Sigma_{1,\te}$ be the positive diagonal operator with symbol
    \begin{align*}
        (\Sigma_{1,\te})_k:=\frac{|k+\te|}{\sqrt{1+|k+\te|^2}}\qquad(k\neq0),
        \qquad
        (\Sigma_{1,\te})_0:=\frac{\te^{-1}}{\sqrt{1+\te^2}},
    \end{align*}
    set $L_{1,\te}^{{\rm fib}\,\Sigma}:=\Sigma_{1,\te}\LfibD_{1,\te}\Sigma_{1,\te}^{-1}$ and $B_{1,\te}:=\Sigma_{1,\te}^2$. If there are $m\geq1$, $\ga\in\RR$ and $r>\ga$ such that both hypotheses of Lemma~\ref{lem:core:lmi:s1} hold (inequality~\eqref{ineq:core:lmi:s1:ga} and~\eqref{eqn:core:lmi:psd}), we will have for $T:=L_{1,\te}^{{\rm fib}\,\Sigma}$ and $G:=\Sigma_{1,\te}$ that
    \begin{align*}
        \Re(B_{1,\te}\LfibD_{1,\te}v,v)_\XL\leq r(B_{1,\te}v,v)_\XL
        \qquad \text{for every }v\in\XL.
    \end{align*}
\end{lemma}

\begin{proof}
    Apply the second claim of Lemma~\ref{lem:core:lmi:s1} with $G:=\Sigma_{1,\te}$: since $G^{-1}TG=\Sigma_{1,\te}^{-1}L_{1,\te}^{{\rm fib}\,\Sigma}\Sigma_{1,\te}=\LfibD_{1,\te}$ and $B(u,v)=(Gu,Gv)_\XL=(B_{1,\te}u,v)_\XL$, it gives $\Re(B_{1,\te}\LfibD_{1,\te}v,v)_\XL\leq r(B_{1,\te}v,v)_\XL$.
\end{proof}

To bound the high modes of $L_{1,\te}^{{\rm fib}\,\Sigma}$, we need the analogue of Lemma~\ref{lem:tm:tail} for the homogeneous weight.

\begin{lemma}
\label{lem:tm:tail:hom}
    Let $\te\in[0,1)$. For $m\neq0,-1$, let $\ti t_{m,1,\te}$ be the coefficient playing, for $L_{1,\te}^{{\rm fib}\,\Sigma}$ in place of $\LfibD_{1,\te}$, the role $t_{m,1,\te}$ plays in Remark~\ref{rem:tm:jacobi}, namely
    \begin{align*}
    	\ti t_{m,1,\te}:=\ti A_{m+1,\te}\,\frac{|m+\te|}{|m+1+\te|}+\ti B_{m,\te}\,\frac{|m+1+\te|}{|m+\te|},
    \end{align*}
    and set $\ka_m:=\min\{|m+\te|,|m+1+\te|\}$. Then, for every such $m$ with $\ka_m\geq1$,
    \begin{align*}
        |\ti{t}_{m,1,\te}|<\frac12\Big(1+\frac1{\ka_m^2}\Big).
    \end{align*}
    Since the right-hand side decreases in $\ka_m$, for every $N_0\geq1$,
    \begin{align}
    \label{ineq:tm:tail:hom:N0}
        \sup_{\ka_m>N_0}|\ti t_{m,1,\te}|\leq\frac12\Big(1+\frac1{N_0^2}\Big).
    \end{align}
\end{lemma}

\begin{proof}
    For $\ka:=m+\te>0$,
    \begin{align*}
        \ti t_{m,1,\te}
        =\ti A_{m+1,\te}\frac{\ka}{\ka+1}+\ti B_{m,\te}\frac{\ka+1}{\ka}
        =\frac12+\frac{\ka^2-\ka-1}{2\ka^2(\ka+1)^2}.
    \end{align*}
    Since $|\ka^2-\ka-1|<(\ka+1)^2$ for every $\ka>0$, this gives $|\ti t_{m,1,\te}|<\frac12+\frac1{2\ka^2}=\frac12+\frac1{2\ka_m^2}$, since $\ka_m=\min\{\ka,\ka+1\}=\ka$ here. For $\ka<-1$, the symmetry $k\mapsto-k-1$, $\te\mapsto1-\te$ of Remark~\ref{rem:AkBksym} applies, since $|k+\te|$ depends on $k,\te$ only through $k+\te$, and reduces to the case $\ka>0$ evaluated at $\nu:=-\ka-1$, with $\ka_m=\min\{|\ka|,|\ka+1|\}=\nu$.
\end{proof}

\begin{lemma}
\label{lem:h1:core:cert}
    Let $\om=1.1$ and $b=\tfrac{13}{40}$. Then~\eqref{eqn:core:lmi:psd} holds:
    \begin{enumerate}[label=(\roman*),ref=(\roman*)]
        \item\label{item:h1core:i} on $\mathmakebox[\widthof{$[$}][c]{(}\mathmakebox[\widthof{$b$}][c]{0},\mathmakebox[\widthof{$\tfrac12$}][c]{b}]$, for $T=L_{1,\te}^{{\rm fib}\,\Sigma}$, $m=8$, $r=\om$ and $\ga=\frac12\big(1+\frac{\mathmakebox[\widthof{$8$}][c]{1}}{\mathmakebox[\widthof{$88$}][c]{64}}\big)$.
        \item\label{item:h1core:ii} on $\mathmakebox[\widthof{$[$}][c]{[}b,\tfrac12]$, for $T=\mathmakebox[\widthof{$L_{1,\te}^{{\rm fib}\,\Sigma}$}][l]{\LfibD_{1,\te}}$, $m=\mathmakebox[\widthof{$8$}][c]{4}$, $r=\om$ and $\ga=\frac12\big(1+\frac{\mathmakebox[\widthof{$8$}][c]{3}}{\mathmakebox[\widthof{$88$}][c]{16}}\big)$.
    \end{enumerate}
\end{lemma}

\proofcompassi

We now have every tool we need to prove Lemma~\ref{lem:fib:bilform}.

\begin{proofsth}{Lemma~\ref{lem:fib:bilform}}
	Let $\te\in[0,1)$. If $\te=0$, Lemma~\ref{lem:H1:the:0} applies with $B_\te$ the usual norm on $\XL$. By Lemma~\ref{lem:tm:symmetry}, we may assume $\te\in\big(0,\tfrac12\big]$, since $B_\te$ for $\te$ gives $B_{1-\te}$ for $1-\te$.

	For $\te\in\big(0,\tfrac{13}{40}\big]$, using the notation of Lemma~\ref{lem:tm:tail:hom}, since $\eta\perp V_8$ every nonzero summand in the Jacobi form has index $m\geq9$ or $m\leq-10$, so $\ka_m>8$ for every such term, and $\eqref{ineq:tm:tail:hom:N0}$ with $N_0=8$ and the Cauchy--Schwarz inequality give~\eqref{ineq:core:lmi:s1:ga} for $T=L_{1,\te}^{{\rm fib}\,\Sigma}$, $m=8$ and $\ga=\frac12\big(1+\frac1{64}\big)$, and Lemma~\ref{lem:h1:core:cert}\ref{item:h1core:i} gives~\eqref{eqn:core:lmi:psd} for $r=\om$. Lemma~\ref{lem:rescaled:fiber} then yields the desired inequality with $B_\te(u,v):=(\Sigma_{1,\te}^2u,v)_\XL$, equivalent to the usual norm since $\Sigma_{1,\te}$ and $\Sigma_{1,\te}^{-1}$ are both bounded on $\XL$.

	Finally, for $\te\in\big[\tfrac{13}{40},\tfrac12\big]$, $\eqref{ineq:tm:tail:N0}$ with $N_0=4$ and the Cauchy--Schwarz inequality give~\eqref{ineq:core:lmi:s1:ga} with $\ga=\frac12\big(1+\frac3{16}\big)$, and Lemma~\ref{lem:h1:core:cert}\ref{item:h1core:ii} gives~\eqref{eqn:core:lmi:psd} for $r=\om$. Lemma~\ref{lem:core:lmi:s1} then gives the desired inequality directly, with $G=I$ and $B_\te$ the usual norm on $\XL$.
\end{proofsth}

\appendix

\section{Fuchsian theory}
\label{sec:appendix:frobenius}

In this appendix we derive explicit local representations of the solutions to~\eqref{ode:hom}. The qualitative statements used in the main text follow from standard Frobenius theory for regular singular points, but the recursive formulas below are useful both for constructive computations and for the rigorous numerical part of the proof.

Although the computations in Lemmas~\ref{lem:loc:0} and~\ref{lem:loc:1} are essentially standard and closely follow~\cite{TeschlODE}, we include them here since the order of presentation is adapted to the supplementary code and may help the reader compare the argument with the implementation.

\subsection{Homogeneous solutions}

This section provides an overview of the Frobenius theory used in our argument.

Let $P$ and $Q$ be analytic functions on a simply connected neighborhood $V^{z_0}$ of $z_0\in\C$, and consider the equation
\begin{align}
\label{ode:fuchs2pqhom}
    (z-z_0)^2 \pa_z^2 F(z) + (z-z_0) P(z) \pa_z F(z) + Q(z) F(z) = 0.
\end{align}
The local structure of solutions on any simply connected set $U\subset V^{z_0}$ with $z_0\notin U$ is well known. After choosing a branch of $\log(z-z_0)$ on $U$, every solution is a linear combination of two explicit local branches involving analytic functions at $z_0$ and, possibly, $\log(z-z_0)$.

\begin{proofsth}{Lemma~\ref{lem:loc:0}}
    Throughout this proof, $p_k$ and $q_k$ denote the Taylor coefficients of $\Ploc0$ and $\Qloc0$ at $z=0$.
    Let
    \begin{align*}
        \Ploc0(z) :=
        \frac{
        (2\te-1)z^2+2\la z-(2\te+1)
        }{z^2-1}
    \andspace
        \Qloc0(z) :=
        \frac{
        (\te-1)^2z^2+2\la\te z-(\te^2-1)
        }{z^2-1}.
    \end{align*}
    
    Then~\eqref{ode:hom} takes the form of~\eqref{ode:fuchs2pqhom} with $P=\Ploc0$, $Q=\Qloc0$, and $z_0=0$:
    \begin{align}
    \label{ode:fuchs2pqhom0}
        z^2 \pa_z^2 F(z) + z \Ploc0(z) \pa_z F(z) + \Qloc0(z) F(z) = 0.
    \end{align}
    
    The indicial equation at $z=0$ has roots $1-\te$ and $-1-\te$.
    Since the difference between the roots equals $2$, the second local $\Fholh20$ branch may contain a logarithmic term. We now derive the corresponding constructive formulas explicitly. The first local $\Fholh10$ branch has the form
    \begin{align*}
        \Fholh10 = z^{-\te+1} \phi_1(z),
    \end{align*}
    where $\phi_1$ is analytic at $0$.
    We choose $\phi_{1}(0)=1$ and write
    \begin{align*}
        \phi_1(z) =: 1 + \sum_{k \geq 1} a_k z^k,\qquad
        \Ploc{0}(z) =: \sum_{k \geq 0} p_k z^k,\qquad
        \Qloc{0}(z) =: \sum_{k \geq 0} q_k z^k,
    \end{align*}
    
    Hence the coefficients $a_k$ are determined for all $k \geq 1$ by the following recurrence
    \begin{align}
    	\label{rec:0}
        \mathcal I(k+r) a_k +
        \sum_{j=0}^{k-1}
        a_j \big(
        (j+r) p_{k-j} + 
         q_{k-j}\big) = 0,
    \end{align}
    where $r = 1-\te$ and $\mathcal I(x):=x(x-1)+xp_0+q_0$.
    
    The Taylor expansions of $\Ploc0$ and $\Qloc0$ are given by
    \begin{align}
        \Ploc0(z)&=
        (1+2\te)
        +
        (1-\la)\frac{z}{1-z}
        -
        (1+\la)\frac{z}{1+z}
        \nonumber
        \\
        &=
        (1+2\te)
        +
        \sum_{k \geq 1}
        	(1-\la) z^k
        	+
        	(1+\la) (-z)^k
    %%%%%%%%%%%%%%%%%%%%%%
        \\
        \Qloc0(z)&=
        -2\la\te z
        -1+\te^2
        -
        (1-\te+\la\te)
        \frac{z^2}{1-z}
        -
        (1-\te-\la\te)
        \frac{z^2}{1+z}
        \nonumber
        \\
        &=
        -2\la\te z
        -1+\te^2
    		-
        \sum_{k \geq 2}\big(
        	(1-\te+\la\te) z^k
        	+
        	(1-\te-\la\te) (-z)^k\big).
    \end{align}
    
    In particular, we write the first Taylor coefficients centered at $0$ of $p$ and $q$
    \begin{align*}
    	p_0&= 1+2\te,&
    	p_1&= -2\la,&
    	p_2&= 2,\\
    	q_0&=-1+\te^2,&
    	q_1&= -2\la\te,&
    	q_2&=-2(1-\te).
    \end{align*}

    Using~\eqref{rec:0}, we have $a_0=1$ and
    \begin{align*}
    	a_1&=
    	 - \frac{r a_0 p_1 + q_1}
    	 {r(r+1) + (r+1)p_0 + q_0}
    	 =\frac{2\la}3,\\
    	a_2&=
    	 - \frac{r a_0 p_2 + (r+1) a_1 p_1 + a_0 q_2 + a_1 q_1}{
    	 (r+2)(r+1) + (r+2) p_0 + q_0}=
    	  \frac{\la^2}{3}.
    \end{align*}
    
    To construct $\Fholh20$, we use reduction of order and write it in the form $C\Fholh10$. Substituting $F(z) = C(z) F^{h,0}_1(z)$ into~\eqref{ode:fuchs2pqhom0} yields
    \begin{align}
    \label{ode:red_ord}
        z^2\pa_z^2 C(z) \Fholh10(z) + 2z^2\pa_z C(z) \pa_z \Fholh10(z) + z\pa_z C(z) \Fholh10(z)  \Ploc0(z) = 0.
    \end{align}
    
    Since $p_0=1+2\te$, the function $\ti P^0$ given by
    \begin{align*}
        \ti P^0(z) 
        := 
        \frac1z\big(
        \Ploc0(z)-(1+2\te)\big) = 
        \sum_{k \geq 0} p_{k+1} z^k
    \end{align*}
    is analytic at $0$. From~\eqref{ode:red_ord} we obtain
    \begin{align*}
        z\frac{\pa_z^2 C(z)}{\pa_z C(z)}
        \Fholh10(z)
        &=
        - 2z \pa_z \Fholh10(z)
        - \Fholh10(z) \Ploc0(z)\\
        &=
        \big(
            -2 z \pa_z \phi_1(z)
            +
            2 (\te-1) \phi_1(z)
            -
            \phi_1(z) \Ploc0(z)
        \big)
        z^{-\te+1}.
    \end{align*}
    Dividing by $z\Fholh10(z) = z^{2-\te}\phi_1(z)$ yields
    \begin{align}
        \frac{\pa_z^2 C(z)}{\pa_z C(z)}
        =
        -2\frac{\pa_z \phi_1(z)}{\phi_1(z)}
        -
        \frac1z
        \big(
        2-2\te
        +
        \Ploc0(z)
        \big)
        =
        -2\frac{\pa_z \phi_1(z)}{\phi_1(z)}
        -
        \frac3z
        -
        \ti P^0(z).
        \label{ode:red_ord2}
    \end{align}
    
    Choose $r>0$ so that $\phi_1\neq 0$ on $D_r(0)$.
    Integrating~\eqref{ode:red_ord2} we obtain 
    \begin{align*}
        \pa_z C(z) 
        =
        \frac{1}{z^3}\psi(z),&&
        \psi(z) := 
        \frac{1}{\phi_1(z)^2} \exp\!\left(
        -\int^z
        \ti P^0(y)
        \dy
        \right),
    \end{align*}
    where the antiderivative is defined by
    \begin{align*}
        \int^z \ti P^0(y)\dy := 
        \sum_{k \geq 0}
            \frac1{k+1} 
            p_{k+1}
            z^{k+1}
        = 
        \sum_{k \geq 1} 
            \frac1k p_k
            z^k.
    \end{align*}
    This primitive is holomorphic at $0$. We have the same for $\psi$ since $\phi_1$ does not vanish at $0$.
    
    Hence, integrating once more, we obtain
    \begin{align*}
        C(z) = 
        \psi_2 \log(z)
        +
        \sum_{\substack{k \geq 0\\k\neq 2}}
        \frac1{k-2}\psi_k z^{k-2}.
    \end{align*}
    Here $\psi(z)=\sum_{k\ge0}\psi_k z^k$.
    
    Define $\phi_2:D_r(0)\to\C$ by
    \begin{align*}
        \phi_2(z) =
        \phi_1(z)
        \sum_{\substack{k \geq 0\\k\neq 2}}
        \frac1{k-2}\psi_k z^k,
    \end{align*}
    so that $\Fholh20$ is given by
    \begin{align*}
        \Fholh20(z) =
        \psi_{2} \Fholh10(z) \log(z)
        +
        z^{-\te-1} \phi_2(z).
    \end{align*}
    
    Finally, we compute the exact value of $\psi_2$. Notice that
    \begin{align*}
        \pa_z \psi = 
        -2\frac{\pa_z \phi_1}{\phi_1} \psi
        -\ti P^0  \psi\andspace
        \pa_z^2 \psi = 
        -2
        \frac{\pa_z^2 \phi_1}{\phi_1}
        \psi
        +
        2\Big(\frac{\pa_z \phi_1}{\phi_1}\Big)^2
        \psi
        + \frac{\big(\pa_z \psi\big)^2}{\psi}
        -
        \pa_z \ti P^0 \psi.
    \end{align*}
    
    Since $\phi_1(0)=\psi(0)=1$, we have
    \begin{align*}
        \psi_2 = 
        \frac12 \left.\frac{d^2\psi}{dz^2}\right|_{z=0}
        =
        - 2a_2
        +
        a_1^2
        +\frac{\psi_1^2}2
        -
        \frac{p_2}2
        =
        - \frac{2\la^2}3
        +
        \frac{4\la^2}{9}
        +
        \frac{2\la^2}{9}
        -
        1=-1.
    \end{align*}
    where we have used that $\psi_1 = -2a_1-p_1 = \frac{2\la}{3}$.
\end{proofsth}

\begin{proofsth}{Lemma~\ref{lem:loc:1}}
    Throughout this proof, $p_k$ and $q_k$ denote the Taylor coefficients of $\Ploc1$ and $\Qloc1$ at $z=1$.
    Dividing~\eqref{ode:hom} by $z^2(z+1)$ and multiplying by $(z-1)$ casts it in the Fuchsian standard form of~\eqref{ode:fuchs2pqhom} with $z_0=1$:
    \begin{align*}
        (z-1)^2\pa_z^2 F(z) + (z-1)\Ploc 1(z)\pa_z F(z) + \Qloc 1(z) F(z) = 0,
    \end{align*}
    where
    \begin{align*}
        \Ploc 1(z) := \frac{\Pglo(z)}{z(z+1)},
        \qquad
        \Qloc 1(z) := \frac{(z-1)\Qglo(z)}{z^2(z+1)}.
    \end{align*}
    Both $\Ploc 1$ and $\Qloc 1$ are analytic on a neighborhood of $z=1$. Setting $p_0:=\Ploc 1(1)$ and $q_0:=\Qloc 1(1)$, direct computation yields
    \begin{align*}
        p_0 = \frac{\Pglo(1)}{2} = -(1-\la),
        \qquad
        q_0 = 0,
    \end{align*}
    so the indicial equation $r(r-1)+r p_0+q_0=0$ becomes $r^2-(2-\la)r=0$, with roots $r_1=0$ and $r_2=2-\la$. Since $\la\in\C-\Z$, the difference $r_2-r_1=2-\la\notin\Z$, and both local solutions are pure power series:
    \begin{align*}
        \Fholh11(z) = \phi_1(z),
        \qquad
        \Fholh21(z) = (z-1)^{2-\la}\phi_2(z),
    \end{align*}
    with $\phi_1,\phi_2$ analytic at $1$ and normalized so that $\phi_1(1)=\phi_2(1)=1$.

    Expanding
    \begin{align*}
        \phi_j(z) = \sum_{k \geq 0} a^{(j)}_k(z-1)^k,
        \qquad
        \Ploc 1(z) = \sum_{k \geq 0} p_k(z-1)^k,
        \qquad
        \Qloc 1(z) = \sum_{k \geq 0} q_k(z-1)^k,
    \end{align*}
    the Taylor coefficients of $\phi_j$ are determined for $k \geq 1$ by the Frobenius recurrence
    \begin{align}
    \label{rec:1}
        \big((k+r_j)(k+r_j-1)+(k+r_j) p_0+q_0\big)  a^{(j)}_k +
        \sum_{l=0}^{k-1}\big((l+r_j)p_{k-l}+q_{k-l}\big) a^{(j)}_l = 0,
    \end{align}
    together with $a^{(j)}_0=1$. The leading factor equals $k(k-2+\la)$ for $r_1=0$ and $k(k+2-\la)$ for $r_2=2-\la$. Both are nonzero for every $k \geq 1$ since $\la\notin\Z$, so~\eqref{rec:1} determines $\phi_1$ and $\phi_2$ uniquely.
\end{proofsth}

\begin{proofsth}{Lemma~\ref{lem:loc:-1}}
    Throughout this proof, $p_k$ and $q_k$ denote the Taylor coefficients of $\Ploc{-1}$ and $\Qloc{-1}$ at $z=-1$.
    Dividing~\eqref{ode:hom} by $z^2(z-1)$ and multiplying by $(z+1)$ casts it in the Fuchsian standard form of~\eqref{ode:fuchs2pqhom} with $z_0=-1$:
    \begin{align*}
        (z+1)^2\pa_z^2 F(z) + (z+1)\Ploc{-1}(z)\pa_z F(z) + \Qloc{-1}(z) F(z) = 0,
    \end{align*}
    where
    \begin{align*}
        \Ploc{-1}(z) := \frac{\Pglo(z)}{z(z-1)},
        \qquad
        \Qloc{-1}(z) := \frac{(z+1)\Qglo(z)}{z^2(z-1)}.
    \end{align*}
    Both $\Ploc{-1}$ and $\Qloc{-1}$ are analytic on a neighborhood of $z=-1$. Setting $p_0:=\Ploc{-1}(-1)$ and $q_0:=\Qloc{-1}(-1)$, direct computation yields
    \begin{align*}
        p_0 = \frac{\Pglo(-1)}{2} = -(1+\la),
        \qquad
        q_0 = 0,
    \end{align*}
    so the indicial equation $r(r-1)+r p_0+q_0=0$ becomes $r^2-(2+\la)r=0$, with roots $r_1=0$ and $r_2=2+\la$. Since $\la\in\C-\Z$, the difference $r_2-r_1=2+\la\notin\Z$, and both local solutions are pure power series:
    \begin{align*}
        \Fholh1{-1}(z) = \phi_1(z),
        \qquad
        \Fholh2{-1}(z) = (z+1)^{2+\la}\phi_2(z),
    \end{align*}
    with $\phi_1,\phi_2$ analytic at $-1$ and normalized so that $\phi_1(-1)=\phi_2(-1)=1$. The Taylor coefficients of $\phi_j$ are determined for $k\geq1$ by the Frobenius recurrence~\eqref{rec:1}, now with $(z-1)$ replaced by $(z+1)$ and with the coefficients $p_k,q_k$ of $\Ploc{-1},\Qloc{-1}$. The leading factor equals $k(k-2-\la)$ for $r_1=0$ and $k(k+2+\la)$ for $r_2=2+\la$, both nonzero for $k\geq1$ since $\la\notin\Z$.
\end{proofsth}

\subsection{Analytic particular solutions}

This section proves Lemmas~\ref{lem:loc:aff0} and~\ref{lem:loc:aff1}. Rather than constructing the particular solutions by variation of constants, we obtain them directly: we look for a power series solving the inhomogeneous equation, read off the resulting Frobenius recurrence, and prove its convergence by a majorant. This is both shorter and closer to what the computer-assisted part actually evaluates, since the same recurrence produces the Taylor coefficients used in Appendix~\ref{sec:appendix:errors}.

\subsubsection{A holomorphic solution from the Frobenius recurrence}

Let $z_0\in\C$, let $P,Q,G$ be holomorphic on a disk $D_R(z_0)$, $R>0$, and consider the inhomogeneous Fuchsian equation in normal form
\begin{align}
\label{ode:aff:normal}
    (z-z_0)^2\pa_z^2 F(z) + (z-z_0) P(z) \pa_z F(z) + Q(z) F(z) = G(z).
\end{align}
Write the local variable $x:=z-z_0$ and expand
\begin{align*}
    P(z)=\sum_{k\ge0}p_k x^k,\qquad
    Q(z)=\sum_{k\ge0}q_k x^k,\qquad
    G(z)=\sum_{k\ge0}g_k x^k.
\end{align*}
We look for a holomorphic solution $F(z)=\sum_{k\ge0}f_k x^k$. Substituting into~\eqref{ode:aff:normal} and collecting the coefficient of $x^k$ gives, for every $k\ge0$,
\begin{align}
\label{rec:aff}
    \mathcal I(k) f_k
    +
    \sum_{j=0}^{k-1}\big(jp_{k-j}+q_{k-j}\big) f_j
    = g_k,
    \qquad
    \mathcal I(k):=
    k(k-1)+kp_0+q_0,
\end{align}
where $\mathcal I$ is the indicial polynomial of~\eqref{ode:aff:normal}, with roots $r_1,r_2$ so that $\mathcal I(k)=(k-r_1)(k-r_2)$.

\begin{lemma}
\label{lem:aff:rec}
    With the notation above, assume
    \begin{enumerate}
        \item[(i)] $\mathcal I(k)\neq0$ for every integer $k\ge1$, and
        \item[(ii)] either $\mathcal I(0)\neq0$, or else $\mathcal I(0)=0$ and $g_0=0$.
    \end{enumerate}
    Then~\eqref{ode:aff:normal} has a holomorphic solution $f$ on $D_R(z_0)$ whose Taylor coefficients are determined by the recurrence~\eqref{rec:aff}. In the resonant case $\mathcal I(0)=0$ the coefficient $f_0$ is free and we fix $f_0:=0$. In particular $f$ has radius of convergence at least $R$.
\end{lemma}

\begin{proof}
    \emph{Coefficients.} For $k\ge1$, hypothesis (i) lets us solve~\eqref{rec:aff} for $f_k$ in terms of $f_0,\dots,f_{k-1}$. For $k=0$, \eqref{rec:aff} reads $\mathcal I(0) f_0=g_0$: if $\mathcal I(0)\neq0$ this fixes $f_0$, while if $\mathcal I(0)=0$ then $g_0=0$ by (ii), the equation $0=0$ holds automatically, and we are free to set $f_0:=0$. Thus all $f_k$ are well defined.

    \emph{Convergence.} Fix $\rho\in(0,R)$. Since $P,Q,G$ are holomorphic on $D_R(z_0)$, Cauchy's estimates provide $M\ge1$ with
    \begin{align*}
        |p_k|\leq M\rho^{-k},\qquad |q_k|\leq M\rho^{-k},\qquad |g_k|\leq M\rho^{-k}\qquad(k\ge0).
    \end{align*}
    Because $\mathcal I(k)=(k-r_1)(k-r_2)\sim k^2$ as $k\to\infty$ and $\mathcal I(k)\neq0$ for every $k\ge1$, the ratios $\frac{|\mathcal I(k)|}{k^2}$ are bounded away from $0$, so there is $c>0$ with
    \begin{align}
        |\mathcal I(k)|\geq c k^2\qquad(k\ge1).
    \end{align}
    From~\eqref{rec:aff}, using $|j p_{k-j}+q_{k-j}|\leq M(j+1)\rho^{-(k-j)}$, we get for $k\ge1$
    \begin{align*}
        |f_k|\leq\frac{1}{c k^2}\Big(M\rho^{-k}+\sum_{j=0}^{k-1}M(j+1)\rho^{-(k-j)}|f_j|\Big).
    \end{align*}
    Set $b_k:=|f_k|\rho^{k}$ and $\Xi_k:=\sum_{j=0}^{k}(j+1)b_j$. Multiplying the previous bound by $\rho^k$ gives
    \begin{align*}
        b_k\leq\frac{M}{c k^2}\big(1+\Xi_{k-1}\big)\qquad(k\ge1).
    \end{align*}
    Since $\frac{k+1}{k^2}\leq \frac{2}{k}$, it follows that for every $k\ge1$
    \begin{align*}
        \Xi_k+1
        = (\Xi_{k-1}+1)+(k+1)b_k
        \leq (\Xi_{k-1}+1)\bigg(1+\frac{2M}{c k}\bigg).
    \end{align*}
    Iterating and using 
    \begin{align*}
        \prod_{j=1}^k\bigg(1+\frac{2M}{cj}\bigg)\leq\exp\!\bigg(\frac{2M}{c}\sum_{j=1}^k\frac1j\bigg)\leq C_1k^{2M/c},
    \end{align*}
    we obtain $\Xi_k\leq C_1(\Xi_0+1) k^{2M/c}$, which grows only polynomially in $k$.
    Hence 
    \begin{align*}
        b_k\leq\frac{M}{c k^2}(1+\Xi_{k-1})\leq C_2 k^{2M/c}
    \end{align*}
    is polynomially bounded as well, so
    \begin{align*}
        |f_k|=b_k \rho^{-k}\leq C_2 \frac{k^{2M/c}}{\rho^k},
        \qquad\text{whence}\qquad
        \limsup_{k\to\infty}|f_k|^{1/k}\leq\frac1\rho.
    \end{align*}
    As $\rho\in(0,R)$ was arbitrary, $\sum_{k\ge0}f_k x^k$ has radius of convergence at least $R$. Its sum $f$ is holomorphic on $D_R(z_0)$ and, since its coefficients satisfy~\eqref{rec:aff}, which is equivalent to~\eqref{ode:aff:normal}, $f$ solves the equation.
\end{proof}

We now apply Lemma~\ref{lem:aff:rec} at the two singular points $z_0=0$ and $z_0=1$. In each case the right-hand side $G$ of~\eqref{ode:LF=G} is holomorphic, and putting~\eqref{ode:LF=G} into the normal form~\eqref{ode:aff:normal} divides $G$ by exactly the factor used for the homogeneous equation in the proofs of Lemmas~\ref{lem:loc:0} and~\ref{lem:loc:1}.

\begin{proofsth}{Lemma~\ref{lem:loc:aff0}}
    At $z_0=0$, dividing~\eqref{ode:LF=G} by $z^2-1$ casts it in the normal form~\eqref{ode:aff:normal} with $P=\Ploc0$ and $Q=\Qloc0$ as in the proof of Lemma~\ref{lem:loc:0}, and right-hand side
    \begin{align*}
        G^0(z):=\frac{G(z)}{z^2-1},
    \end{align*}
    which is holomorphic on $D_r(0)$ since $r<1$. The indicial polynomial has roots $1-\te$ and $-1-\te$, so
    \begin{align*}
        \mathcal I(n)=(n-1+\te)(n+1+\te)\neq0\qquad\text{for every integer }n\ge0,
    \end{align*}
    because $\te\in(0,1)$. Hence hypotheses (i)--(ii) of Lemma~\ref{lem:aff:rec} hold, with $\mathcal I(0)\neq0$, and Lemma~\ref{lem:aff:rec} produces a holomorphic solution $\Fholp0:D_r(0)\to\C$ of~\eqref{ode:LF=G} whose Taylor coefficients are computed recursively by~\eqref{rec:aff}.
\end{proofsth}

\begin{proofsth}{Lemma~\ref{lem:loc:aff1}}
    At $z_0=1$, dividing~\eqref{ode:LF=G} by $z^2(z+1)$ and multiplying by $(z-1)$ casts it in the normal form~\eqref{ode:aff:normal} with $P=\Ploc1$ and $Q=\Qloc1$ as in the proof of Lemma~\ref{lem:loc:1}, and right-hand side
    \begin{align*}
        G^1(z):=\frac{(z-1) G(z)}{z^2(z+1)},
    \end{align*}
    which is holomorphic on $D_r(1)$. The indicial polynomial has roots $r_1=0$ and $r_2=2-\la$, so
    \begin{align*}
        \mathcal I(n)=n (n-2+\la),
    \end{align*}
    which vanishes at $n=0$ and, since $\la\notin\Z$, is nonzero for every $n\ge1$. The root $r_1=0$ is resonant. However, the normalization above carries an explicit factor $(z-1)$, so that
    \begin{align*}
        g_0=G^1(1)=0
    \end{align*}
    for \emph{every} holomorphic $G$. Hence hypotheses (i)--(ii) of Lemma~\ref{lem:aff:rec} hold in the resonant case with $g_0=0$, and Lemma~\ref{lem:aff:rec} produces a holomorphic solution $\Fholp1:D_r(1)\to\C$ of~\eqref{ode:LF=G}, normalized by $\Fholp1(1)=f_0=0$, whose Taylor coefficients are computed recursively by~\eqref{rec:aff}.

    Finally, let $F$ be any holomorphic solution of~\eqref{ode:LF=G} at $1$. Then $\ti F:=F-\Fholp1$ is holomorphic at $1$ and solves~\eqref{ode:hom}, so by Lemma~\ref{lem:loc:1} we have $\ti F=C_1\Fholh11+C_2\Fholh21$ for some $C_1,C_2\in\C$. Since $\ti F$ and $\Fholh11=\phi_1$ are holomorphic at $1$ while $\Fholh21=(z-1)^{2-\la}\phi_2$ is not (as $\la\notin\Z$), we must have $C_2=0$, and therefore $F=\Fholp1+C_1\Fholh11$.
\end{proofsth}

%%%%%%%%%%%%%%%%%%%%%%%%%%%%%%%%%%%%%%
%%%%%%%%%%%%%%%%%%%%%%%%%%%%%%%%%%%%%%

\subsection{A priori estimate at the boundary singular points}
\label{sec:appendix:apriori}

This appendix contains the proof of Lemma~\ref{lem:loc:apriori}. The argument rests on a pair of weighted Hardy inequalities that we state and prove first.

\subsubsection{Weighted Hardy operators}

For $\sigma\geq0$, we use on $\RR$ the same Sobolev convention as on $\T$:
\begin{align*}
    \norm{f}_{H^\sigma(\RR)}^2
    :=\int_\RR(1+|\xi|^{2\sigma})|\ha f(\xi)|^2\,d\xi,
    \qquad
    \ha f(\xi):=\frac1{\sqrt{2\pi}}\int_\RR f(x)e^{-ix\xi}\,dx.
\end{align*}
At $\sigma=0$ the multiplier $|\xi|^0$ means $1$, also at $\xi=0$.
For $\de>0$, $\beta\in\C$, and $w$ supported in $(-\de,\de)$, define for $x\in(-\de,\de)-\{0\}$,
\begin{align*}
    (\Hard0\beta w)(x):=|x|^{-\beta}\int_0^{|x|}\tau^{\beta-1}w(\sgn(x)\tau)\,d\tau,\qquad
    (\Hard\de\beta w)(x):=|x|^\beta\int_{|x|}^\de\tau^{-\beta-1}w(\sgn(x)\tau)\,d\tau.
\end{align*}
For $|x|\geq\de$ use the same formula for $\Hard0\beta w$ (with $w$ extended by zero) and set $\Hard\de\beta w=0$. Writing $D_x:=x\pa_x$, these operators satisfy
\begin{equation}
    \label{eq:hardyODE}
    D_x(\Hard0\beta w)+\beta\Hard0\beta w=w,
    \qquad
    D_x(\Hard\de\beta w)-\beta\Hard\de\beta w=-w.
\end{equation}
\begin{lemma}
\label{lem:hardy}
Let $\sigma\geq0$ and $w\in H^\sigma(\RR)$ be supported in $(-\de,\de)$.
\begin{enumerate}[label=\textup{(\roman*)},ref=\textup{(\roman*)}]
    \item\label{Hardy:enu1} If $\Re\beta>\frac12$, then
    \begin{align*}
        \norm{\Hard0\beta w}_{H^\sigma(\RR)}
        \leq\frac{\norm{w}_{H^\sigma(\RR)}}{\Re\beta-\frac12}.
    \end{align*}
    \item\label{Hardy:enu2} If $\Re\beta>\sigma-\frac12$, then
    \begin{align*}
        \norm{\Hard\de\beta w}_{H^\sigma(\RR)}
        \leq\frac{\norm{w}_{H^\sigma(\RR)}}{\Re\beta-\sigma+\frac12}.
    \end{align*}
\end{enumerate}
\end{lemma}
\begin{proof}
Let $w$ be as in the statement and put $T_tw(x):=w(tx)$ for $t>0$. Changing variables gives
\begin{align}
    \label{eq:hardy:dilation-norm}
    \norm{T_tw}_{H^\sigma(\RR)}^2
    &=t^{-1}\int_\RR(1+t^{2\sigma}|\xi|^{2\sigma})|\ha w(\xi)|^2\,d\xi.
\end{align}
For $0<t\leq1$ the weight in~\eqref{eq:hardy:dilation-norm} is at most
$1+|\xi|^{2\sigma}$. For $t\geq1$ it is at most
$t^{2\sigma}(1+|\xi|^{2\sigma})$. Thus
\begin{equation}
    \label{eq:hardy:dilation-bound}
    \norm{T_tw}_{H^\sigma(\RR)}
    \leq
    \begin{cases}
        t^{-1/2}\norm{w}_{H^\sigma(\RR)},&0<t\leq1,\\
        t^{\sigma-1/2}\norm{w}_{H^\sigma(\RR)},&t\geq1.
    \end{cases}
\end{equation}
Under the respective hypotheses of the two parts, the substitution $\tau=t|x|$ and the support of $w$ give, for $x\ne0$,
\begin{align}
    \label{eq:hardy:dilation-representations}
    (\Hard0\beta w)(x)=\int_0^1t^{\beta-1}T_tw(x)\,dt,\qquad
    (\Hard\de\beta w)(x)=\int_1^\infty t^{-\beta-1}T_tw(x)\,dt.
\end{align}

For part~\ref{Hardy:enu1}, assume $\Re\beta>\frac12$. The first case of~\eqref{eq:hardy:dilation-bound} gives
\begin{align}
    \label{eq:hardy:origin-integrability}
    \int_0^1\norm{t^{\beta-1}T_tw}_{H^\sigma(\RR)}\,dt
    &\leq\norm{w}_{H^\sigma(\RR)}
      \int_0^1t^{\Re\beta-3/2}\,dt
     =\frac{\norm{w}_{H^\sigma(\RR)}}{\Re\beta-\frac12}.
\end{align}

For part~\ref{Hardy:enu2}, assume instead $\Re\beta>\sigma-\frac12$. The second case of~\eqref{eq:hardy:dilation-bound} gives
\begin{align}
    \label{eq:hardy:outer-integrability}
    \int_1^\infty\norm{t^{-\beta-1}T_tw}_{H^\sigma(\RR)}\,dt
    &\leq\norm{w}_{H^\sigma(\RR)}
      \int_1^\infty t^{-\Re\beta+\sigma-3/2}\,dt
     =\frac{\norm{w}_{H^\sigma(\RR)}}{\Re\beta-\sigma+\frac12}.
\end{align}
The finite bounds in~\eqref{eq:hardy:origin-integrability} and~\eqref{eq:hardy:outer-integrability} show that the integrals in~\eqref{eq:hardy:dilation-representations} converge in $H^\sigma(\RR)$. The triangle inequality then yields
\begin{align*}
    \norm{\Hard0\beta w}_{H^\sigma(\RR)}
    \leq\frac{\norm{w}_{H^\sigma(\RR)}}{\Re\beta-\frac12}\andspace
    \norm{\Hard\de\beta w}_{H^\sigma(\RR)}
    \leq\frac{\norm{w}_{H^\sigma(\RR)}}{\Re\beta-\sigma+\frac12}.
\end{align*}
\end{proof}

\begin{proofsth}{Lemma~\ref{lem:loc:apriori}}
    Throughout, $A\lesssim B$ means $A\leq C_{\la,\te,\eps}B$.
    Write $z=e^{ix}$, so $\pa_x=iz\pa_z$ and $z^2\pa_z^2=(z\pa_z)^2-z\pa_z=-\pa_x^2+i\pa_x$.
    Hence 
    \begin{align*}
        z^2(z^2-1)\pa_z^2F=(z^2-1)(-\pa_x^2+i\pa_x)F
        \andspace
        z\Pglo\pa_zF=-i\Pglo\pa_xF.
    \end{align*}
    
    The trace $f:=F|_\T$ therefore satisfies
    \begin{equation}
    \label{eq:trace:app}
      (1-z^2) \pa_x^2 f+i\bigl(z^2-1-\Pglo\bigr) \pa_xf+\Qglo f=g,\qquad g:=G|_\T.
    \end{equation}
    Since $F\in\XHhol{s+1}$ and $G\in\XHhol s$, the functions $f,\pa_xf,g$ are continuous on $\T$.

    Write $\sigma:=s-1\in[0,1]$.  Since $(1+k^2)^{s+1}\lesssim k^4(1+k^2)^{s-1}+1$,
    \begin{equation}
    \label{eq:reduction:app}
      \norm{f}_\XH{s+1}^2
      \lesssim
      \norm{\pa_x^2 f}_{H^{s-1}(\T)}^2
      +
      \norm{f}_\XL^2,
    \end{equation}
    and $\norm{f}_\XL\lesssim\norm{F}_{\XHhol{s+1-\eps}}$. It therefore suffices to prove
    \begin{equation}
    \label{eq:goal:app}
      \norm{\pa_x^2 f}_{H^{s-1}(\T)}
      \lesssim
      \norm{G}_{\XHhol s}+\norm{F}_{\XHhol{s+1-\eps}}.
    \end{equation}
    
    Fix $0<\de<\frac\pi2$ and a smooth partition $1=\chi_0^2+\chi_+^2+\chi_-^2$ on $\T$ where 
    \begin{align*}
        \chi_+&=1 \text{ on }
        \left\{|x|\leq\frac12\de
        \right\},\quad 
        \supp\chi_+\subset
        \left\{|x|<\frac34\de
        \right\},\\
        \chi_-&=1 \text{ on }
        \left\{|x-\pi|\leq\frac12\de
        \right\},\quad 
        \supp\chi_-\subset
        \left\{|x-\pi|<\frac34\de
        \right\},
    \end{align*}
     and $\chi_0$ is supported away from $z=\pm1$. Then
    \begin{equation}
    \label{eq:pou:app}
      \norm{v}_{H^{s-1}(\T)}\lesssim\norm{\chi_0 v}_{H^{s-1}(\T)}+\norm{\chi_+ v}_{H^{s-1}(\T)}+\norm{\chi_- v}_{H^{s-1}(\T)}.
    \end{equation}

    \paragraph{Interior piece.}
    On $\supp\chi_0$ one has $|1-z^2|\geq c(\de)>0$, so~\eqref{eq:trace:app} gives
    \begin{align*}
        \chi_0 \pa_x^2 f=\frac{\chi_0}{1-z^2}\big(g-i(z^2-1-\Pglo)\pa_xf-\Qglo f\big).
    \end{align*}
    The factor $\frac{\chi_0}{1-z^2}$ is smooth, so multiplication by it is bounded on $H^{s-1}(\T)$, giving
    \begin{align*}
        \norm{\chi_0 \pa_x^2 f}_{H^{s-1}(\T)}\lesssim
        \norm{g}_{H^{s-1}(\T)}
        +
        \norm{\pa_xf}_{H^{s-1}(\T)}
        +
        \norm{f}_{H^{s-1}(\T)}
        \lesssim
        \norm{G}_\XHhol{s}+
        \norm{F}_\XHhol{s+1-\eps}.
    \end{align*}

    For functions supported in either singular chart, the $H^\sigma(\T)$ norm is equivalent to the $H^\sigma(\RR)$ norm of the zero extension in the local coordinate.

    \paragraph{Near $z=-1$ (admissible branch).}
    Here the local exponent is $2+\la$, the admissible branch. Writing $x$ for the
    local coordinate at this singular point (so that $z^2-1=e^{2ix}-1$ vanishes at
    $x=0$), set
    \begin{align*}
        h(x):=\frac{e^{2ix}-1}{x}=2i\int_0^1e^{2isx}\ds.
    \end{align*}
    Notice that $h$ is smooth at $0$, $h(0)=2i$, $h\ne0$ for $|x|<\de$, after choosing $\de>0$ sufficiently small. Dividing~\eqref{eq:trace:app} by $-h$ gives on $\{|x|<\de\}$:
    \begin{equation}
    \label{eq:loc1:app}
      D_x\pa_x f=\ti{g}+\ti{p}\pa_xf+\ti{q} f,\qquad
      \ti{p}:=\frac{c_1}{h},\quad\ti{q}:=\frac{\Qglo}{h},\quad\ti{g}:=-\frac{g}{h},
    \end{equation}
    where $c_1=i(z^2-1-\Pglo)$ and $D_x f(x):=x\pa_x f(x)$.
    
    Differentiating~\eqref{eq:loc1:app} and setting $v=\pa_x^2f$:
    \begin{equation}
    \label{eq:v1:app}
      D_x v+(1-\ti{p}) v=\kappa_-,\qquad\kappa_-:=\pa_x\ti{g}+(\pa_x \ti{p}+\ti{q})\pa_x f+f \pa_x\ti{q}.
    \end{equation}
    
    Multiplying~\eqref{eq:v1:app} by $\chi_-$ and writing $V:=\chi_- v$,
    \begin{equation}
    \label{eq:v1loc:app}
      D_x V+(1-\ti{p})V=
      \chi_-\kappa_-+
      vD_x\chi_-=:K_-.
    \end{equation}
    The term $v D_x\chi_-$ is supported in $\{\frac12\de<|x|<\frac34\de\}$, where the interior estimate applies, giving
    \begin{equation}
      \norm{K_-}_{H^{s-1}(\T)}\lesssim\norm{G}_{\XHhol s}+\norm{F}_{\XHhol{s+1-\eps}}.
    \end{equation}
    Extend $K_-$ by zero to $\RR$ in the local coordinate.

Write $1-\ti p=-\la+x\rho$ and $R(x):=\int_0^x\rho(y)\,dy$.
Extend $e^RK_-$ by zero to $\RR$.
On either side of zero, the difference between $V$ and $-e^{-R}\Hard\de\la[e^RK_-]$ is a multiple of $|x|^\la e^{-R(x)}$.
Both vanish near $\pm\de$, so the multiples are zero and
\begin{equation}
    \label{eq:Vrep:app}
    V=-e^{-R}\Hard\de\la[e^RK_-].
\end{equation}
Since $\Re\la>\frac12\geq\sigma-\frac12$, Lemma~\ref{lem:hardy}\ref{Hardy:enu2} applies. A smooth cutoff supported in $(-\de,\de)$ and equal to one on $\supp\chi_-$ gives
\begin{align*}
    \norm{\chi_-v}_{H^\sigma(\RR)}
    \lesssim\norm{\Hard\de\la[e^RK_-]}_{H^\sigma(\RR)}
    \lesssim\norm{K_-}_{H^\sigma(\RR)}
    \lesssim\norm{G}_{\XHhol s}+\norm{F}_{\XHhol{s+1-\eps}}.
\end{align*}

\paragraph{Near $z=+1$ (excluded branch).}
The same local computation, in the coordinate $y$ centered at this point, gives
\begin{align*}
    y\pa_yv+(1-\ti p_+)v=\kappa_+,
    \qquad 1-\ti p_+(0)=\la.
\end{align*}
Set $V_+:=\chi_+v$ and $K_+:=\chi_+\kappa_++vD_y\chi_+$, where $D_y:=y\pa_y$. The interior estimate on $\supp D_y\chi_+$ gives, after extension by zero to $\RR$,
\begin{align*}
    \norm{K_+}_{H^\sigma(\RR)}
    \lesssim\norm{G}_{\XHhol s}+\norm{F}_{\XHhol{s+1-\eps}}.
\end{align*}
Write $1-\ti p_+(y)=\la+y\rho_+(y)$ and $R_+(y):=\int_0^y\rho_+(r)\,dr$.
Extend $e^{R_+}K_+$ by zero to $\RR$.
By Lemma~\ref{lem:hardy}\ref{Hardy:enu1}, $e^{-R_+}\Hard0\la[e^{R_+}K_+]$ belongs to $H^\sigma(-\de,\de)$.
On either side of zero, its difference from $V_+$ is a multiple of $|y|^{-\la}e^{-R_+(y)}$.
Both $V_+$ and the particular solution belong to $L^2(-\de,\de)$. Since $\Re\la>\frac12$, the multiple must vanish. Thus
\begin{align*}
    V_+=e^{-R_+}\Hard0\la[e^{R_+}K_+]
    \qquad\text{on }(-\de,\de).
\end{align*}
A smooth cutoff supported in $(-\de,\de)$ and equal to one on $\supp\chi_+$ gives
\begin{align*}
    \norm{\chi_+v}_{H^\sigma(\RR)}
    \lesssim\norm{K_+}_{H^\sigma(\RR)}
    \lesssim\norm{G}_{\XHhol s}+\norm{F}_{\XHhol{s+1-\eps}}.
\end{align*}
    
    \paragraph{Conclusion.}
    Inserting all three bounds into~\eqref{eq:pou:app} gives~\eqref{eq:goal:app}. Combined with~\eqref{eq:reduction:app} this yields
    \begin{align*}
      \norm{F}_\XHhol{s+1}=\norm{f}_\XH{s+1}
      \lesssim\norm{\pa_x^2 f}_{H^{s-1}(\T)}+\norm{f}_\XL
      \lesssim\norm{G}_{\XHhol s}+\norm{F}_{\XHhol{s+1-\eps}},
    \end{align*}
    completing the proof.
\end{proofsth}

%%%%%%%%%%%%%%%%%%%%%%%%%%%%%%%%%
%%%%%%%%%%%%%%%%%%%%%%%%%%%%%%%%%

\section{Error bounds for Taylor approximations}
\label{sec:appendix:errors}

The computer-assisted verification of Lemma~\ref{lem:negposjump} requires rigorous enclosures for the entries of $\Mcrit$, which are values at $x_0=\frac12$ of analytic functions associated with the Fuchsian operator $\OLdisk\la\te$. Each such function is approximated in practice by a polynomial Taylor truncation matching the first coefficients of the exact solution at the relevant singular point. The aim of this appendix is to develop the residual estimate that allows us to control the distance between an analytic function $f$ and its truncated Taylor approximation $f^{\rm ap}$.

\subsection{A residual estimate for inhomogeneous Fuchsian equations}

Throughout this section let $b>0$, let $p,q,g:[0,b]\to\C$ be analytic, and consider the Fuchsian operator
\begin{align*}
    \mathcal L f(x) := x^2 \pa_x^2 f(x) + xp(x) \pa_xf(x) + q(x)f(x)
\end{align*}
together with an analytic solution $f$ of
\begin{align}
\label{ode:appB:LFG}
    \mathcal L f = g \quad\text{on } [0,b].
\end{align}
Given an analytic approximation $f^{\rm ap}:[0,b]\to\C$ matching the first $N_0$ derivatives of $f$ at $0$, we measure its accuracy by the residual
\begin{align}
\label{def:appB:R}
    \xi(x) := \int_0^x\int_0^y \frac{g(s)-\mathcal L f^{\rm ap}(s)}{s^2} \ds\dy,
\end{align}
which is well defined and analytic on $[0,b]$ provided $g-\mathcal L f^{\rm ap}$ vanishes to order at least $2$ at $0$.

For $N_0\geq 2$, write
\begin{align*}
    \mathcal A_{N_0}([0,b])
    :=
    \{h \text{ analytic on } [0,b] : \pa_x^k h(0)=0 \text{ for every } 0\leq k<N_0\},
\end{align*}
and, for $h\in\mathcal A_{N_0}([0,b])$ written as $h=x^{N_0}u$, define
\begin{align*}
    \norm h_{\mathcal A_{N_0}} := \norm{u} _{L^\infty([0,b])}.
\end{align*}

\begin{lemma}
\label{lem:sec_ord:LS}
    Let $N_0 \geq 2$ and assume the contraction condition
    \begin{align}
    \label{appB:L}
        \ell:=\frac{\norm p _{L^\infty([0,b])}}{N_0}
        +
        \frac{\norm{D_x p-p-q}_{L^\infty([0,b])}}{N_0(N_0-1)}
        < 1,
    \end{align}
    where $D_x = x \pa_x$.
    
    Let $f$ be an analytic solution of~\eqref{ode:appB:LFG}, let $f^{\rm ap}:[0,b]\to\C$ be analytic with
    \begin{align*}
        {f^{\rm ap}}^{(k)}(0) = f^{(k)}(0)\qquad\text{for every }0\leq k\leq N_0-1,
    \end{align*}
    and let $\tau:=f-f^{\rm ap}$ and $\xi$ as in~\eqref{def:appB:R}. Then $\tau\in\mathcal A_{N_0}([0,b])$, and
    \begin{align}
    \label{appB:Ebound}
        \norm \tau _{\mathcal A_{N_0}([0,b])}
        \leq
        \frac{1}{1-\ell}\norm \xi _{\mathcal A_{N_0}([0,b])}.
    \end{align}
    In particular,
    \begin{subequations}
    \begin{align}
    \label{appB:Linfbound}
        \norm \tau _{L^\infty([0,b])}
        &\leq
        \frac{b^{N_0}}{1-\ell} \norm \xi _{\mathcal A_{N_0}([0,b])},
        \\
    \label{appB:Eprimebound}
        \norm{\pa_x \tau}_{L^\infty([0,b])}
        &\leq
        \frac{N_0  \ell  b^{N_0-1}}{1-\ell} \norm \xi _{\mathcal A_{N_0}([0,b])}
        +
        \norm{\pa_x \xi}_{L^\infty([0,b])}.
    \end{align}
    \end{subequations}
    
\end{lemma}

\begin{proof}
    The matching conditions imply that $\tau=f-f^{\rm ap}$ belongs to
    $\mathcal A_{N_0}([0,b])$. Define $T$ on this space by
    \begin{align*}
        Th(x):=-\int_0^x\int_0^y
        \frac{s p(s)\pa_x h(s)+q(s)h(s)}{s^2}\ds\dy.
    \end{align*}
    For $h\in\mathcal A_{N_0}([0,b])$, the numerator vanishes to order
    at least $N_0$, so the integrals converge, $Th$ is analytic, and
    $Th\in\mathcal A_{N_0}([0,b])$. Since $h(0)=\pa_x h(0)=0$,
    Taylor's integral formula gives
    \begin{align*}
        (I-T)h(x)=\int_0^x\int_0^y
        \frac{\mathcal Lh(s)}{s^2}\ds\dy.
    \end{align*}
    Apply this identity directly to $\tau$. The equation
    $\mathcal L\tau=g-\mathcal Lf^{\rm ap}$ and the definition of $\xi$
    yield
    \begin{align}
        \label{eqn:I-Tt=xi}
        (I-T)\tau=\xi.
    \end{align}

    Writing $\tau(x)=x^{N_0}u(x)$, a direct computation gives
    \begin{align*}
        sp(s)\pa_x\tau(s)+q(s)\tau(s)
        =
        s^{N_0}\big(N_0 p(s)+q(s)\big)u(s)
        +
        s^{N_0+1}p(s)\pa_x u(s).
    \end{align*}

    Hence, we have
    \begin{align*}
        T\tau(x) &= -\int_0^x\int_0^y \frac{s p(s) \pa_x \tau(s)+q(s)\tau(s)}{s^2} \ds\dy\\
        &= -\int_0^x\int_0^y
        s^{N_0-2}\big(N_0 p(s)+q(s)\big)u(s)
        +
        s^{N_0-1}p(s)\pa_x u(s)
         \ds\dy\\
        &=
        -\int_0^x y^{N_0-1}p(y)u(y)\dy
        +
        \int_0^x\int_0^y s^{N_0-2}\big(s\pa_x p(s)-p(s)-q(s)\big)u(s) \ds\dy,
    \end{align*}
    where we have integrated by parts.

    Bounding each integrand by its $L^\infty$ norm and using
    \begin{align*}
        \int_0^x y^{N_0-1}\dy = \frac{x^{N_0}}{N_0},
        \qquad
        \int_0^x\int_0^y s^{N_0-2} \ds\dy = \frac{x^{N_0}}{N_0(N_0-1)},
    \end{align*}
    we obtain
    \begin{align*}
        \frac{|T\tau(x)|}{x^{N_0}} \leq \ell  \norm{u} _{L^\infty([0,b])}\qquad
        \text{for all }x\in(0,b],
    \end{align*}
    hence
    \begin{align*}
        \norm{T\tau}_{\mathcal A_{N_0}([0,b])}\leq \ell \norm \tau _{\mathcal A_{N_0}([0,b])}.
    \end{align*}

    By~\eqref{eqn:I-Tt=xi} and the preceding estimate,
    \begin{align*}
        \norm\tau_{\mathcal A_{N_0}}
        \leq\norm{T\tau}_{\mathcal A_{N_0}}+\norm\xi_{\mathcal A_{N_0}}
        \leq\ell\norm\tau_{\mathcal A_{N_0}}+\norm\xi_{\mathcal A_{N_0}}.
    \end{align*}
    Rearranging and using $\ell<1$ proves~\eqref{appB:Ebound}.
    The bound~\eqref{appB:Linfbound} follows from
    $\norm\tau_{L^\infty([0,b])}\leq b^{N_0}\norm\tau_{\mathcal A_{N_0}}$.

    For the derivative bound~\eqref{appB:Eprimebound}, differentiating~\eqref{eqn:I-Tt=xi} yields
    \begin{align*}
        \pa_x \tau(x)
        =
        -x^{N_0-1}p(x)u(x)
        +
        \int_0^x s^{N_0-2}\big(s\pa_xp(s)-p(s)-q(s)\big)u(s) \ds
        +
        \pa_x\xi(x).
    \end{align*}
    Bounding each integrand by its $L^\infty$ norm and using
    \begin{align*}
        x^{N_0-1} \leq b^{N_0-1},
        \qquad
        \int_0^x s^{N_0-2} \ds = \frac{x^{N_0-1}}{N_0-1} \leq \frac{b^{N_0-1}}{N_0-1},
    \end{align*}
    together with the identity $\norm p _{L^\infty([0,b])} + \frac{\norm{D_x p-p-q}_{L^\infty([0,b])}}{N_0-1} = N_0 \ell$, gives
    \begin{align*}
        \norm{\pa_x\tau}_{L^\infty([0,b])} \leq N_0  \ell  b^{N_0-1} \norm{u} _{L^\infty([0,b])}
        +
        \norm{\pa_x \xi}_{L^\infty([0,b])},
    \end{align*}
    and the bound~\eqref{appB:Eprimebound} follows from $\norm{u} _{L^\infty([0,b])}=\norm \tau _{\mathcal A_{N_0}([0,b])}\leq (1-\ell)^{-1}\norm \xi _{\mathcal A_{N_0}([0,b])}$.
\end{proof}

\subsection{Application to the entries of $\Mcrit$}
\label{sec:app:apply}

The matrix $\Mcrit$ defined in Definition~\ref{def:Mat} is assembled from the
values at $x_0=\frac12$ of the functions
\begin{align*}
    \Fholhte1{1}\pm
    \andspace
    \FholpAff i {\te^\pm} k,
    \quad i\in\{1,2\},\quad k\in\{0,1\},
\end{align*}
together with their first derivatives $\pa_z$.  For each such global
function $F$ we pass to the local coordinate $x$ at the relevant
singular point. Near $z=1$ set $f(x):=F(1-x)$ (so $x=1-z$), and near $z=0$ set
$f(x):=F(x)$ (so $x=z$).  Each $f$ is analytic near $x=0$ and satisfies a local
Fuchsian equation $\mathcal{L}_{z_0}f=g$ identified below.  Writing $b_0:=\frac12$,
the global midpoint $z=x_0=\frac12$ corresponds to $x=b_0$ in both charts, and the matrix entries of Definition~\ref{def:Mat} relate to
$f$ via
\begin{align}\label{eq:local-global}
    F(x_0) = f(x_*),
    \qquad
    \pa_z F(x_0) = \sigma\pa_x f(x_*),
    \qquad
    x_* :=
    b_0,
    \qquad
    \sigma :=
    \begin{cases}
       -1 & \text{near }z=1,\\
       +1 & \text{near }z=0.
    \end{cases}
\end{align}

\paragraph{Local equations at $z=0$.}
Set $x:=z$ and write the
local operator at $z=0$ as
\begin{align}
    \mathcal{L}_0 f(x)
    :=
    x^2\pa_x^2 f(x)+x \ploc{0} (x)\pa_x f(x)+\qloc{0}(x)f(x),
\end{align}
where $\ploc{0}$ and $\qloc{0}$ are the coefficients of the Frobenius normal form at $z_0=0$, written in the local variable $x=z$. From $\Ploc0,\Qloc0$ in the proof of Lemma~\ref{lem:loc:0} one obtains the rational functions
\begin{align*}
    \ploc{0}(x) = \Ploc0(x) = \frac{\Pglo (x)}{x^2-1},
    \qquad
    \qloc{0}(x) = \Qloc0(x) = \frac{\Qglo (x)}{x^2-1},
\end{align*}
which are analytic on $|x|<1$, with nearest poles at $x=\pm1$.

The norms $\bignorm{p}_{L^\infty}$ and $\bignorm{D_x p-p-q}_{L^\infty}$ entering the contraction condition~\eqref{appB:L} are bounded by hand in the following lemma.

\begin{lemma}
\label{lem:sup:00}
    It holds for $\la\in\Slambd$ and $\te\in(0,1)$ that
    \begin{align*}
        \bignorm{
        \ploc0
        }_{L^\infty([0,\frac12])}
        <4,
        \andspace
        \bignorm{D_x \ploc{0} - \ploc{0} - \qloc{0} 
        }_{L^\infty([0,\frac12])}
        <8.
    \end{align*}

    Consequently, the contraction condition~\eqref{appB:L} is satisfied for all $N_0 \geq 6$.
\end{lemma}

\begin{proof}
    We have
    \begin{align*}
        \ploc0(x)&=
        1+2\te
        -
        (1+\la)\frac{x}{1+x}
        +
        (1-\la)\frac{x}{1-x}
        = 
        1+2\te
        +
        2x\frac{x-\la}{1-x^2}.
    \end{align*}

    Taking the supremum yields
    \begin{align*}
        \bignorm{\ploc0}_{L^\infty([0,\frac12])} \leq 
        1+2\te
        +
        \frac{|x-\la|}{1-\frac{1}{4}} \leq 3 + 
        \frac45
        <4,
    \end{align*}
    Here we have used that $\la\in\{0.56,0.6\}$ so that $|x-\la| \leq \frac{3}{5}$ for all $x\in[0,\frac12]$.

    The second bound in the statement is derived in the same way. A direct computation yields
    \begin{align}
        x \pa_x \ploc{0}(x) - \ploc{0}(x) - \qloc{0}(x) 
        =\frac{
        r_{\la,\te}(x)
        }
        {(x^2 - 1)^2}.
    \end{align}
    where the function $r_{\la,\te}$ is given by
    \begin{align*}
        r_{\la,\te}(x):=
        -x^4 \te^2 - (2 \la \te+ 4 \la) x^3 + (2 \te^2 + 2 \te + 4)x^2 + 2x \la \te - \te^2 - 2\te.
    \end{align*}
    Notice that each monomial of $r_{\la,\te}(x)$ has constant sign for $\la,\te\geq0$ and $x\geq0$. We then have for all $x\in[0,\frac12]$, $\te\in(0,1)$ and $\la\in\Slambd$ that
    \begin{align*}
        (2 \te^2 + 2 \te + 4)x^2 + 2x \la \te \in
        [0,3]\andspace
        -x^4 \te^2 - (2 \la \te+ 4 \la) x^3 - \te^2 - 2\te \in[-4,0].
    \end{align*}
    Therefore, we have
    \begin{align*}
        \bignorm{D_x \ploc{0} - \ploc{0} - \qloc{0} 
        }_{L^\infty([0,\frac12])} \leq \frac{16}{9} 
        \sup_{x\in[0,\frac12]}
        |r_{\la,\te}(x)| \leq 8.
    \end{align*}

    Finally, we check that the contraction condition~\eqref{appB:L} is satisfied for all $N_0 \geq 6$:
    \begin{align*}
        \ell \leq \frac{4}{N_0} + \frac{8}{N_0(N_0-1)} \leq \frac{28}{30}<1.
    \end{align*}
\end{proof}

\paragraph{Local equations at $z=1$.}
Set $x:=1-z$ and write the local operator at $z=1$ as
\begin{align}
\label{eq:op-m1}
    \mathcal{L}_{1} f(x)
    :=
    x^2 \pa_x^2 f(x)+x\ploc{1}(x)\pa_x f(x)+\qloc{1}(x)f(x),
\end{align}
where $\ploc{1}$ and $\qloc{1}$ are the coefficients of the Frobenius normal form~\eqref{ode:fuchs2pqhom} at $z_0=1$, written in the local variable $x$. Constructing them from~\eqref{def:PQtl} as in the proof of Lemma~\ref{lem:loc:1} gives the rational functions
\begin{align*}
    \ploc{1}(x) = \frac{\Pglo(1-x)}{(x-1)(x-2)},
    \qquad
    \qloc{1}(x) = \frac{x \Qglo(1-x)}{(x-1)^2(x-2)},
\end{align*}
which are analytic on $|x|<1$, with nearest poles at $x=1$ and $x=2$.

The corresponding bounds at $z=1$ are as follows.

\begin{lemma}
\label{lem:sup:-1}
    It holds for $\la\in\Slambd$ and $\te\in(0,1)$ that
    \begin{align*}
        \bignorm{
        \ploc{1}
        }_{L^\infty([0,\frac12])} \leq 8
        \andspace
        \bignorm{
        D_x \ploc{1} - \ploc{1} - \qloc{1}
        }_{L^\infty([0,\frac12])}
        \leq
        16.
    \end{align*}
    Consequently, the contraction condition~\eqref{appB:L} is satisfied for all $N_0 \geq 10$.
\end{lemma}

\begin{proof}
    The first estimate is derived as follows
    \begin{align*}
        \bignorm{
        \ploc{1}
        }_{L^\infty([0,\frac12])}&=
        \sup_{x\in[0,\frac12]}
        \left|
        \frac{(2\te-1)(x-1)^2-2\la (x-1)-(2\te+1)}{(x-1)(x-2)}\right|\\
        & \leq 2
        \sup_{x\in[0,\frac12]}
        \big|(2\te-1)(x-1)^2-2\la (x-1)-(2\te+1)\big|.
    \end{align*}
    As in the proof of Lemma~\ref{lem:sup:00}, we enclose each monomial separately. Let $x\in[0,\frac12]$. Then
    \begin{align*}
        (2\te-1)(x-1)^2
        \in[-1,1],\qquad
        -2\la (x-1)\in 
        \left[\tfrac12,\tfrac65\right],\qquad
        -(2\te+1)\in [-3,-1].
    \end{align*}
    Hence, we obtain
    \begin{align*}
        \bignorm{
        \ploc{1}
        }_{L^\infty([0,\frac12])}
         \leq 2
        \left|
            [-1,1] 
            + 
            \big[\tfrac12,\tfrac65\big] 
            + 
            [-3,-1]
        \right|
        \leq 
        2\big|[-4,2]\big| \leq 8.
    \end{align*}

    In order to compute the second bound, we write
    \begin{align*}
        x\pa_x \ploc{1}(x)
        -
        \ploc{1}(x)
        -
        \qloc{1}(x)=
        \frac{r_{\la,\te}(x)}
        {(x - 2)^2 (x - 1)^2},
    \end{align*}
    where $r_{\la,\te}$ is given by
    \begin{align*}
        r_{\la,\te}(x)&:=
        -x^4 \te^2 + 
        (4 \te^2 + 2 \te \la + 4 \la) x^3\\ &\qquad+ 
        (- 4 \te^2 - 6 \te \la + 2 \te - 12 \la + 4)x^2 +
        (4 \te \la - 4 \te + 12 \la - 8) x
        - 4 \la + 4.
    \end{align*}

    Hence, for every $x\in[0,\frac12]$ we have
    \begin{align*}
        \begin{gathered}
        -x^4 \te^2 \in \left[-\tfrac{1}{16},0\right],
        \qquad
        (4 \te^2 + 2 \te \la + 4 \la)x^3
        \in[0,1],\qquad
        (- 4 \te^2 - 6 \te \la + 2 \te - 12 \la + 4)x^2\in 
        [-3,0],
        \\
        (4 \te \la - 4 \te + 12 \la - 8) x\in
        [-2,0],
        \qquad
        - 4 \la + 4 \in \left[\tfrac85,2\right].
        \end{gathered}
    \end{align*}
    The fourth bound holds because $4\te\la-4\te+12\la-8=4\te(\la-1)+4(3\la-2)\in[-4,0]$ for $\te\in[0,1]$, $\la\in\Slambd$ (both summands are non-positive since $\la<1$ and $3\la<2$), so the product with $x\leq\frac12$ lies in $[-2,0]$. Adding all the intervals yields
    \begin{align*}
        r_{\la,\te}(x) \in [-4,3]\qquad\text{for all }x\in[0,\tfrac12].
    \end{align*}

    Since $|x-2|\geq 1$ and $|x-1|\geq\frac12$ on
    $[0,\frac12]$, we obtain
    \begin{align*}
        \bignorm{
        D_x \ploc{1}
        -
        \ploc{1}
        -
        \qloc{1}
        }_{L^\infty([0,\frac12])}
        \leq
        \frac{4}{1^2(1/2)^2}
        = 16.
    \end{align*}
    The contraction condition then follows from
    \begin{align*}
        \ell
        \leq
        \frac{8}{N_0}
        +
        \frac{16}{N_0(N_0-1)}
        \leq
        \frac{8}{10}+\frac{16}{90}
        = \frac{44}{45}
        <1
    \end{align*}
    for every $N_0 \geq 10$.
\end{proof}

\subsection{Rigorous control of the tails}

The contraction constant in~\eqref{appB:L} is controlled by the hand estimates of Lemmas~\ref{lem:sup:00} and~\ref{lem:sup:-1}. 
The code does not recompute those suprema from Taylor coefficients. 
What remains to be computed for each matrix function is the residual size $\bignorm{\xi}_{\mathcal{A}_{N_0}}$ and the derivative size $\bignorm{\pa_x\xi}_{L^\infty}$. 
These are infinite sums, and this subsection records the Cauchy estimates used by the code to bound their tails after the exact coefficients have been summed up to order $N_{\rm res}$.

We fix throughout $N_0\geq 2$, $N_{\rm res}\geq N_0$, $b_0=\frac12$, and a Cauchy radius $\rho\in(b_0,1)$. 
We set $r:=\frac{b_0}\rho\in(0,1)$, which is the geometric
ratio governing all tail bounds below.

We first record explicit sup-bounds $M_p(\rho)$ and $M_q(\rho)$ at each singular point.
The constants $c^p_j$ and $c^q_j$ in Lemmas~\ref{lem:cau:Mbound:00} and~\ref{lem:cau:Mbound:m1} below are the partial fraction coefficients of
\begin{align*}
    \frac{\Pglo(z)}{z(z^2-1)}
    \andspace
    \frac{\Qglo(z)}{z^2(z^2-1)}.
\end{align*}
Explicitly, $c^p_0,c^p_1,c^p_2$ are the coefficients of $\frac1z,\frac1{z-1},\frac1{z+1}$, respectively, in the first expansion, and $c^q_0,c^q_1,c^q_2,c^q_3$ are the coefficients of $\frac1z,\frac1{z^2},\frac1{z-1},\frac1{z+1}$, respectively, in the second. See~\eqref{PQfracform}.
The local formulas use $z=x$ near $z=0$ and $z=1-x$ near $z=1$.

In each case, with $z_0\in\{0,1\}$ the corresponding singular point, $M_p(\rho)$ and $M_q(\rho)$ denote explicit upper bounds for 
\begin{align*}
    \max_{|x|\leq \rho}|\ploc{z_0}(x)| = 
    \bignorm{\Ploc{z_0}}_{L^\infty(D_\rho(z_0))}
    \andspace
    \max_{|x|\leq \rho}|\qloc{z_0}(x)| = 
    \bignorm{\Qloc{z_0}}_{L^\infty(D_\rho(z_0))}.
\end{align*}

\begin{lemma}
\label{lem:cau:Mbound:00}
    Fix $\rho\in(0,1)$. In the local coordinate $x=z$, the coefficients of
    $\mathcal{L}_0$ can be written as
    \begin{align*}
        \ploc{0}(x) = c^p_0 + x\Big(\frac{c^p_1}{x-1}+\frac{c^p_2}{x+1}\Big)
        \andspace
        \qloc{0}(x) = c^q_1 + c^q_0 x
                   + x^2\Big(\frac{c^q_2}{x-1}+\frac{c^q_3}{x+1}\Big).
    \end{align*}
    Since $|x\pm 1|\geq 1-\rho$ for $|x|\leq\rho<1$,
    \begin{align*}
        M_p(\rho) &= |c^p_0|+\frac{\rho (|c^p_1|+|c^p_2|)}{1-\rho}
        \andspace
        M_q(\rho) = |c^q_1|+|c^q_0| \rho+\frac{\rho^2(|c^q_2|+|c^q_3|)}{1-\rho}
    \end{align*}
    are upper bounds for $\max_{|x|\leq\rho}|\ploc{0}(x)|$ and
    $\max_{|x|\leq\rho}|\qloc{0}(x)|$, respectively.
\end{lemma}

\begin{proof}
    For $|x|\leq \rho <1$ we have
    \begin{align*}
        |x\pm 1| \geq 1-\rho.
    \end{align*}
    Applying the triangle inequality to the decompositions
    \begin{align*}
        \ploc{0}(x) = c^p_0 + x\Big(\frac{c^p_1}{x-1}+\frac{c^p_2}{x+1}\Big)
        \andspace
        \qloc{0}(x) = c^q_1 + c^q_0 x
                   + x^2\Big(\frac{c^q_2}{x-1}+\frac{c^q_3}{x+1}\Big)
    \end{align*}
    gives
    \begin{align*}
        |\ploc{0}(x)|
        &\leq |c^p_0|+\frac{|x|(|c^p_1|+|c^p_2|)}{1-\rho}
        \leq |c^p_0|+\frac{\rho(|c^p_1|+|c^p_2|)}{1-\rho},
        \\
        |\qloc{0}(x)|
        &\leq |c^q_1|+|c^q_0| |x|
        +\frac{|x|^2(|c^q_2|+|c^q_3|)}{1-\rho}
        \leq |c^q_1|+|c^q_0| \rho
        +\frac{\rho^2(|c^q_2|+|c^q_3|)}{1-\rho}.
    \end{align*}
    Taking the maximum over $|x|\leq \rho$ gives the claimed bounds.
\end{proof}

\begin{lemma}
\label{lem:cau:Mbound:m1}
In the local coordinate $x=1-z$ of~\eqref{eq:op-m1}, the coefficients of
$\mathcal L_1$ are
\begin{align*}
 \ploc{1}(x)&=c^p_1+x\Big(\frac{c^p_0}{x-1}+\frac{c^p_2}{x-2}\Big),\\
 \qloc{1}(x)&=-c^q_2x+x^2\Big(-\frac{c^q_0}{x-1}
                  +\frac{c^q_1}{(x-1)^2}-\frac{c^q_3}{x-2}\Big).
\end{align*}
For $|x|\leq\rho<1$, the bounds $|x-1|\geq1-\rho$ and
$|x-2|\geq2-\rho$ give
\begin{align*}
 M_p(\rho)&=|c^p_1|+\rho\Big(\frac{|c^p_0|}{1-\rho}
                            +\frac{|c^p_2|}{2-\rho}\Big),\\
 M_q(\rho)&=|c^q_2|\rho+\rho^2\Big(\frac{|c^q_0|}{1-\rho}
          +\frac{|c^q_1|}{(1-\rho)^2}+\frac{|c^q_3|}{2-\rho}\Big)
\end{align*}
as upper bounds for $\max_{|x|\leq\rho}|\ploc1(x)|$ and
$\max_{|x|\leq\rho}|\qloc1(x)|$, respectively.
\end{lemma}
\begin{proof}
Apply the triangle inequality to the displayed partial fraction
representations and use the lower bounds for $|x-1|$ and $|x-2|$.
\end{proof}

\begin{lemma}
\label{lem:forcing:Mbound}
Let $\te\in(0,1)$ and $\rho\in(0,1)$. For $j=1,2$, the affine term
$\Affterm j\te$ in~\eqref{holsysa} gives the following forcing terms in the
local equations at $z=0$ and $z=1$, respectively:
\begin{align*}
 g_{0,j}(x)&=\frac{\Affterm j\te(x)}{x^2-1},
 &g_{1,j}(x)&=-\frac{x\Affterm j\te(1-x)}{(1-x)^2(2-x)}.
\end{align*}
Define
\begin{align*}
M_{g_{0,1}}(\rho)&:=\frac{\te(2-\te)\rho}{1-\rho^2},&
M_{g_{1,1}}(\rho)&:=\frac{\te(2-\te)\rho}{(1-\rho)(2-\rho)}\\
M_{g_{0,2}}(\rho)&:=\frac{(1-\te)(1+\te)}{1-\rho^2},&
M_{g_{1,2}}(\rho)&:=\frac{(1-\te)(1+\te)\rho}{(1-\rho)^2(2-\rho)}.
\end{align*}
For each $k\in\{0,1\}$ and $j\in\{1,2\}$,
\begin{align*}
\sup_{|x|\leq\rho}|g_{k,j}(x)|\leq M_{g_{k,j}}(\rho).
\end{align*}

For the homogeneous solution, take $M_g=0$.
\end{lemma}
\begin{proof}
The formulas follow from~\eqref{holsysa} by setting $z=x$ or $z=1-x$ and
dividing by the leading coefficient. The bounds follow by the triangle
inequality.
\end{proof}

We now bound the residual tail, including the normalized forcing.

\begin{lemma}
\label{lem:tails:xi}
Let $f^{\rm ap}(x)=\sum_{m=0}^{N_0-1}f_m^{\rm ap}x^m$, with
$N_0\geq2$, $N_{\rm res}\geq N_0$, $b_0=\frac12$, and
$\rho\in(b_0,1)$. Let $M_p,M_q$ be the bounds at radius $\rho$ from
Lemma~\ref{lem:cau:Mbound:00} or Lemma~\ref{lem:cau:Mbound:m1}, according to
the local chart, and let $M_g$ be the corresponding bound from
Lemma~\ref{lem:forcing:Mbound}. Define
\begin{align*}
 C_{\rm val}&:=\sum_{m=0}^{N_0-1}|f_m^{\rm ap}|\rho^m,
 &C_{\rm der}&:=\sum_{m=1}^{N_0-1}m|f_m^{\rm ap}|\rho^m.
\end{align*}
Set $r:=\frac{b_0}{\rho}$. 
Then, for $k>N_{\rm res}$,
\begin{align*}
 |\xi_k|&\leq\frac{M_pC_{\rm der}+M_qC_{\rm val}+M_g}{k(k-1)}\rho^{-k},
 &k|\xi_k|&\leq\frac{M_pC_{\rm der}+M_qC_{\rm val}+M_g}{k-1}\rho^{-k},
\end{align*}
and consequently
\begin{align*}
 \sum_{k>N_{\rm res}}|\xi_k|b_0^{k-N_0}
 &\leq\frac{M_pC_{\rm der}+M_qC_{\rm val}+M_g}{N_{\rm res}(N_{\rm res}+1)b_0^{N_0}}
       \frac{r^{N_{\rm res}+1}}{1-r},\\
 \sum_{k>N_{\rm res}}k|\xi_k|b_0^{k-1}
 &\leq\frac{M_pC_{\rm der}+M_qC_{\rm val}+M_g}{N_{\rm res}b_0}
       \frac{r^{N_{\rm res}+1}}{1-r}.
\end{align*}
\end{lemma}

\begin{proof}
For $k>N_{\rm res}\geq N_0$, the coefficient of
$x^2\partial_x^2f^{\rm ap}$ vanishes. Set
\begin{align*}
 d_k:=\big(g-pD_xf^{\rm ap}-qf^{\rm ap}\big)_k.
\end{align*}
Cauchy's estimates give
\begin{align*}
 |d_k|
 \leq(M_pC_{\rm der}+M_qC_{\rm val}+M_g)\rho^{-k}.
\end{align*}
Since $\xi_k=\frac{d_k}{k(k-1)}$ by~\eqref{def:appB:R},
\begin{align*}
 \sum_{k>N_{\rm res}}|\xi_k|b_0^{k-N_0}
 &\leq
 \sum_{k>N_{\rm res}}
\frac{M_pC_{\rm der}+M_qC_{\rm val}+M_g}{k(k-1)}
\frac{r^k}{b_0^{N_0}}
\leq
\frac{M_pC_{\rm der}+M_qC_{\rm val}+M_g}
 {N_{\rm res}(N_{\rm res}+1)b_0^{N_0}}
 \frac{r^{N_{\rm res}+1}}{1-r},\\
 \sum_{k>N_{\rm res}}k|\xi_k|b_0^{k-1}
 &\leq
\sum_{k>N_{\rm res}}
\frac{M_pC_{\rm der}+M_qC_{\rm val}+M_g}{k-1}
\frac{r^k}{b_0}
\leq
\frac{M_pC_{\rm der}+M_qC_{\rm val}+M_g}
 {N_{\rm res}b_0}
 \frac{r^{N_{\rm res}+1}}{1-r}.
\end{align*}
\end{proof}

\subsection{Certified bounds for the matrix entries}

We now combine the residual estimate with the inputs established above. The
contraction constants come from the hand estimates of
Lemmas~\ref{lem:sup:00} and~\ref{lem:sup:-1}. The residual norms are computed
from the exact coefficients up to $N_{\rm res}$, plus the Cauchy tail of
Lemma~\ref{lem:tails:xi}.

\begin{lemma}
\label{lem:delta-cert}
Fix $N_0\geq 2$ and assume that the corresponding contraction constant
$L_f$ satisfies $L_f<1$. For every local representative $f$ in the sense
of~\eqref{eq:local-global} of each function among
$\Fholhte1{1}{\pm}$ and $\FholpAff{j}{\te^\pm}{k}$,
$j\in\{1,2\}$, $k\in\{0,1\}$, let $f^{\rm ap}$ be its order-$N_0$
Taylor approximation and let $\xi_f$ be the residual defined
by~\eqref{def:appB:R}. Then there are explicitly computable bounds
$\de_f,\de'_f>0$, valid for each fixed $\la\in\Slambd$ and for
$\te$ in the certified interval, such that
\begin{align}
    \bigl|f(b_0) - f^{\rm ap}(b_0)\bigr| \leq \de_f,
    \qquad
    \bigl|\pa_x f(b_0) - \pa_x f^{\rm ap}(b_0)\bigr| \leq \de'_f,
\end{align}
with $b_0=\frac12$. Explicitly,
\begin{align}\label{eq:delta-explicit}
    \de_f
    =
    \frac{b_0^{N_0}}{1-L_f} \bignorm{\xi_f}_{\mathcal{A}_{N_0}},
    \qquad
    \de'_f
    =
    \frac{N_0 L_f b_0^{N_0-1}}{1-L_f} \bignorm{\xi_f}_{\mathcal{A}_{N_0}}
    +
    \bignorm{\pa_x \xi_f}_{L^\infty([0,b_0])},
\end{align}
where $L_f$ is the contraction constant~\eqref{appB:L} for the corresponding
local operator.
\end{lemma}

\begin{proof}
Each local representative $f$ is analytic near $x=0$ and satisfies
$\mathcal{L}_{z_0}f=g$ with $\mathcal{L}_{z_0}\in\{\mathcal{L}_{1},
\mathcal{L}_0\}$, as described in Section~\ref{sec:app:apply}. The
approximation $f^{\rm ap}$ matches the first $N_0$ Taylor coefficients of
$f$ by construction, so Lemma~\ref{lem:sec_ord:LS} applies. Its value and
derivative estimates give~\eqref{eq:delta-explicit}. The condition
$L_f<1$ is guaranteed by Lemmas~\ref{lem:sup:00} and~\ref{lem:sup:-1} for
the values of $N_0$ used in the proof. The quantities
$\bignorm{\xi_f}_{\mathcal{A}_{N_0}}$ and
$\bignorm{\pa_x\xi_f}_{L^\infty}$ are then obtained from the exact finite
sums and the Cauchy tail of Lemma~\ref{lem:tails:xi}.
\end{proof}

\begin{corollary}
\label{cor:matrix-enclose}
    Let the local functions defining the entries of $\Mcrit$ be ordered as in
    Definition~\ref{def:Mat}. For each entry of $\Mcrit$, let
    $\varepsilon_{jk}$ be the sum of the corresponding bounds $\de_f$ or
    $\de'_f$ from Lemma~\ref{lem:delta-cert}, according to whether the entry
    involves values or derivatives and whether contributions from both singular
    points are present. Then the true matrix $\Mcrit$ is enclosed
    entrywise by the truncated matrix $\Mcritap$ enlarged by these
    radii:
    \begin{align*}
        \bigl|(\Mcrit)_{jk}-(\Mcritap)_{jk}\bigr| \leq \varepsilon_{jk}.
    \end{align*}
\end{corollary}

\section{Computer-assisted proofs}
\label{sec:appendix:impl}

This appendix describes the computer-assisted proofs of Lemmas~\ref{lem:negposjump} and~\ref{lem:h1:core:cert}. We explain how the programs represent the quantities appearing in these lemmas, describe the main routines, and give the verification procedure for each result. The analytic estimates used to control the errors in Lemma~\ref{lem:negposjump} are established in Appendices~\ref{sec:appendix:frobenius} and~\ref{sec:appendix:errors}.

The source code is available at \getvar{linkdgreg}. The programs are written in Python for SageMath~10.9. Symbolic computations use exact rational arithmetic. Numerical evaluations use Arb ball arithmetic through \verb|RealBallField| at $\getvar{proof_bits}$ bits of precision. A real ball represents an interval containing the exact quantity, and arithmetic operations preserve this inclusion. A numerical inequality is accepted only when the resulting enclosure proves it on the entire parameter interval under consideration.

\subsection{Supplementary code and computational data}

The supplementary files are organized according to the two lemmas being verified. The data needed for the proofs are constructed by the programs from the equations and rational parameters specified below.

\paragraph{Files for the determinant certificate.}
The construction of the matrix $\Mcrit$ and its residuals is implemented in \verb|local_fuchs.py| and \verb|cert_residuals.py|. The modules \verb|cert_eval.py| and \verb|cert_det_post.py| enclose the truncation errors and the determinant. The covering of the required parameter set is implemented in \verb|cert_proof.py|. Finally, \verb|run_proof.py| runs the symbolic checks in \verb|verify_residuals.py| and then the determinant certificate.

For this certificate, the parameters are
\begin{align*}
    N&=\getvar{prime_split_N}, &
    N_0&=\getvar{proof_n0}, &
    N_{\rm res}&=\getvar{proof_nres},\\
    b_0&=\getvar{proof_bzero}, &
    \rho&=\getvar{proof_rho}, &
    \texttt{center}&=\getvar{proof_center}.
\end{align*}
Here $N$ is the prime threshold in Lemma~\ref{lem:cover}, $N_0$ is the number of Taylor coefficients retained in each local approximation, and $N_{\rm res}$ is the largest residual index computed explicitly. The point $b_0$ is the local coordinate of the matching point, and $\rho$ is the radius used for the Cauchy estimates. The two values of $\la$ are $\getvar{low_bnd_eig}$ and $\getvar{upp_bnd_eig}$.

\paragraph{Local solution data.}
For each fixed rational value of $\la$, the local Taylor coefficients are elements of
\begin{align*}
    K=\mathrm{Frac}(\Q[\te]).
\end{align*}
The program uses the same global coordinate $z$ as the manuscript. The keys \verb|p1| and \verb|00| correspond to $z=1$ and $z=0$, with local coordinates $x=1-z$ and $x=z$, respectively. Both charts reach the matching point at $x=\frac12$. The keys \verb|the| and \verb|eht| respectively represent the branches $\te$ and $1-\te$.

For each chart and branch, \verb|all_res_sym()| returns a block containing the coefficients of the approximations $f^{\rm ap}$, the coefficients $\xi_k$ of their residuals, and the partial fraction coefficients \verb|term_pp| and \verb|term_qq|. The function keys \verb|phi1|, \verb|fp1|, and \verb|fp2| identify the normalized homogeneous solution and the two particular solutions. Only \verb|fp1| and \verb|fp2| are needed at $z=0$.

Each approximation has the form
\begin{align*}
    f^{\rm ap}(x)=\sum_{k=0}^{N_0-1}f_k x^k.
\end{align*}
For each such function, the program computes the finite sequence $(\xi_0,\ldots,\xi_{N_{\rm res}})$ exactly over $K$. Its entries with $k<N_0$ vanish by construction. The full residual generally has infinitely many nonzero coefficients. The contribution of the coefficients with $k>N_{\rm res}$ is bounded separately by Lemma~\ref{lem:tails:xi}.

\paragraph{Files for the $H^1$ certificate.}
The modules \verb|cert_near_0.py| and \verb|cert_away_0.py| verify the two parts of Lemma~\ref{lem:h1:core:cert}. Both use the interval matrix test \verb|certify_maxeig_below()| from \verb|cert_bloch.py|. The driver \verb|run_h1_metrics.py| calls both coverings with $r=\frac{11}{10}$ and $200$ bits of precision. It constructs the required matrices during the run and verifies every accepted interval. The two coverings use $m=8$ on $\big(0,\frac{13}{40}\big]$ and $m=4$ on $[\frac{13}{40},\frac12]$.

\subsection{Objects, methods and routines}

\paragraph{Construction of the local solutions.}
The routines in \verb|local_fuchs.py| construct the normalized local equations and solve their coefficient recurrences. The global partial fraction identities are
\begin{subequations}
\label{PQfracform}
\begin{align}
    \frac{\Pglo(z)}{z(z^2-1)}
    &=\frac{1+2\te}{z}+\frac{\la-1}{z-1}-\frac{\la+1}{z+1},\\
    \frac{\Qglo(z)}{z^2(z^2-1)}
    &=-\frac{2\la\te}{z}-\frac{1-\te^2}{z^2}
      +\frac{1-\te+\la\te}{z-1}-\frac{1-\te-\la\te}{z+1}.
\end{align}
\end{subequations}
The affine terms satisfy
\begin{align*}
    \frac{\Affterm1\te(z)}{z^2(z^2-1)}
    &=-\frac{\te(2-\te)}{2}
      \left(\frac2z-\frac1{z-1}-\frac1{z+1}\right),\\
    \frac{\Affterm2\te(z)}{z^2(z^2-1)}
    &=-\frac{1-\te^2}{2}
      \left(\frac2{z^2}-\frac1{z-1}+\frac1{z+1}\right).
\end{align*}
The routines \verb|default_parameter_matrices()| and \verb|specialize_terms()| produce these coefficients for each branch. The routines \verb|local_coeffs_at_p1()| and \verb|local_coeffs_at_0()| then construct the corresponding normalized local equations. The signs of the affine terms agree with Definition~\ref{def:Mat}.

The homogeneous approximation is computed by \verb|hom_sol_fro_const()| using the Frobenius recurrence. The particular approximations are computed directly by \verb|particular_sol_r1_is_0_direct()| at $z=1$ and \verb|particular_sol_ordinary_direct()| at $z=0$. These routines use the coefficient recurrence of the local equation, with constant coefficient zero for the particular solutions at $z=1$. The routine \verb|symbolic_matrix_A()| assembles the values and derivatives at the matching point in the order prescribed by Definition~\ref{def:Mat}. Global derivatives are obtained as $-\partial_x$ in the first chart and $\partial_x$ in the second chart.

\paragraph{Residual bounds.}
The module \verb|cert_eval.py| implements the estimates of Appendix~\ref{sec:appendix:errors} for inputs enclosed in real balls. Its main routines are as follows.
\begin{itemize}
    \item \verb|L_bound()| computes the rational upper bound for the contraction constant in~\eqref{appB:L}, using Lemmas~\ref{lem:sup:00} and~\ref{lem:sup:-1}.
    \item \verb|rational_M_00()| and \verb|rational_M_p1()| compute $M_p(\rho)$ and $M_q(\rho)$ from Lemmas~\ref{lem:cau:Mbound:00} and~\ref{lem:cau:Mbound:m1}. The routine \verb|forcing_M()| computes the forcing bound $M_g$ of Lemma~\ref{lem:forcing:Mbound}.
    \item \verb|fapprox_norms()| computes $C_{\rm val}$ and $C_{\rm der}$ from the coefficients of $f^{\rm ap}$. The routine \verb|xi_tails()| combines these with $M_p$, $M_q$, and $M_g$ to bound the residual tails in Lemma~\ref{lem:tails:xi}.
    \item \verb|norm_AN0()| and \verb|norm_deriv()| compute the finite sums involving $\xi_k$ for $N_0\leq k\leq N_{\rm res}$.
    \item \verb|delta_block()| adds the finite and tail contributions, checks that the contraction bound is less than one, and applies Lemma~\ref{lem:delta-cert}. Its output is the pair $(\de_f,\de'_f)$ bounding the errors in the value and derivative of $f$.
\end{itemize}

\paragraph{Taylor models for rational functions.}
The entries of the truncated matching matrix are rational functions of $\te$. To enclose them on $[c-r,c+r]$, \verb|_entry_tm()| in \verb|cert_det_post.py| constructs a Taylor model of order one. Write an entry as $\frac{u}{v}$, where $u,v\in\Q[\te]$, and set $t=\te-c$. Exact polynomial arithmetic gives
\begin{align*}
    u(c+t)-(a_0+a_1t)v(c+t)=t^2h(t),
\end{align*}
where $a_0$ and $a_1$ are the first two Taylor coefficients of $\frac{u}{v}$. If $v(c+t)=\sum_j v_jt^j$ and $h(t)=\sum_j h_jt^j$, the program computes
\begin{align*}
    d=|v_0|-\sum_{j\geq1}|v_j|r^j.
\end{align*}
When $d>0$, the remainder is bounded by
\begin{align*}
    \left|\frac{u(c+t)}{v(c+t)}-a_0-a_1t\right|
    \leq \frac{r^2}{d}\sum_j|h_j|r^j,
    \qquad |t|\leq r.
\end{align*}
If the denominator bound is not positive, the interval must be subdivided. At a single rational parameter value, the entry is evaluated directly.

The class \verb|TM| represents the polynomial part and an interval remainder, and implements addition and multiplication with certified remainders. The exact rational coefficients are enclosed in real balls before the models are combined. The routine \verb|_tm_det()| computes a determinant by cofactor expansion in this arithmetic, and \verb|enclose()| returns an interval containing all its values. The routine \verb|eval_series_tight()| also uses shifted numerators and denominators to enclose the coefficients needed in the residual estimates.

\paragraph{Interval matrix test.}
The routine \verb|certify_maxeig_below()| tests an upper bound for the largest eigenvalue of a real symmetric matrix using interval $LDL^\top$ factorization. In the $H^1$ certificate it is applied to $-H$ with upper bound zero, where $H$ is the matrix in~\eqref{eqn:core:lmi:psd}. The test succeeds when all pivot enclosures for $H$ are strictly positive. This proves positive definiteness and hence the required nonnegative quadratic form. An undecided pivot causes the interval test to fail and the covering routine to subdivide the parameter interval.

\paragraph{Adaptive subdivision.}
The covering routines start with a list containing the entire rational parameter interval. They remove one interval, enclose the required quantity over it, and accept the interval if the desired inequality is certified. Otherwise they bisect it and return both halves to the list. The accepted intervals together with the pending intervals always cover the original domain. A successful return requires that the pending list be empty. Bounds on the subdivision depth or the number of intervals cause an error if the verification cannot be completed.

\subsection{Details of the proofs}

\begin{proofsth}{Lemma~\ref{lem:negposjump}}
Fix one of the two rational endpoints $\la\in\Slambd$. The routine \verb|build_post()| constructs
\begin{align*}
    \Mcritap=\texttt{symbolic\_matrix\_A}(N_0-1,\la),
\end{align*}
obtained by replacing every local solution in Definition~\ref{def:Mat} by its polynomial approximation. The routine \verb|all_res_sym()| constructs the corresponding residual data. Both computations are exact over $K$ and are performed once for each value of $\la$.

Let $I$ be a parameter interval. The Taylor model determinant gives an enclosure $\nu^{\rm ap}(I)$ containing $\det\Mcritap$ for all $\te\in I$. Next, \verb|delta_per_function()| encloses the residual coefficients, approximation coefficients, and partial fraction coefficients on $I$. It passes these to \verb|delta_block()| to obtain the bounds $(\de_f,\de'_f)$ from~\eqref{eq:delta-explicit}. The routine \verb|_radius_matrix()| arranges them into the entrywise radii $\varepsilon_{jk}$ of Corollary~\ref{cor:matrix-enclose}. Thus
\begin{align*}
    \Mcrit=\Mcritap+E,
    \qquad |E_{jk}|\leq\varepsilon_{jk}
\end{align*}
uniformly for $\te\in I$.

Choose certified upper bounds $S_j$ for the row sums $\sum_k|(\Mcritap)_{jk}|$, and set $\varepsilon_j=\sum_k\varepsilon_{jk}$. Multilinearity in the rows and the bound of a determinant by the product of the row $\ell^1$ norms give
\begin{align*}
    |\det\Mcrit-\det\Mcritap|
    \leq \Delta,
    \qquad
    \Delta=\prod_j(S_j+\varepsilon_j)-\prod_jS_j.
\end{align*}
Indeed, the difference of the determinants is the sum of all terms containing at least one row of $E$. Bounding each such term gives precisely the expansion on the right. The routine \verb|certified_det_post()| therefore returns an enclosure of
\begin{align*}
    \nu^{\rm ap}(I)+[-\Delta,\Delta].
\end{align*}
A strictly negative enclosure proves the required sign at $\la=\getvar{low_bnd_eig}$, and a strictly positive enclosure proves it at $\la=\getvar{upp_bnd_eig}$.

It remains to cover the parameter values in Lemma~\ref{lem:cover}. The routine \verb|prove()| first checks each prime $p<N$ at
\begin{align*}
    \te_p=\frac{\lfloor\frac p2\rfloor}{p}.
\end{align*}
This is the value returned by \verb|theta_of()| with \verb|center| equal to $\frac12$. The remaining primes are covered by the interval
\begin{align*}
    J=\left[\frac12-\frac1{2N},\frac12\right].
\end{align*}
The routine \verb|verify_det_sign()| applies adaptive bisection to $J$, accepting an interval only when the determinant enclosure has the required strict sign. It subdivides intervals for which the sign is undecided or a coefficient denominator has not been bounded away from zero. A certified opposite sign or an unresolved interval below the prescribed minimum width causes an error.

Before the covering begins, the program checks the contraction bounds for both local charts. It also checks finitely many parameter values as a diagnostic. An undecided point enclosure causes an error reporting the truncation order and arithmetic precision. These point checks accompany the subsequent certification over the whole interval.

The driver \verb|run_proof.py| performs the symbolic consistency checks and then both sign coverings with the parameters given above. The residual estimates and determinant enclosures on the accepted intervals provide the computer-assisted verification of the lemma.
\end{proofsth}

\begin{proofsth}{Lemma~\ref{lem:h1:core:cert}}
Set $r=\frac{11}{10}$. For each parameter interval, the program constructs an interval enclosure of the real symmetric matrix
\begin{align*}
    H_\te=
    \begin{pmatrix}
        M^{V_m}&M^{\rm bd}\\
        (M^{\rm bd})^\top&2(r-\ga)I_2
    \end{pmatrix}
\end{align*}
from~\eqref{eqn:core:lmi:psd}. The core has indices $-m,\ldots,m$, and the two boundary indices are $m+1$ and $-m-1$. The analytic tail bounds supply $\ga$. Thus the finite matrix test includes the interaction between the core and the infinite complement through the two boundary modes.

For part~\ref{item:h1core:i}, \verb|certify_near_0_lmi()| uses $T=L_{1,\te}^{{\rm fib}\,\Sigma}$, $m=8$, and
\begin{align*}
    \ga=\frac12\left(1+\frac1{64}\right),
\end{align*}
as supplied by Lemma~\ref{lem:tm:tail:hom}. The driver sets the core correction \verb|X| to zero, so the metric for the rescaled operator is the identity. The matrix entries are first constructed as exact rational functions of the parameter and then evaluated in ball arithmetic. The routine \verb|cover_near_0_region()| starts with $[0,\frac{13}{40}]$ and certifies its subintervals by the interval matrix test. The endpoint zero is handled by the rational expressions for the rescaled matrix. The conclusion needed here is the inequality for $0<\te\leq\frac{13}{40}$.

For part~\ref{item:h1core:ii}, \verb|certify_away_0_lmi()| uses $T=\LfibD_{1,\te}$, $m=4$, and
\begin{align*}
    \ga=\frac12\left(1+\frac3{16}\right),
\end{align*}
as supplied by Lemma~\ref{lem:tm:tail}. It evaluates the shifted-weight matrix entries directly in ball arithmetic. The routine \verb|cover_away_0_region()| applies the same interval matrix test and subdivision procedure on $[\frac{13}{40},\frac12]$.

In both cases, a subinterval is accepted only when every pivot enclosure is strictly positive. When all intervals have been accepted, the inequality holds throughout the corresponding parameter range. The driver \verb|run_h1_metrics.py| recomputes both coverings and reports success only after both have been completed. Together they verify the two assertions of the lemma.
\end{proofsth}

\paragraph{Reproducing the computations.}
From the directory containing the supplementary code, the two certificates are run with
\begin{verbatim}
sage -python run_proof.py
sage -python run_h1_metrics.py
\end{verbatim}
The first command includes the exact symbolic consistency checks and the determinant sign verification. The second verifies both parameter ranges of the $H^1$ estimate.

%%%%%%%%%%%%%%%%%%%%%%%%%%%%%%%%%
%%%%%%%%%%%%%%%%%%%%%%%%%%%%%%%%%

\subsection*{Acknowledgments}

The author acknowledges financial support from the Severo Ochoa Programme for Centers of Excellence (Grants \nolinkurl{CEX2019-000904-S} and \nolinkurl{CEX-2023-001347-S}) funded by \nolinkurl{MCIN/AEI/10.13039/501100011033}. 
Additional support was provided by Grants \nolinkurl{PID2020-114703GB-I00}, \nolinkurl{PID2024-158418NB-I00}, and \nolinkurl{RED2024-153842-T}, all funded by \nolinkurl{MICIU/AEI/10.13039/501100011033}.
The author also thanks \'Angel Castro and Javier G\'omez-Serrano for their valuable comments and discussions in the preparation of this work.

\subsection*{Generative AI statement}

During the preparation of this paper, the author used different Claude (Anthropic) and GPT (OpenAI) models to assist with copyediting and language revision throughout. Claude was used to port the author's MATLAB code to Sage for the computer-assisted verification, including the adaptations needed for interval arithmetic. In the nonlinear part, the mathematical arguments were developed jointly by the author, GPT, and Claude. The author reviewed and verified all content and takes full responsibility for it.

\setlength{\bibitemsep}{4pt}
{\setlength{\emergencystretch}{1em}\printbibliography}

\begin{flushleft}
	\bigskip
	Miguel M. G. Pascual-Caballo\\
	\textsc{Instituto de Ciencias Matem\'aticas\\
		28049 Madrid, Spain}\\
	\textit{E-mail address:} \nolinkurl{miguel.martinez@icmat.es}
\end{flushleft}

\end{document}